\documentclass[10pt, reqno]{amsart}

\usepackage{amsmath, amsthm, amscd, amsfonts, amssymb, graphicx, xcolor}
\usepackage[bookmarksnumbered, colorlinks, plainpages]{hyperref}
\usepackage[all]{xy}
\usepackage{cleveref}
\usepackage{tikz-cd}
\usetikzlibrary{arrows.meta}
\usepackage{mathrsfs}
\usepackage[left=0.8in, right=0.8in]{geometry}
\usepackage{setspace}
\newtheorem{theorem}{Theorem}[section]
\newtheorem{lem}[theorem]{Lemma}
\newtheorem*{thmA}{Theorem A}
\newtheorem*{thmB}{Theorem B}
\newtheorem{prop}[theorem]{Proposition}
\newtheorem{corollary}[theorem]{Corollary}
\theoremstyle{definition}
\newtheorem{definition}[theorem]{Definition}

\theoremstyle{remark}
\newtheorem{remark}[theorem]{Remark}
\numberwithin{equation}{section}
\crefrangeformat{equation}{(#1)--(#2)}

\begin{document}
\setcounter{page}{1}

\newcommand{\calB}{\mathcal{B}}
\newcommand{\calG}{\mathcal{G}}
 \newcommand{\calP}{\mathcal{P}}
\newcommand{\calM}{\mathcal{M}}
\newcommand{\calN}{\mathcal{N}}
\newcommand{\calW}{\mathcal{W}}
\newcommand{\calD}{\mathcal{D}}
\newcommand{\calU}{\mathcal{U}}
\newcommand{\calS}{\mathcal{S}}
\newcommand{\T}{\mathcal{T}}

\newcommand{\id}{\mathrm{Id}}
\newcommand{\Span}{\mathrm{Span}}
\newcommand{\Nontip}{\mathrm{Nontip}}
\newcommand{\Tip}{\mathrm{Tip}}
\newcommand{\CTip}{\mathrm{CTip}}
\newcommand{\rmBar}{\mathrm{Bar}}
\newcommand{\Hom}{\mathrm{Hom}}
\newcommand{\Ker}{\mathrm{Ker}}
\newcommand{\rmIm}{\mathrm{Im}}

\newcommand{\rmChar}{\mathrm{Char}}

\newcommand{\rmH}{\mathrm{H}}
\newcommand{\rmHH}{\mathrm{HH}}

%\color{darkgray}{
%\noindent 
%{\small Annals of Mathematics and Computer Science}\hfill     {\small ISSN: 2789-7206}\\
%{\small Vol 20 (2024) 1-3}\hfill  {\small https://doi.org/10.56947/amcs.v20.223}}

%\centerline{}

%\centerline{}

%------------------------------------------------------------------------------

%Title of the paper
\title[Hochschild cohomology ring of the Xu--Snashall   algebra]{Gerstenhaber algebra structure on the Hochschild cohomology ring of the Xu--Snashall   algebra}

%Author names and affiliations
\author{Qi Long,  Ziyang Shi and Guodong Zhou}

\address{ 
	  School of Mathematical Sciences
     Key Laboratory of MEA (Ministry of Education),
	   Shanghai Key Laboratory of PMMP,
	  East China Normal University
	  Shanghai 200241,
	   P. R. China} 
\email{51275500029@stu.ecnu.edu.cn }
\email{51265500052@stu.ecnu.edu.cn }
 \email{gdzhou@math.ecnu.edu.cn}

%\dedicatory{This paper is dedicated to Professor ABCD}

%\date{Received: xxxxxx; Revised: yyyyyy; Accepted: zzzzzz.
%\newline \indent \(^{*}\) Corresponding author
%\newline \indent © The Author(s) 2025. This article is licensed under a Creative Commons Attribution-
%\newline \indent NonCommercial-NoDerivatives 4.0
%International License. To view a copy of the licence, visit 
%\newline \indent \url{https://creativecommons.org/licenses/by-nc-nd/4.0/}}

%Abstract, keywords, math subject classification
\begin{abstract}  Let $A$ be a finite dimensional algebra and let $\rmHH^*(A)$ be its Hochschild cohomology ring, which is a Gerstenhaber algebra. 
Denote by $\calN$ (resp. $G$, $\calG$) the ideal (resp. weak Gerstenhaber ideal, Gerstenhaber ideal) generated by all homogeneous nilpotent elements. 

Motivated by their work on support varieties via Hochschild cohomology, 
Snashall and Solberg conjectured that $\rmHH^*(A)/\calN$ is a finitely generated algebra. Xu constructed a counterexample to the Snashall-Solberg conjecture over a base field of characteristic two,  and Snashall generalized this example  to arbitrary characteristic. 

Hermann further asked whether $\rmHH^*(A)/G$  is a finitely generated algebra and suggested  considering first the Xu--Snashall  algebra. 
In this paper, we answer this question for the Xu--Snashall  algebra. In fact, by explicitly computing the Gerstenhaber algebra structure on the Hochschild cohomology ring, we show that  $G=\calN$; hence $\rmHH^*(A)/G=\rmHH^*(A)/\calN$ is not a finitely generated algebra. Furthermore, we show that   $\rmHH^*(A)/\calG\cong K$. Therefore, one may  ask whether, for a finite dimensional algebra $A$, $\rmHH^*(A)/\calG$  is always a finitely generated algebra.

Our main tools are  two-sided Anick resolutions and weak self-homotopies. 
%16G20%(1991–now) Representations of quivers and partially ordered sets
\end{abstract} \maketitle

\tableofcontents

%------------------------------------------------------------------------------

\section{Introduction}

It is well known from the work of Gerstenhaber \cite{Ger63} that the  Hochschild cohomology  of a ring has a rich structure: it has a cup product which makes it a graded commutative algebra and a Gerstenhaber  Lie bracket which makes it a (shifted) graded Lie algebra, and the cup product and the Lie bracket are compatible in a  certain sense. Such a structure is later called a Gerstenhaber algebra.

Let $A$ be a finite dimensional algebra and $\rmHH^*(A)$ be its Hochschild cohomology ring. 
Denote by $\calN$ the ideal generated by  all homogeneous nilpotent elements (with respect to the cup product), so  the quotient $\rmHH^*(A)/\calN$ is still a graded commutative algebra;
denote by  $G$ the weak Gerstenhaber ideal  generated by all  homogeneous nilpotent elements, that is, the minimal ideal (with respect to the cup product) containing all homogeneous nilpotent elements and closed under the Lie bracket, so  the quotient $\rmHH^*(A)/G$ is a graded commutative algebra, but not a Gerstenhaber algebra in general; 
denote by $\calG$ the Gerstenhaber ideal generated by all homogeneous nilpotent elements, so the quotient $\rmHH^*(A)/\calG$ is a Gerstenhaber algebra.

Motivated by their work on support varieties via Hochschild cohomology \cite{SS04, EHTSS04}, 
Snashall and Solberg \cite{SS04} conjectured that the Hochschild cohomology ring of a finite dimensional algebra modulo the ideal generated by all homogeneous nilpotent elements, say $\rmHH^*(A)/\calN$, is a finitely generated algebra. Xu \cite{Xu08} constructed a counterexample over a base field of characteristic two to the Snashall--Solberg conjecture, and later Snashall \cite{Sna09} generalized this example to arbitrary characteristic.
 
Let us recall the quiver with relations of the Xu--Snashall algebra. Let $Q$ be the quiver
\[
\begin{tikzcd}
    1 \arrow[loop above, out =120, in =60, looseness=4,"a"] \arrow[loop below,out =300, in =240, looseness=5,"b"] \arrow[r,"c"] & 2
\end{tikzcd}
\] 
and let $I$ be the ideal of the path algebra $KQ$ generated by the relations $\{a^2, b^2, ab-ba, ac\}$. The Xu--Snashall algebra is defined to be $A=KQ/I$.  
\begin{thmA} [\cite{Xu08, Sna09}] 
When $\rmChar(K)= 2$,
\[\rmHH^*(A)/\calN \cong  K \oplus K[a,b]b,
    \]
where $|a|=|b| = 1$; and when $\rmChar(K)\neq 2$, 
\[ \rmHH^*(A)/\calN \cong  K \oplus K[a^2,b^2]b^2,\]
where $|a|=|b| = 1$. 
\end{thmA}

These are well-known examples of non-Noetherian rings. Further counterexamples to the Snashall-Solberg conjecture have been obtained in \cite{HF12, XZ12}.

Hermann \cite{Her16} asked further whether the quotient of the Hochschild cohomology ring by the weak Gerstenhaber ideal generated by all homogeneous nilpotent elements, say $\rmHH^*(A)/G$, is a finitely generated algebra. He suggested considering first the  Xu--Snashall algebra.

The goal of this article is to answer Hermann's question   for the Xu--Snashall algebra. More precisely, we show the followng result:

\begin{thmB}[{see Theorem \ref{Main Theorem}}]\label{ThmB}
    For the Xu--Snashall algebra, $G=\calN$; hence, $\rmHH^*(A)/G=\rmHH^*(A)/\calN$ is not a finitely generated algebra.   Furthermore, we show that  $\rmHH^*(A)/\calG\cong K$. 
\end{thmB}
The first statement of the above theorem  answers Hermann's question in the negative. Motivated by the second statement,  
one may therefore ask whether, for a finite dimensional algebra $A$,  $\rmHH^*(A)/\calG$  is always a finitely generated algebra. 

Recall that Snashall \cite{Sna09} computed $\rmHH^*(A)/\calN$ by using the fact that it is isomorphic to the graded center of the Koszul dual algebra modulo the ideal generated by nilpotent elements.  Oke \cite{Oke22} also considered the Xu--Snashall algebra and showed the same  result $G=\calN$ as ours by using the homotopy lifting method of Volkov \cite{Vol19}. However, he did not determine the Gerstenhaber algebra structure on $\rmHH^*(A)$. 

 In this paper, we give the  Gerstenhaber Lie bracket on $\rmHH^*(A)$ in terms of generators and relations (see Theorems~\ref{thm:  char two Gerstenhaber bracket} and \ref{thm:  char not two Gerstenhaber bracket}), and hence  Theorem B follows trivially. 
 Our strategy goes as follows. 
 We construct the two-sided Anick  resolution by using the algebraic Morse theory introduced in \cite{Koz05, Sko06, JW09} and further developed in \cite{Sko18, CLZ24}. To compute the cup product and the  Lie bracket, we construct comparison morphisms between the reduced bar resolution and the two-sided Anick    resolution; besides algebraic Morse theory, we also use the method of weak self-homotopies (see \cite[Chapter IX Theorem 6.2]{ML67}, \cite{BZZ09} and \cite[Section 2]{IIVZ15}).

This article is organized as follows. In Section \ref{sec: Preliminaries}, we recall noncommutative Gr\"obner-Shirshov systems, the Gerstenhaber structure on Hochschild cohomology, two-sided Anick resolutions and the method for constructing comparison morphisms between two projective resolutions via weak self-homotopies. In Section \ref{sec: Anick resolution}, we explicitly compute the two-sided Anick resolution $B_*^{\calM}$ of $A$. In Section \ref{sec: Hochschild cohomology groups}, we compute the Hochschild cohomology groups of $A$ and their dimensions as $K$-vector spaces. In Section \ref{sec: Comparison morphisms}, we construct the comparison morphisms between the reduced bar resolution $B_*$ of $A$ and the resolution $B_*^{\calM}$. In Section \ref{sec: Cup Product}, we describe a product on $B_*^{\calM}$, also called the cup product, which induces a well-defined product on $\rmHH^*(A)$. In Section \ref{sec: The Hochschild cohomology ring}, we show that $\rmHH^*
(A)$ is isomorphic to a polynomial ring in countably many variables modulo some relations, which allows us to recover the results obtained in \cite{Xu08, Sna09}. In Section \ref{Section: Gerstenhaber Lie Bracket}, we describe a bracket on $B_*^{\calM}$, also called
the Lie bracket, which induces a well-defined bracket on $\rmHH^*(A)$. We also obtain $G=\mathcal{N}$ and hence $\rmHH^*(A)/G=\rmHH^*(A)/\mathcal{N}$, showing that the algebra $A$ serves as a counterexample to Hermann's question. Furthermore, we obtain that $\rmHH^*(A)/\calG\cong K$.

Throughout this paper, this paper $K$ denotes a base field. 

\section{Preliminaries}\label{sec: Preliminaries}
\subsection{Noncommutative Gr\"obner-Shirshov systems}\mbox{}

Let \(\Lambda\) be an associative algebra  with a \(K\)-basis \(\calB\). 
We call \(\calB\) a multiplicative basis of \(\Lambda\) if  for any $b,  b'\in \calB$, $bb'\in \calB$ or $bb'=0$. The elements of \(\calB\) are called monomials.

A well-order  on \(\calB\) is a total order \(<\) such that there are no infinite descending chains in \(\calB\).
For a nonzero element \(v = \sum_{i=1}^{n} \alpha_i b_i \in \Lambda\) with $\alpha_i \in K^*=K\setminus \{0\}$, $b_i \in \calB$, denote by 
\(\Tip(v)\) the maximal \(b_i\) under \(<\), and by \(\CTip(v)\) its coefficient \(\alpha_i\).
For a nonempty subset \(X\) of \(\Lambda\), set
\[\Tip(X)=\{b \in \calB \mid b=\Tip(v)\text{ for some }v\in X\} \quad\text{and}\quad \Nontip(X)=\calB\setminus \Tip(X).\]

Let $I$ be a two-sided ideal of $\Lambda$. By \cite[Theorem 2.1]{Gre99}, there is a decomposition of vector spaces
    \[\Lambda = I \oplus \Span_K(\Nontip(I)).\]
Hence, the set \(\Nontip(I)\) is a basis of the quotient algebra \(\Lambda/I\).
%for any nonzero \(v\in \Lambda\), it can be written uniquely as \(v= i_v +N(v)\) with \(i_v \in I\) and \(N(v) \in \Span_K(\Nontip(I))\). 
%\(N(v)\) is called the normal form of \(v\).  

A well-order \(<\) on \(\calB\) is admissible if  it satisfies the following conditions: 
\begin{enumerate}
\item If \(pr \neq 0\) and \(qr \neq 0\), then \(p < q\) implies \(pr < qr\),
\item if \(sp \neq 0\) and \(sq \neq 0\), then \(p < q\) implies \(sp < sq\),
\item if \(p = qr\), then \(p \ge q\) and \(p \ge r\),
\end{enumerate}
for all \(p, q, r, s \in \calB\). 
For instance, the left length-lexicographic order is an example of an admissible order.

\begin{definition} [{\cite[Definition 2.4]{Gre99}}]
     Let \(\Lambda\) be a \(K\)-algebra with multiplicative basis \(\calB\) and admissible order \(<\). For an ideal \( I\subset \Lambda\), a subset \( \mathscr{G} \subset I \) is a Gr\"obner-Shirshov system (or generating set) for \( I \) with respect to \( < \) if  \(\langle \Tip(\mathscr{G}) \rangle = \langle \Tip(I) \rangle \).
\end{definition}

For \(p,q \in \calB \), \(p\) divides \(q\) if there are \(a,b \in \calB\) such that \(q=apb\). Let \(I\) be an ideal of \(\Lambda\) generated by a  set \(X\) such that, for any \(v\in \Lambda\), the set \(\{x\in X \mid \Tip(x) \le \Tip(v)\}\) is finite.  
For a fixed nonzero \(v\in \Lambda\), enumerate this finite set as \(x_1, \dots, x_n\). Following \cite[Section 2.3.2]{Gre99}, we say that \(v\) reduces to \(r\) modulo \(X\), written \(v \Rightarrow_X r\), if
there exist \( m_1, \ldots, m_n \in \mathbb{N}\) and elements \( u_{i,j}, w_{i,j}, r \in \Lambda \) for \( 1\le i \le n\) and \( 1\le j \le m_i \) such that \begin{enumerate}
    \item \( v = \sum_{i=1}^n \sum_{j=1}^{m_i} u_{i,j} x_i w_{i,j} + r\);
    \item \(\Tip(v) \ge \Tip(u_{i,j} x_i w_{i,j}) \text{ for all } i, j\);
    \item for \( b \in \calB \) occurring in \( r \), \( \Tip(x) \) does not divide \( b \), for all \( x \in X \).
\end{enumerate}
In this case, \(r\) is called the remainder of \(v\) upon division by \(X\).

Let \( f, g \in \Lambda \). Suppose that there are \( a, b \in \calB \) such that $\Tip(f)a = b\Tip(g)$, $\Tip(f) \nmid b$ and $\Tip(g) \nmid a$.
Then the overlap relation of \( f \) and \( g \) by \( a,b \) is defined as
\[
o(f, g, a, b) = \frac{1}{\CTip(f)} f a - \frac{1}{\CTip(g)} b g.
\]

Recall that a  subset \(X\subset \Lambda\) is called tip reduced if for any distinct \(x,y \in X\), \( \Tip(x) \nmid \Tip(y)\).
Let \(v = \sum_{i=1}^{n} \alpha_i b_i \in \Lambda\) be a nonzero element, where \(\alpha_i \in K^*\) and \(b_i \in \calB\). The element \(v\) is said to be  uniform if for each \(c \in \calB\), either \(cb_i = 0\) for all \(i\)  or \(cb_i \neq 0\) for all \(i\); and for each \(d \in \calB\), either \(b_id = 0\) for all \(i\) or \(b_id \neq 0\) for all \(i\).

Green gave the following termination theorem \cite[Theorem 2.3]{Gre99}, which provides a criterion for a generating set of an ideal to be a Gr\"obner-Shirshov system.  This is a several vertices version of the classical Diamond-Composition Lemma. 
\begin{theorem} \label{thm: GS basis}
    Let \( \Lambda \) be a \( K \)-algebra with multiplicative basis \( \calB \) and admissible order \( < \). Suppose that \( \mathscr{G}\) is a set of uniform, tip reduced elements in \( \Lambda \). If for every overlap relation
\[
o(g_1, g_2, p, q) \Rightarrow_\mathscr{G} 0,
\]
with \( g_1, g_2 \in \mathscr{G} \) and \(p,q \in \calB \), then \( \mathscr{G} \) is a Gr\"obner-Shirshov system for \( \langle \mathscr{G} \rangle \).
\end{theorem}

\begin{definition}\label{regs}
    A Gr\"obner-Shirshov system \(\mathscr{G}\) for an ideal  \(I\) is called reduced if for all distinct \(g,h \in \mathscr{G}\), \( \Tip(g) \nmid \Tip(h)\), and for every $g \in \mathscr{G}$, $\CTip(g) = 1$ and $g - \Tip(g) \in \Span_K(\Nontip(I))$.
\end{definition}

Note that a reduced Gr\"obner-Shirshov system  is unique.

%---------------------------------------------------------------------------------------%

\subsection{Hochschild cohomology and Gerstenhaber algebras}\mbox{}

In this section, we recall the definition of a Gerstenhaber algebra. 
%and then describe how Hochschild cohomology naturally carries a Gerstenhaber algebra structure.
\begin{definition}
    A Gerstenhaber algebra   is a graded \(K\)-vector space \(\rmH = \bigoplus_{n\in\mathbb{Z}} \rmH^n\) equipped with two bilinear maps: a cup product
\[\cup: \rmH^n \times \rmH^m \to \rmH^{n+m}, \quad (a, b) \mapsto a \cup b\]
and a Lie bracket of degree \(-1\)
\[[-,-]: \rmH^n \times \rmH^m \to \rmH^{n+m-1}, \quad (a, b) \mapsto [a, b]\]
such that for all homogeneous elements \(a, b, c \in \rmH\) (with \(|a|\) denoting the degree of \(a\)):
\begin{enumerate}
    \item \((\rmH, \cup)\) is a graded commutative algebra, that is, \(\cup\) is an associative multiplication satisfying \(a \cup b = (-1)^{|a| \cdot |b|} b \cup a\);
    \item \((\rmH, [-,-])\) is a graded Lie algebra of degree \(-1\), that is, \( [-,-]\) satisfies
\[[a, b] = -(-1)^{(|a|-1)(|b|-1)}[b, a]\]
and the graded Jacobi identity
\[(-1)^{(|a|-1)(|c|-1)}[[a, b], c] + (-1)^{(|b|-1)(|a|-1)}[[b, c], a] + (-1)^{(|c|-1)(|b|-1)}[[c, a], b] = 0;\]
    \item the cup product and the Lie bracket satisfy the compatibility condition called the Poisson rule, that is,
\[[a, b\cup c] = [a, b]\cup  c + (-1)^{(|a|-1)|b|}b\cup [a, c].\]
\end{enumerate}
\end{definition}
Hochschild \cite{Hoc45} introduced cohomology of associative algebras, and later Gerstenhaber \cite{Ger63} showed that the Hochschild cohomology ring admits a Gerstenhaber algebra structure.  We briefly recall the definition of Hochschild cohomology and describe how this Gerstenhaber algebra structure arises.

 Let \( A \) be an associative \(K\)-algebra with unit \( 1 \). Denote its enveloping algebra by \( A^e = A \otimes_K A^{\text{op}} \). There is a canonical projective resolution $\rmBar_*(A)$  of the \(A^e\)-module \(A\), called the (unnormalized) bar resolution : 
\[\rmBar_*(A): \cdots \to A \otimes_K A^{\otimes_K n} \otimes_K A \xrightarrow{b_n} A \otimes_K A^{\otimes_K (n-1)} \otimes_K A \to \cdots \to A \otimes_K A \otimes_K A \xrightarrow{b_1} A\otimes_K A (\xrightarrow{\mu} A\to 0),\]
where   \(\rmBar_n(A)=A^{\otimes_K {(n+2)}},  n\geq 0\), \(\mu: A \otimes_K A \to A\) is the multiplication of \(A\), and for $n\geq 1$, 
\[
b_n(a_0 \otimes a_1 \otimes \cdots \otimes a_n \otimes a_{n+1}) = \sum_{i=0}^n (-1)^i a_0 \otimes \cdots \otimes a_{i-1} \otimes a_i a_{i+1} \otimes a_{i+2} \otimes \cdots \otimes a_{n+1}.
\]

Applying the functor \(\Hom_{A^e}(-,A)\) to the bar resolution $\rmBar_*(A)$ gives the Hochschild cochain complex \(C^*(A) = \Hom_{A^e}(\rmBar_*(A), A)\). For each \(n \ge 0\), there is a \(K\)-space isomorphism:
\[
C^n(A) =\Hom_{A^e}(\rmBar_n(A), A) \xrightarrow{\sim} \Hom_K(A^{\otimes_K n}, A), \quad f \mapsto \bigl( a_1 \otimes \cdots \otimes a_n \mapsto f(1 \otimes a_1 \otimes \cdots \otimes a_n \otimes 1) \bigr),
\]
where by convention \(A^{\otimes_K 0}=K\). Thus we obtain the complex, still denoted by $C^*(A)$, 
\[
C^*(A): A=\Hom_{{K}}({K}, A) \xrightarrow{\delta^0} \Hom_{K}(A, A) \to \cdots \to \Hom_{{K}}(A^{\otimes_K n}, A) \xrightarrow{\delta^{n}} \Hom_{K}(A^{\otimes_K {(n+1)}}, A) \to \cdots,
\]
where for \(n \ge 0\),
%a linear map \(f: A^{\otimes_K {n-1}} \to A\), and \(a_1, \dots, a_{n} \in A\),
\[
\begin{aligned}
\delta^n(f)(a_1 \otimes \cdots \otimes a_{n+1})& =-(-1)^n f\circ b_{n+1}(a_1 \otimes \cdots \otimes a_{n+1}) \\
&=(-1)^{n+1} a_1 f(a_2 \otimes \cdots \otimes a_{n+1}) \\
&\quad + \sum_{i=1}^{n} (-1)^{n+1-i} f(a_1 \otimes \cdots \otimes a_{i-1} \otimes a_i a_{i+1} \otimes a_{i+2} \otimes \cdots \otimes a_{n+1}) \\
&\quad +  f(a_1 \otimes \cdots \otimes a_{n}) a_{n+1}.
\end{aligned}
\]
\begin{definition}
    The Hochschild cohomology \(\rmHH^*(A)\) of $A$ is the cohomology of the complex \(C^*(A)\).
\end{definition}
Since \(\rmBar_*(A)\) is a projective \(A^e\)-resolution of \(A\), we have, for $n\ge 0$,
\[
\rmHH^n(A) \cong \operatorname{Ext}_{A^e}^{n}(A, A).
\]
 
  For \(f \in C^m(A)\) and \(g \in C^n(A)\), the cup product \(f \cup g \in C^{m+n}(A)= \Hom_K(A^{\otimes_K (m+n)}, A) \) is given by
\[f\cup g(a_1\otimes \cdots\otimes a_{m+n})=(-1)^{mn} f(a_1\otimes \cdots\otimes a_m) g(a_{m+1}\otimes \cdots\otimes a_{m+n}).\]
When \(m=0\) or \(n=0\), \(f\) or \(g\) is evaluated at \(1_K \in K\).  Moreover, the Hochschild differential $\delta^*$ is a derivation with respect to the cup product. Hence the cup product induces a well-defined product on cohomology,
\[
\cup: \rmHH^m(A) \times \rmHH^n(A) \to \rmHH^{m+n}(A).
\]
By \cite[Section 7, Corollary 1]{Ger63}, $\rmHH^*(A)$  with this cup product is a graded commutative algebra.

The Gerstenhaber Lie bracket is defined as follows. Let \(f \in C^m(A)\) and \(g \in C^n(A)\). If \(m,n\ge 1\), then for \(1 \le i \le m\), \(f \circ_i g \in C^{m+n-1}(A) \) is given by
\[ f \circ_i g(a_1\otimes\cdots\otimes a_{m+n-1}) =(-1)^{(i-1)(n-1)} f(a_1\otimes\cdots\otimes a_{i-1}\otimes g(a_i\otimes \cdots\otimes a_{i+n-1})\otimes a_{i+n}\otimes \cdots\otimes a_{m+n-1}).\]
If \(m\ge 1\) and \(n=0\), then $g=g(1_K)\in A$ and  for \(1 \le i \le m\), \(f \circ_i g \in C^{m-1}(A) \) is given by
\[ f \circ_i g(a_1\otimes\cdots\otimes a_{m-1}) =(-1)^{i-1} f(a_1\otimes\cdots\otimes a_{i-1}\otimes g(1_K)\otimes a_i\otimes \cdots\otimes a_{m-1}).\]
If \(m=0\), then define \(f \circ_i g =0 \). Define
\[f \overline{\circ} g = \sum_{i=1}^{m}  f \circ_i g, \quad  [f, g] = f \overline{\circ} g - (-1)^{(m-1)(n-1)} g \overline{\circ} f. \]
% By \cite[page 279]{Ger63}, \((C^*(A), [-,-])\) is a graded Lie algebra. Moreover, the differential \(b\) is a left derivation of \((C^*(A), [-,-])\)(see \cite[page 280]{Ger63}):
% \[b([f, g]) = [b(f), g] + (-1)^{|f|-1}[f, b(g)], \forall f, g \in C^*(A),\]
% This identity ensures that the Lie bracket is well-defined on the Hochschild cohomology groups, inducing a Lie bracket 
% \[[-,-]: \rmHH^m(A) \times \rmHH^n(A) \to \rmHH^{m+n-1}(A)\]
% By \cite{Ger63}, this induced Lie bracket makes \(\rmHH^*(A)\) into a graded Lie algebra of degree -1. 
% And the Lie bracket and the cup product are compatible \cite[Section 8]{Ger63}. All in all, the Hochschild cohomology \(\rmHH^*(A)\) is a Gerstenhaber algebra.
%By \cite[p. 279]{Ger63}, \((C^*(A), [-,-])\) is a graded Lie algebra. 
Moreover, the Hochschild differential \(\delta^*\) is a  derivation with respect to  this Lie bracket (\cite[Page 280]{Ger63}):
\[
\delta^*([f, g]) = [\delta^*(f), g] + (-1)^{|f|-1} [f, \delta^*(g)].
\]
Hence, the Lie bracket induces a well-defined bracket on cohomology
\[
[-,-]: \rmHH^m(A) \times \rmHH^n(A) \to \rmHH^{m+n-1}(A).
\]
By \cite[Theorem 8]{Ger63}, \(\rmHH^*(A)\) with this induced bracket is a graded Lie algebra of degree \(-1\). 

Furthermore, Gerstenhaber showed that the Lie bracket and the cup product satisfy the Poisson rule in \cite[Section 8]{Ger63}. Therefore, \(\rmHH^*(A)\), equipped with these two operations, is a Gerstenhaber algebra.

%---------------------------------------------------------------------------------------%
\subsection{Two-sided Anick resolutions}\mbox{}

Let \(Q=(Q_0, Q_1, s, t)\) be a finite quiver with vertex set \(Q_0\), arrow set \(Q_1\), and with two maps \(s,t: Q_1 \rightarrow Q_0\) assigning each arrow its source \(s(\alpha)\) and  target \(t(\alpha)\) to  each arrow \( \alpha \in Q_1\), respectively. A path in \(Q\) is a sequence of arrows \(\alpha_1 \alpha_2 \cdots \alpha_r\) with \(r\ge 1\) such that \(t(\alpha _i)= s(\alpha_{i+1}) \) for \(1 \le i \le r-1\), whose  length is defined to be \(r\). The vertices are considered as paths of length $0$. For \(n \in \mathbb{N}\), denote by \(Q_n\) the set of paths of length \(n\) and by \(Q_{\ge n}\) the set of paths of length at least \(n\). Denote by \(\calB= Q_{\ge 0}\),  the set of all finite paths in \(Q\). 
The path algebra \(KQ\) of \(Q\) is the \(K\)-vector space with  basis \(\calB\), whose multiplication is given by concatenation of paths. Moreover, $\calB$ is a multiplicative basis of  \(KQ\).

Let \(I\subseteq\Span_K(Q_{\ge2})\) be a two-sided ideal of \(KQ\), and let \(\mathscr{G}\) be a reduced Gr\"obner-Shirshov system for \(I\) with respect to an admissible order \(<\). Set \(\calW=\Tip(\mathscr{G})\) and \(\calB_+ =\calB\setminus Q_0\). A monomial \(u\in \calB_+\) is a proper right  (left) factor of \(v \in \calW\) if \(v=bu\) (\(v=ub\)) for some \(b \in \calB_+\).

Following \cite[Section 4]{CLZ24}, we recall the Ufnarovski\u{\i} graph  \(Q_{\calW} = (V, E)\) with respect to \(\calW\), originally introduced by Ufnarovski\u{\i} \cite{Ufn89}. Its vertex set is
\[V = Q_0 \cup Q_1 \cup \{ u \in \calB \mid u \text{ is a proper right factor of some } v \in \calW \},\]
and its arrow set is
\[\begin{aligned}
    E &= \{ e \to x \mid e \in Q_0, x  \in Q_1\  \text{with}\  e=s(x)\} \cup \\
& \quad \quad\quad\quad\{ u \to v  \mid uv \in \calB \cap \langle \mathrm{Tip}(I) \rangle, w \notin \langle \mathrm{Tip}(I) \rangle \text{ for all proper left factors } w \text{ of } uv \}.
\end{aligned}
\]

For \(i \ge 0\), let \(\calW^{(i)}\) be the set of Anick \(i\)-chains, that is, all sequences \((v_1, v_2, \dots, v_i, v_{i+1}) \in \calB_+^{i+1} \)  such that 
\[
s(v_1) \to v_1 \to v_2 \to \cdots \to v_i \to v_{i+1}
\]
forms a path in \(Q_{\calW}\). By definition, \(\calW^{(-1)} = Q_0\), $\calW^{(0)}=Q_1$ and  for $(v_1, v_2)\in \calW^{(1)}$, $v_1$ is an arrow and $v_1v_2\in \calW$.

Let \(A = KQ/I\) with identity $1_A=\sum_{e\in Q_0}e$ and let \(E = \bigoplus_{e \in Q_0} Ke \subset A\). 
Then \(E\) is a semisimple subalgebra of \(A\), and \(A = E \oplus A_+\) as vector spaces, where \(A_+ = \Span_K(\Nontip(I)\setminus Q_0)\).
Following \cite[Lemma 2.1]{Cib90}, the reduced two-sided bar resolution \(B_*=(B_*, d)\) of \(A\) is defined by
\[B_0 = A \otimes_E A,\quad B_n = A \otimes_E (A_+)^{\otimes_E n} \otimes_E A \cong A^e \otimes_{E^e} (A_+)^{\otimes_E n},\; n \geq 1,\]
 where \(E^e= E \otimes_K E^{op}\). For  $n\ge 1$, the differential 
 $$d_n: B_n=A^e \otimes_{E^e} (A_+)^{\otimes_E n} \to B_{n-1}=A^e \otimes_{E^e} (A_+)^{\otimes_E (n-1)}$$
 is given by
 \begin{equation}
\begin{split}
d_n((1_A\otimes_K 1_A)\otimes_{E^e}(a_1\otimes_E \cdots \otimes_E a_n) ) =&  (a_1 \otimes_K 1_A)\otimes_{E^e} (a_2\otimes_E \cdots \otimes_E a_n) \\
&+ \sum_{i=1}^{n-1} (-1)^i (1_A\otimes_K 1_A)\otimes_{E^e} (a_1\otimes_E \cdots \otimes_E a_i a_{i+1}\otimes_E \cdots \otimes_E a_n) \\
&+ (-1)^n (1_A \otimes_K a_n)\otimes_{E^e} (a_1\otimes_E \cdots \otimes_E a_{n-1}) 
\end{split}
\end{equation}
The augmentation  \( d_0 : B_0=A \otimes_E A \to A \) is given by \( d_0(a_0\otimes_E a_1)=a_0a_1\).

As an \(A^e\)-module, \(B_n\) admits a direct sum decomposition:
\[ B_n = \bigoplus_{(w_1, w_2, \cdots, w_n)\in \mathcal{V}_n } A^e \otimes_{E^e} K( w_1 \otimes_E w_2 \otimes_E \cdots \otimes_E w_n) = \bigoplus_{(w_1, w_2, \cdots, w_n)\in \mathcal{V}_n } A^e \cdot (1 \otimes_E w_1 \otimes_E w_2 \otimes_E \cdots \otimes_E w_n \otimes_E 1), \]
where for $n\ge 1$,
\[\mathcal{V}_n = \{ (w_1, w_2, \cdots, w_n) \mid 
w_1, \cdots, w_n \in \Nontip(I) \setminus Q_0,  w_1 w_2 \cdots w_n \text{ is a path in } Q \},\] 
and \(\mathcal{V}_0=Q_0\).
%the direct sum runs over all sequences \((w_1, w_2, \cdots, w_n)\) satisfying \(w_i \in \Nontip(I) \setminus Q_0\) for all \(1 \leq i \leq n\) and such that \(w_1w_2\cdots w_n\) is a path in \(Q\). 
For each such sequence \((w_1, w_2, \cdots, w_n)\), denote \(s(w_1w_2\cdots w_n) = s(w_1)= e_i\) and \(t(w_1w_2\cdots w_n) = t(w_n) = e_j\). Then each direct summand \( A^e \cdot (1 \otimes_E w_1 \otimes_E w_2 \otimes_E \cdots \otimes_E w_n \otimes_E 1)\) can be identified with the indecomposable projective \(A^e\)-module \(Ae_i \otimes_K e_jA\). 

Recall the construction of  the weighted quiver \(\overline{Q}_B\) for $B_*$. The vertex set of \(\overline{Q}_B\) is 
$$\{1 \otimes_E w_1 \otimes_E w_2 \otimes_E \cdots \otimes_E w_n \otimes_E 1\mid (w_1, w_2, \cdots, w_n)\in \mathcal{V}_n, n\in \mathbb{N}\}.$$
 For simplicity, we identify each vertex $1 \otimes_E w_1 \otimes_E w_2 \otimes_E \cdots \otimes_E w_n \otimes_E 1$ with the sequence \((w_1, w_2, \cdots, w_n)\).
 The arrows of \(\overline{Q}_B\) starting from a vertex \((w_1, w_2, \cdots, w_n)\)  are as follows: 
    \begin{enumerate}
        \item There is an arrow 
\[\begin{tikzpicture}[>=Stealth, baseline=(current bounding box.center)]
    \node (top) at (0, 0.4) {$(w_1, w_2, \cdots, w_n)$};
    \node (bot) at (5.5, -0.5) {$(w_2, \cdots, w_n)$};
    \draw[->] (top.south east) -- (bot.north west) 
        node[midway, above, font=\small, yshift=0pt] {$d_n^0$}
        node[midway, below, font=\small, yshift=-0pt] {$w_1 \otimes 1$};
\end{tikzpicture}\]
with weight $w_1 \otimes 1$;

         \item for $1 \le i \le n-1$, assume that $w_i w_{i+1} \equiv \sum_j \lambda_j u_j \pmod{I}$ with all $u_j \in \Nontip(I)\setminus Q_0$ and $\lambda_j \in K^*$, and for each \(j\), we have an arrow 
\[\begin{tikzpicture}[>=Stealth, baseline=(current bounding box.center)]
    \node (top) at (0, 0.4) {$(w_1, \cdots, w_{i-1}, w_i, w_{i+1}, \cdots, w_n)$};
    \node (bot) at (7, -0.5) {$(w_1, \cdots, w_{i-1}, u_j, w_{i+2}, \cdots, w_n)$};
    \draw[->] (top.south east) -- (bot.north west) 
        node[midway, above, font=\small, yshift=0pt] {$d_n^i$}
        node[midway, below, font=\small, yshift=-0pt] {$(-1)^i\lambda_j$};
\end{tikzpicture}\]
which  is also denoted by \(d_n^i\) and has weight \((-1)^i\lambda_j\);

        \item there is an arrow
\[\begin{tikzpicture}[>=Stealth, baseline=(current bounding box.center)]
    \node (top) at (0, 0.4) {$(w_1, \cdots, w_{n-1}, w_n)$};
    \node (bot) at (5.5, -0.5) {$(w_1, \cdots, w_{n-1})$};
    \draw[->] (top.south east) -- (bot.north west) 
        node[midway, above, font=\small, yshift=0pt] {$d_n^n$}
        node[midway, below, font=\small, yshift=-0pt] {$(-1)^n 1 \otimes w_n$};
\end{tikzpicture}\]
with weight \((-1)^n 1 \otimes w_n\).
    \end{enumerate}

For \(w \in \calB\), let \(V_{w,i}\) be the set of all vertices \((w_1, \dots, w_n)\) in \(\overline{Q}_B\) such that \(w = w_1 \cdots w_n\) and \(i\) is the largest integer (\(\ge-1\)) satisfying \((w_1, \cdots, w_{i+1})\) is an Anick \(i\)-chain. 

Let \(\calM\) be the set of arrows in \(\overline{Q}_B\) of the form
\begin{equation} \label{arrow in M}
    (w_1, \cdots, w_{i-1}, w_i, w_{i+1}, w_{i+2}, \cdots, w_n) \to (w_1, \cdots, w_{i-1}, w_i w_{i+1}, w_{i+2}, \cdots, w_n)
\end{equation}
where \((w_1, \cdots, w_{i-1}, w_i, w_{i+1}, w_{i+2}, \cdots, w_n) \in V_{w,i-1}\), \((w_1, \dots, w_{i-1}, w_i w_{i+1}, w_{i+2}, \dots, w_n) \in V_{w,i-2}\). Note that 

\noindent here \(1\le i\le n-1\), \(w_i w_{i+1} \in \Nontip(I) \setminus Q_0\), and the arrow (\ref{arrow in M}) has weight \((-1)^i\).
Moreover, by \cite[Page 15]{CLZ24}, \(\calM\) is a partial matching of \(\overline{Q}_B\); that is, every vertex in \(\overline{Q}_B\) is incident to at most one arrow in \(\calM\) and every arrow in \(\calM\) has invertible weight.

Under this partial matching \(\calM\), a new weighted quiver \(\overline{Q}_B^{\calM}\)  is obtained from \(\overline{Q}_B\) as follows:
\begin{enumerate}
    \item[(1)] The arrows in \(\overline{Q}_B\) that are not in \(\calM\) are retained and drawn as thick arrows;
    \item[(2)] each arrow in \(\calM\) of the form (\ref{arrow in M}) 
     %\[
    %d_n^i: (w_1, \dots, w_{i-1}, w_i, w_{i+1}, w_{i+2}, \dots, w_n) \to (w_1, \dots, w_{i-1}, w_i w_{i+1}, w_{i+2}, \dots, w_n)
   % \]
    is replaced by a dotted arrow  in the opposite direction, denoted by \(d_n^{-i}\), whose weight is the negative of the original weight. More precisely, the arrow \(d_n^{-i}\) has weight \((-1)^{i+1}\) and is given by 
\[
\begin{tikzpicture}[>=Stealth, baseline=(current bounding box.center)]
    % 正确定义节点
    \node (top) at (0, 0.4)  {$(w_1, \cdots, w_{i-1}, w_i w_{i+1}, w_{i+2}, \dots, w_n)$};
    \node (bot) at (-7.5, -0.5) {$(w_1, \cdots, w_{i-1}, w_i, w_{i+1}, w_{i+2}, \dots, w_n)$};
    
    % 虚线箭头 + 标注
    \draw[->, dashed] (top.south west) -- (bot.north east) 
        node[midway, above, font=\small, yshift=0pt] {$d_n^{-i}$}
        node[midway, below, font=\small, yshift=-0pt] {$(-1)^{i+1}$};
\end{tikzpicture}
\]
\end{enumerate}

 For \( n \in \mathbb{N} \), set
\[ \begin{aligned}
    \calU_n &= \{ v \in \mathcal{V}_n \mid v \text{ is the source of an arrow in } \calM \},\\
    \calD_n &= \{ v \in \mathcal{V}_n \mid v \text{ is the target of an arrow in } \calM \},\\
    \mathcal{V}_n^{\calM} &= \mathcal{V}_n \setminus ( \calU_n \cup \calD_n ).
\end{aligned}\]

For a path \( p \) in \(\overline{Q}_B^{\calM}\), let \( \varphi_p^{\calM} \) be the product of the weights of all the arrows along \( p \).

A path in \(\overline{Q}_B^{\calM}\) is called zigzag if its arrows alternate between thick and dotted arrows. 
For two vertices \( w \in \mathcal{V}_n \) and \( w' \in \mathcal{V}_m \), let \( \calP^{\calM}(w, w') \) be the set of all zigzag paths from \( w \) to \( w' \) in \(\overline{Q}_B^{\calM}\). If such a path exists, \(m\) must be $n$, $n-1$, or $n+1$. For two vertices \(w, w'\) in \(\mathcal{V}_n\), let \(\calP_1^{\calM}(w, w')\) be the set of all zigzag paths from \(w\) to \(w'\) in \(\overline{Q}_B^{\calM}\) starting with a dotted arrow.

By \cite[Theorem 4.3 ]{CLZ24}, \(\mathcal{V}_n^{\calM} =\calW^{(n-1)}\). Together with \cite[Theorem 3.3]{CLZ24}, this yields a new complex \(B_*^{\calM}=(B^{\calM}_*, d^{\calM})\) from \(B_*\), defined as follows: for \(n\ge 0\),
\[B_n^{\calM} = A \otimes_E K\calW^{(n-1)} \otimes_E A\cong \bigoplus_{(w_1,w_2,\dots,w_n)\in \calW^{(n-1)}} A^e \cdot (w_1,w_2,\cdots,w_n).\] 
The augmentation $d_0^{\calM} : B_0^{\calM}= A \otimes_E A \to A$ is given by \( d_0(a_0\otimes_E a_1)=a_0a_1\), and for \(n\ge 1\), the differential
$$d_n^{\calM} : B_n^{\calM} \to B_{n-1}^{\calM}$$
is given by
\begin{equation}
 d_n^{\calM}=\sum_{\substack{p \in \calP^{\calM}(w,w') \\[2pt] w \in \calW^{(n-1)},\; w' \in \calW^{(n-2)}}} \varphi_p^{\calM}.
\end{equation}
The complex \(B_*^{\calM}\), called the two-sided Anick resolution of \(A\), is a projective \(A^e\)-resolution of \(A\) homotopy equivalent to \(B_*\).

%---------------------------------------------------------------------------------------%
\subsection{Comparison morphisms and weak self-homotopies.}\mbox{}

Let $\Gamma$ be a $K$-algebra.  Let \(f: M \to N\) be a homomorphism of \(\Gamma\)-modules. Given projective resolutions 
$P_*\to M\to 0$ and $Q_*\to N\to 0$, by the Comparison Theorem, there exists a chain map \(f_*: P_*  \to Q_*\) lifting \(f\), unique up to homotopy. This chain map is called comparison morphism. %In this section, we first recall weak self-homotopy and use it, following \cite[Section 2]{IIVZ15}, to construct comparison morphisms between projective resolutions. 
\begin{definition}  [{\cite[Section 1.3]{BZZ09}}]
   Let
\[
\cdots \to Q_n \xrightarrow{d_n^Q} Q_{n-1} \xrightarrow{d_{n-1}^Q} \cdots \xrightarrow{d_2^Q} Q_1 \xrightarrow{d_1^Q} Q_0 \xrightarrow{d_0^Q} N \to 0
\] 
be a complex of \(\Gamma\)-modules. A weak self-homotopy on the complex \(Q_*\) consists of \(K\)-linear maps \(t_n:Q_n\to Q_{n+1}\) for \(n\ge 0\) and \(t_{-1}: N \to Q_0\) such that \(\id_{Q_n}=t_{n-1}\circ d_n^Q+d_{n+1}^Q\circ t_n\) for \(n\ge 0\) and \(\id_N=d_0^Q\circ t_{-1}\).
\end{definition}

Let \(f: M\to N\) be a homomorphism of \(\Gamma\)-modules. Let 
\[
\cdots \to P_n \xrightarrow{d_n^P} P_{n-1} \xrightarrow{d_{n-1}^P} \cdots \xrightarrow{d_2^P} P_1 \xrightarrow{d_1^P} P_0 \xrightarrow{d_0^P} M \to 0
\] 
be a projective \(\Gamma\)-resolution of \(M\). Then for \(n\ge 0\), there are sets \(\{e_{n,i}\}_{i \in X_n} \subset P_n\) and \(\{f_{n,i}\}_{i \in X_n} \subset \Hom_{\Gamma}(P_n, \Gamma)\) such that \(x=\sum_i f_{n,i}(x)e_{n,i}\) for \(x\in P_n\). Suppose that $Q_* \to N$
is a projective \(\Gamma\)-resolution of \(N\) equipped with a weak self-homotopy \(\{t_n\}_{n\ge -1}\) (with the convention \(Q_{-1}=N\)).

We now construct a chain map \(f_*:P_* \to Q_*\) lifting \(f\) by the method of \cite{IIVZ15}. Set \(f_{-1}=f\). Define \(f_0\) on the generators \(e_{0,i}\) of \(P_0\) by
\[f_0(e_{0,i})=t_{-1}\circ f_{-1}\circ d_0^P(e_{0,i}).\]
Then \(d_0^Q\circ f_0(e_{0,i})=f_{-1}\circ d_0^P(e_{0,i})\). For any \(x \in P_0\), define \(f_0(x)=\sum_i f_{0,i}(x)f_0(e_{0,i})\), which gives a well-defined \(\Gamma\)-module homomorphism \(f_0:P_0 \to Q_0\) satisfying \(d_0^Q\circ f_0=f_{-1}\circ d_0^P\).

Assume inductively that the \(\Gamma\)-module homomorphisms \(f_0,\cdots, f_{n-1}\) have been constructed and satisfy \(d_i^Q\circ f_i=f_{i-1}\circ d_i^P\) for \(1 \le i \le n-1\). Define \(f_n\) on the generators \(e_{n,i}\) of \(P_n\) by 
\[f_n(e_{n,i})=t_{n-1}\circ f_{n-1}\circ d_n^P(e_{n,i}).\]
Then \(d_n^Q\circ f_n(e_{n,i})=f_{n-1}\circ d_n^P(e_{n,i})\). For any \(x \in P_n\), define \(f_n(x)=\sum_i f_{n,i}(x)f_n(e_{n,i})\), which gives a well-defined \(\Gamma\)-module homomorphism \(f_n: P_n\to Q_n\), satisfying \(d_n^Q\circ f_n=f_{n-1}\circ d_n^P\). This inductive construction yields the following proposition.
\begin{prop}[{\cite[Proposition 2.2]{IIVZ15}}]
    The map \(f_*:P_* \to Q_*\) defined above is a chain map lifting $f: M\to N$.
\end{prop}

Note that when $\Gamma=A^e$ for a $K$-algebra $A$ and $Q_*$ is a projective bimodule resolution of $A$, then as $A$ is projective as left or right $A$-module, a weak self-homotopy on $Q_*$ may be chosen to  consist of homomorphisms of left or right $A$-modules.

%---------------------------------------------------------------------------------------%

\section{The two-sided Anick resolution}\label{sec: Anick resolution}
 Recall the quiver with relations of the Xu--Snashall algebra (\cite[Example 4.1]{Sna09}):
 Let \(K\) be a field and let \(Q\) be the following quiver:
\[
\begin{tikzcd}
    1 \arrow[loop above, out =120, in =60, looseness=4,"a"] \arrow[loop below,out =300, in =240, looseness=5,"b"] \arrow[r,"c"] & 2
\end{tikzcd}
\]
Then the Xu--Snashall algebra is given by  \(A=KQ/I\), where \(I\) is the ideal of the path algebra \(KQ\) generated by the relations \(\{a^2, b^2, ab-ba, ac\}\). 

%Then \(A\) is an associative \(K\)-algebra with unit \(1_A=e_1+e_2\).

Consider the  left length-lexicographic order \(<\) on the multiplicative basis \(\calB = Q_{\ge 0}\) of \(KQ\), uniquely determined by imposing \(e_2 < e_1 < c < b < a\). 
Set $f_1 := a^2$, $f_2 := b^2$, $f_3 := ab - ba$, $f_4 := ac$  and \( \mathscr{G} := \{f_1, f_2, f_3, f_4\}\). The tips of \(f_i\) are as follows:
\[
\Tip(f_1) = a^2,\quad 
\Tip(f_2) = b^2,\quad 
\Tip(f_3) = ab,\quad 
\Tip(f_4) = ac.
\]
Note that \(\mathscr{G}\) is tip-reduced and consists of uniform elements. 

We verify that every overlap relation \(o(f_i, f_j, p, q)\) with \(1 \le i,j \le 4\) reduces to \(0\); these overlap relations are displayed as follows:
\[
\begin{aligned}
o(f_1, f_1, a, a) &= a^2a - aa^2 = 0,\\
%o(f_1, f_1, e_1, e_1) &= a^2e_1 - e_1a^2 = 0,\\
o(f_1, f_3, b, a) &= a^2b - a(ab - ba) = aba = e_1f_3a + bf_1e_1 ,\\
o(f_1, f_4, c, a) &= a^2c - a(ac) = 0,\\
o(f_2, f_2, b, b) &= b^2b - bb^2 = 0,\\
%o(f_2, f_2, e_1, e_1) &= b^2e_1 - e_1b^2 = 0,\\
o(f_3, f_2, b, a) &= (ab - ba)b - ab^2 = -bab =-bf_3e_1 -e_1f_2a.
%o(f_3, f_3, e_1, e_1) &= (ab - ba)e_1 - e_1(ab - ba) = 0,\\
%o(f_4, f_4, e_2, e_1) &= ace_2 - e_1ac=0.
\end{aligned}
\]
%Every overlap relation reduces to 0. 
Hence, by Theorem \ref{thm: GS basis}, \(\mathscr{G}\) is a Gröbner-Shirshov system for \(I\) with respect to \(<\). Consequently,
\[  \Nontip(I) =\{e_1,e_2, a, b, c, ba, bc\}.\]
We have thus proved the following result:
\begin{lem}Under the above setup, 
\(\mathscr{G}\) is a reduced Gröbner-Shirshov system for \(I\) with respect to \(<\) and 
$\Nontip(I) =\{e_1, e_2, a, b, c, ba, bc\}$ is a linear basis of the Xu--Snashall algebra $A$. 
\end{lem}

Set \(\calW=\Tip(\mathscr{G})=\{a^2, b^2, ab, ac\}\).  The Ufnarovski\u{\i} graph \(Q_\calW=(V,E)\) of \(Q\) with respect to \(\calW\) is given by 
\[V=\{e_1, e_2, a, b, c\} \text{ and }
E= \{e_1 \to a, e_1 \to b, e_1 \to c, a\to a, a\to b, a\to c, b\to b\},\]
and can be drawn as follows:
\[
\begin{tikzcd}
&& e_1 \arrow[ld]  \arrow[d]  \arrow[rd] & e_2\\
 &   b \arrow[loop below, out =300, in =240, looseness=5] & a \arrow[loop below,out =300, in =240, looseness=5] \arrow[l]\arrow[r] & c
\end{tikzcd}
\]

By construction, the Anick  chains of \(A\) are as follows:
\[ 
    \calW^{(-1)}= \{e_1, e_2\}, \; \calW^{(0)}= \{(a), (b), (c)\},\]
    and for $i\ge 1$, 
    $$\calW^{(i)}=\{(\underbrace{a, \cdots,a}_{i+1}), (\underbrace{a, \cdots,a}_{i+1-r}, \underbrace{b, \cdots, b}_{r}), 1 \le r \le i , (\underbrace{b, \cdots,b}_{i+1}), (\underbrace{a, \cdots,a}_{i}, c)\}.$$

For convenience, we introduce the following notations:
\[\begin{gathered}
    s_0^0 = e_1, s_1^0 =e_2, 
    s_0^1= (a), s_1^1= (b), s_2^1 = (c),\end{gathered}\]
    and for $n\ge 2$, 
   \[\begin{gathered}s_0^n=(\underbrace{a, \cdots,a}_{n}), s_r^n=(\underbrace{a, \cdots,a}_{n-r}, \underbrace{b, \cdots, b}_{r}), 1 \le r \le n, s_{n+1}^n= (\underbrace{a, \cdots,a}_{n-1}, c).\end{gathered}\]
   
Let \(E=Ke_1\oplus Ke_2 \). The components of the two-sided Anick resolution \(B_*^{\calM}\) of the algebra \(A=KQ/I\) are given by
\[B_n^{\calM}  = A \otimes_E K\calW^{(n-1)} \otimes_E A \cong \bigoplus_{i=0}^{n+1} A^e \cdot s_i^n.\] 
where each direct summand is a projective \(A^e\)-module   isomorphic to  some \(Ae_i \otimes_K e_jA\) for   \(i,j\in \{1,2\}\). 
For simplicity, write \(\otimes\) for \(\otimes_K\) whenever there is no confusion. 

Now we compute the differential \(d^{\calM}\) by giving the images of generators under the differential, which can be summarized  as follows:
\begin{theorem}
\begin{enumerate}
    \item When \(n=0\), $$d_0^{\calM}: B_0^{\calM}=A^e \cdot e_1\oplus A^e \cdot e_2\to A$$
is given by $$d_0^{\calM}(e_1) = e_1, d_0^{\calM}(e_2) = e_2.$$
  \item When \(n=1\),  $$d_1^{\calM}: B_1^{\calM}=A^e \cdot (a)\oplus A^e \cdot (b)\oplus A^e \cdot (c)\to B_0^{\calM}=A^e \cdot e_1\oplus A^e \cdot e_2$$
is given by  \[ \begin{aligned}
    &d_1^{\calM}((a))= (a \otimes 1- 1\otimes a)  e_1, \\
    &d_1^{\calM}((b))= (b \otimes 1- 1\otimes b) e_1, \\
    &d_1^{\calM}((c))= (c \otimes 1)  e_2 - (1\otimes c)  e_1.
\end{aligned}
\]
 
\item When \(n\ge 2\), $$d_n^{\calM}: B_n^{\calM}  = \bigoplus_{i=0}^{n+1} A^e \cdot s_i^n \to  B_{n-1}^{\calM}  = \bigoplus_{i=0}^{n} A^e \cdot s_i^{n-1} $$ is given by 
  \[\begin{aligned}
    &d_n^{\calM}(s_0^n)= (a \otimes 1 + (-1)^n1\otimes a)  s_0^{n-1}, \\
    &d_n^{\calM}(s_r^n)= (a \otimes 1 + (-1)^{n+r}1\otimes a) s_r^{n-1}+((-1)^n 1\otimes b+(-1)^{n-r}b \otimes 1)s_{r-1}^{n-1}, 1\le r \le n-1, \\
   &d_n^{\calM}(s_n^n)= (b \otimes 1+ (-1)^n 1\otimes b) s_{n-1}^{n-1}, \\
   &d_n^{\calM}(s_{n+1}^n)= (a \otimes 1) s_n^{n-1} + (-1)^n (1\otimes c)  s_0^{n-1}.
\end{aligned}\]
\end{enumerate}
\end{theorem}

The rest of this section is devoted to the proof of this result. 

The computations of $d_0^{\calM}$ and $d_1^{\calM}$ are easy.

When \(n=2\), the generators $(a,a)$, $(b,b)$, and $(a,c)$ each have two outgoing thick arrows with targets in $\calW^{(0)}$. 
The generator $(a,b)$ has three outgoing thick arrows; one of the thick arrows is
\[
\begin{tikzcd}
 (a,b) \arrow[r,"-1"]  & (ba) 
\end{tikzcd}
\]
with weight $-1$. The vertex $(ba)$ has one outgoing dotted arrow to $(b,a)$ with weight $1$. The vertex $(b,a)$ has two outgoing thick arrows with targets in $\calW^{(0)}$. Consequently, 
$$d_2^{\calM}: B_2^{\calM}=A^e \cdot (a, a)\oplus A^e \cdot (a, b)\oplus A^e \cdot (b, b)\oplus  A^e \cdot (a, c)\to B_1^{\calM}=A^e \cdot (a)\oplus A^e \cdot (b)\oplus A^e \cdot (c)$$
is given by
\[ \begin{aligned}
    &d_2^{\calM}((a,a))= (a \otimes 1 + 1\otimes a)  (a),\\
    &d_2^{\calM}((a,b))= (1\otimes b-b \otimes 1 ) (a) + (a \otimes 1 - 1 \otimes a) (b), \\
    &d_2^{\calM}((b,b))= (b \otimes 1+ 1\otimes b) (b), \\
    &d_2^{\calM}((a,c))= (a \otimes 1) (c) + (1\otimes c)  (a).
\end{aligned}\]

When \(n\ge 3\), the computation of 
$$d_n^{\calM}: B_n^{\calM}  = \bigoplus_{i=0}^{n+1} A^e \cdot s_i^n \to  B_{n-1}^{\calM}  = \bigoplus_{i=0}^{n} A^e \cdot s_i^{n-1} $$ is considerably harder. Each of the generators \((\underbrace{a, \cdots,a}_{n})\), \((\underbrace{b, \cdots, b}_{n})\), and \((\underbrace{a, \cdots,a}_{n-1},c)\) has exactly two outgoing thick arrows, both with targets in \(\calW^{(n-2)}\). Thus, the first easy formulas are 
\[ \begin{aligned}
    &d_n^{\calM}((\underbrace{a, \cdots,a}_{n}))= (a \otimes 1 + (-1)^n1\otimes a)  (\underbrace{a, \cdots,a}_{n-1}),\\
    &d_n^{\calM}((\underbrace{b, \cdots, b}_{n}))= (b \otimes 1+ (-1)^n 1\otimes b)(\underbrace{b, \cdots, b}_{n-1}), \\
    &d_n^{\calM}((\underbrace{a, \cdots,a}_{n-1},c))= (a \otimes 1) (\underbrace{a, \cdots,a}_{n-2},c) + (-1)^n (1\otimes c)  (\underbrace{a, \cdots,a}_{n-1}).
\end{aligned}\]

Then we compute \(d_n^{\calM}((\underbrace{a, \cdots,a}_{n-1},b))\), which amounts to finding all zigzag paths from $(\underbrace{a, \cdots,a}_{n-1},b)$ to  $\calW^{(n-2)}$ in the new weighted quiver $\overline{Q}_B^{\calM}$.
There are exactly two zigzag paths from \((\underbrace{a, \cdots,a}_{n-1},b)\) to \((\underbrace{a, \cdots,a}_{n-2},b)\):  one is
\[
\begin{tikzcd}
 (\underbrace{a, \cdots,a}_{n-1},b) \arrow[r,"a \otimes 1"]  & (\underbrace{a, \cdots,a}_{n-2},b) 
\end{tikzcd}
\]
with weight $a \otimes 1$, and the other is
\[\begin{tikzcd}
   & (\underbrace{a, \cdots,a}_{n-1},b) \arrow[r,"(-1)^{n-1}"] & (\underbrace{a, \cdots,a}_{n-2},ba) \arrow[ld,dashed,"(-1)^n",swap] \\
   & (\underbrace{a, \cdots,a}_{n-2},b,a) \arrow[r,"(-1)^n1\otimes a"]  & (\underbrace{a, \cdots,a}_{n-2},b) 
\end{tikzcd}\]
with weight $(-1)^{n+1}1 \otimes a$.

There are exactly two zigzag paths from \((\underbrace{a, \cdots,a}_{n-1},b)\) to \((\underbrace{a, \cdots,a}_{n-1})\): one is
\[\begin{tikzcd}
   & (\underbrace{a, \cdots,a}_{n-1},b) \arrow[r,"(-1)^{n-1}"] & (\underbrace{a, \cdots,a}_{n-2},ba) \arrow[ld,dashed,"(-1)^n",swap]\\
   & (\underbrace{a, \cdots,a}_{n-2},b,a) \arrow[r,"(-1)^{n-2}"]  & (\underbrace{a, \cdots,a}_{n-3},ba,a) \arrow[ld,dashed,"(-1)^{n-1}",swap]\\
   & (\underbrace{a, \cdots,a}_{n-3},b,a,a) \arrow[r,"\cdots"]  & (\cdots) \arrow[ld,dashed,"\cdots",swap]\\
   &(\cdots) \arrow[r,"\cdots"] & (ba,\underbrace{a, \cdots,a}_{n-2}) \arrow[ld,dashed,"(-1)^{2}",swap]\\
   &(b,\underbrace{a, \cdots,a}_{n-1}) \arrow[r,"b \otimes 1"] & (\underbrace{a, \cdots,a}_{n-1})
\end{tikzcd}\]
with weight $(-1)^{n-1} b \otimes 1$, and the other is
\[
\begin{tikzcd}
 (\underbrace{a, \cdots,a}_{n-1},b) \arrow[r,"(-1)^n 1 \otimes b"]  & (\underbrace{a, \cdots,a}_{n-1}) 
\end{tikzcd}
\]
with weight $(-1)^n 1 \otimes b$.
Moreover, there is no zigzag path from \((\underbrace{a, \cdots,a}_{n-1},b)\) to any other vertex in \(\calW^{(n-2)}\). Hence,
\[d_n^{\calM}((\underbrace{a, \cdots,a}_{n-1},b) )=(a\otimes 1+(-1)^{n+1}1\otimes a)(\underbrace{a, \cdots,a}_{n-2},b)+((-1)^n 1\otimes b+(-1)^{n-1}b\otimes1)(\underbrace{a, \cdots,a}_{n-1}).\] 
%that is, 
%\[d_n^{\calM}(s_1^n)=(a\otimes 1+(-1)^{n+1}1\otimes a)s_1^{n-1}+((-1)^n 1\otimes b+(-1)^{n-1}b\otimes1)s_0^{n-1}.\]

Similarly, we can compute
\[d_n^{\calM}((a,\underbrace{b, \cdots,b}_{n-1}))=(a\otimes 1-1\otimes a)(\underbrace{b, \cdots,b}_{n-1})+((-1)^n 1\otimes b-b\otimes1)(a, \underbrace{b, \cdots,b}_{n-2}).\]

All cases with \(n=3\) have been discussed above. It remains to compute the remaining cases with \(n\ge 4\), namely, 
$$d_n^{\calM}((\underbrace{a, \cdots,a}_{n-r},\underbrace{b, \cdots,b}_r))$$
for \(2 \le r \le n-2\), which amounts to finding all zigzag paths from $(\underbrace{a, \cdots,a}_{n-r},\underbrace{b, \cdots,b}_{r})$ to $W^{(n-2)}$ in the new weighted quiver $\overline{Q}_B^{\calM}$.
 The only outgoing arrows from \((\underbrace{a, \cdots,a}_{n-r},\underbrace{b, \cdots,b}_{r})\) are the following three thick arrows: 
\[   \begin{tikzcd}
 (\underbrace{a, \cdots,a}_{n-r},\underbrace{b, \cdots,b}_{r}) \arrow[r,"a \otimes 1"]  & (\underbrace{a, \cdots,a}_{n-r-1}, \underbrace{b, \cdots,b}_{r}) \in \calW^{(n-2)}
    \end{tikzcd}\]
\[\begin{tikzcd}
 (\underbrace{a, \cdots,a}_{n-r},\underbrace{b, \cdots,b}_{r}) \arrow[r,"(-1)^n 1 \otimes b"]  & (\underbrace{a, \cdots,a}_{n-r}, \underbrace{b, \cdots,b}_{r-1}) \in \calW^{(n-2)} 
    \end{tikzcd}\]
\begin{equation}\label{arrow1}
    \begin{tikzcd}
 (\underbrace{a, \cdots,a}_{n-r},\underbrace{b, \cdots,b}_{r}) \arrow[r,"(-1)^{n-r}"]  & (\underbrace{a, \cdots,a}_{n-r-1}, ba, \underbrace{b, \cdots,b}_{r-1}) \in \calD_{n-1}
    \end{tikzcd}
\end{equation}
Only the target of the arrow (\ref{arrow1}) does not belong to \(\calW^{(n-2)}\), and it is the source of exactly one dotted arrow in \(\overline{Q}_B^{\calM}\):
\[\begin{tikzcd}
 & (\underbrace{a, \cdots,a}_{n-r-1}, ba, \underbrace{b, \cdots,b}_{r-1}) \arrow[ld,dashed,"(-1)^{n-r+1}",swap] \\
 (\underbrace{a, \cdots,a}_{n-r-1}, b, a, \underbrace{b, \cdots,b}_{r-1})  
    \end{tikzcd}\] 
whose composition with the arrow (\ref{arrow1}) has weight $-1$. Hence, the problem reduces to finding all zigzag paths from \((\underbrace{a, \cdots,a}_{n-r-1}, b, a, \underbrace{b, \cdots,b}_{r-1})\) to $\calW^{(n-2)}$ in the new weighted quiver $\overline{Q}_B^{\calM}$. 
To handle these paths systematically, we introduce some notations.
\iffalse
We want to find a zigzag path in $\overline{Q}_B^{\calM}$ starting from the vertex \((\underbrace{a, \cdots,a}_{n-r},\underbrace{b, \cdots,b}_{r})\), passing through the arrow  (\ref{arrow1}) and ending at some vertex in \(\calW^{(n-2)}\) is equivalent to finding a zigzag path in $\overline{Q}_B^{\calM}$ from the vertex \((\underbrace{a, \cdots,a}_{n-r-1}, b, a, \underbrace{b, \cdots,b}_{r-1})\) to some vertex in \(\calW^{(n-2)}\).
\fi

Let \(\T_n\) be the set of all sequences of the form \((\underbrace{a, \cdots,a}_{m_1}, \underbrace{b, \cdots,b}_{m_2}, \underbrace{a, \cdots,a}_{m_3}, \underbrace{b, \cdots,b}_{m_4},*)\) satisfying the following conditions:
\begin{enumerate}
    \item  \(m_1 \ge 0\), $m_2 \ge 1$, $m_3 \ge 1$, $m_4 \ge 0$;
    \item $*$ is a (possibly empty) sequence of length $n-\sum_{i=1}^4m_i$ over $\{a,b\}$;
    \item the number of occurrences of each of \(a\) and \(b\) in the whole sequence is at least 2;
    \item if $m_4 \ge 1$ and $*\neq \emptyset$, then \(*\)  starts with \(a\). If $m_4 = 0$, $*=\emptyset$.
\end{enumerate}

Let \(\calS_n\) be the set of all 5-tuples \((m_1, m_2, m_3, m_4, *)\) satisfying the following conditions:
\begin{enumerate}
    \item  \(m_1 \ge 0, m_2 \ge 1, m_3 \ge 1, m_4 \ge 0, \sum_{i=1}^4m_i \le n\) and \(m_i \in \mathbb{Z}\) for \(i=1,2,3,4\);
    \item  $*$ is a (possibly empty) sequence of length $n-\sum_{i=1}^4m_i$ over $\{a,b\}$;
    \item the number of \(a\)'s in \(*\) is at least \(2-m_1-m_3\), and the number of \(b\)'s in \(*\) is at least  \(2 - m_2 -m_4\);
    \item if $m_4 \ge 1$ and $*\neq \emptyset$, then \(*\)  starts with \(a\). If $m_4 = 0$, $*=\emptyset$.
\end{enumerate} 
The sets \(\T_n\) and \(\calS_n\) are in bijection via the natural correspondence
\[(\underbrace{a, \cdots,a}_{m_1}, \underbrace{b, \cdots,b}_{m_2}, \underbrace{a, \cdots,a}_{m_3}, \underbrace{b, \cdots,b}_{m_4},*) \longleftrightarrow (m_1, m_2, m_3, m_4, *).\]
From now on, we identify \((m_1, m_2, m_3, m_4, *)\in \calS_n\) with the corresponding sequence in \(\T_n\).

A short path is defined as either a single thick arrow or a thick arrow followed by a dotted arrow.  
Since $\T_n \subset \calU_n$, all the arrows in $\overline{Q}_B^{\calM}$ starting at $\T_n$ are thick arrows. All zigzag paths from a sequence in $\T_n$ to $\calW^{(n-2)}$ can be decomposed into short paths. Then we list all short paths that start at $(\underbrace{a, \cdots,a}_{m_1}, \underbrace{b, \cdots,b}_{m_2}, \underbrace{a, \cdots,a}_{m_3}, \underbrace{b, \cdots,b}_{m_4},*)$ in $\T_n$  as follows: 

\noindent \textbf{Case 1} If \(m_1 \ge 1\),

\noindent \hspace*{2em}\textbf{Case 1.1} If \(m_3 \ge 2\), from this vertex, there is only one short path
\begin{equation} \label{path1}
    \begin{tikzcd}
   & (\underbrace{a, \cdots,a}_{m_1}, \underbrace{b, \cdots,b}_{m_2}, \underbrace{a, \cdots,a}_{m_3}, \underbrace{b, \cdots,b}_{m_4},*) \arrow[r,"(-1)^{m_1}"] & (\underbrace{a, \cdots,a}_{m_1-1}, ba, \underbrace{b, \cdots,b}_{m_2-1}, \underbrace{a, \cdots,a}_{m_3}, \underbrace{b, \cdots,b}_{m_4},*) \arrow[ld,dashed,"(-1)^{m_1+1}",swap] \\
   & (\underbrace{a, \cdots,a}_{m_1-1}, b, a, \underbrace{b, \cdots,b}_{m_2-1}, \underbrace{a, \cdots,a}_{m_3}, \underbrace{b, \cdots,b}_{m_4},*)
        \end{tikzcd}
\end{equation}
 with weight $-1$.
 
\noindent \hspace*{4em}\textbf{Case 1.1.1} If \(m_2 \ge 2\), the short path (\ref{path1}) reaches \((m'_1, m'_2, m'_3, m'_4, *') \in \T_n\), where 
\begin{equation}
    m'_1= m_1-1 \ge 0,\;  m'_2= 1,\;  m'_3=1,\;  m'_4 =m_2-1\ge 1, \text{ and } *' = (\underbrace{a, \cdots,a}_{m_3}, \underbrace{b, \cdots,b}_{m_4},*).
\end{equation}

\noindent \hspace*{4em}\textbf{Case 1.1.2} If \(m_2=1\), the short path (\ref{path1}) reaches \((m'_1, m'_2, m'_3, m'_4, *')\in \T_n\), where
\begin{equation}
    m'_1= m_1-1 \ge 0,\;  m'_2= 1,\;  m'_3=m_3+1\ge 2,\;  m'_4 =m_4\ge 0, \text{ and } *'=*.
\end{equation}

\noindent \hspace*{2em}\textbf{Case 1.2} If \(m_3=1\), \( m_4 \ge 1\), from this vertex, there are exactly two short paths. The first is  
\begin{equation} \label{path2}
    \begin{tikzcd}
   & (\underbrace{a, \cdots,a}_{m_1}, \underbrace{b, \cdots,b}_{m_2}, a, \underbrace{b, \cdots,b}_{m_4},*) \arrow[r,"(-1)^{m_1}"] & (\underbrace{a, \cdots,a}_{m_1-1}, ba, \underbrace{b, \cdots,b}_{m_2-1}, a, \underbrace{b, \cdots,b}_{m_4},*) \arrow[ld,dashed,"(-1)^{m_1+1}",swap] \\
   & (\underbrace{a, \cdots,a}_{m_1-1}, b, a, \underbrace{b, \cdots,b}_{m_2-1}, a, \underbrace{b, \cdots,b}_{m_4},*)
        \end{tikzcd}
\end{equation}
with weight $-1$. This leads to two subcases:

\noindent \hspace*{4em}\textbf{Case 1.2.1} If \(m_2\ge 2\), the short path (\ref{path2}) reaches \((m'_1, m'_2, m'_3, m'_4, *' )\in \T_n\), where 
\begin{equation}
    m'_1= m_1-1 \ge 0,\;  m'_2= 1,\;  m'_3=1,\;  m'_4 =m_2-1\ge 1, \text{ and } *'=(a, \underbrace{b, \cdots,b}_{m_4},*).
\end{equation}

\noindent \hspace*{4em}\textbf{Case 1.2.2} If \(m_2=1\), the short path (\ref{path2}) reaches \((m'_1, m'_2, m'_3, m'_4, *' )\in \T_n\), where 
\begin{equation}
    m'_1= m_1-1 \ge 0,\;  m'_2= 1,\;  m'_3=2,\;  m'_4 =m_4\ge 1, \text{ and } *'=*.
\end{equation}

The second is
\begin{equation} \label{path3}
    \begin{tikzcd}[column sep=5em]
   & (\underbrace{a, \cdots,a}_{m_1}, \underbrace{b, \cdots,b}_{m_2},a, \underbrace{b, \cdots,b}_{m_4},*) \arrow[r,"(-1)^{m_1+m_2+1}"] & (\underbrace{a, \cdots,a}_{m_1}, \underbrace{b, \cdots,b}_{m_2}, ba, \underbrace{b, \cdots,b}_{m_4-1},*) \arrow[ld,dashed,"(-1)^{m_1+m_2+2}",swap] \\
   & (\underbrace{a, \cdots,a}_{m_1}, \underbrace{b, \cdots,b}_{m_2+1}, a, \underbrace{b, \cdots,b}_{m_4-1},*)
        \end{tikzcd}
\end{equation}
with weight $-1$. This leads to two subcases:

\noindent \hspace*{4em}\textbf{Case 1.2.3}   If \(m_4\ge 2\), the short path (\ref{path3}) reaches \((m'_1, m'_2, m'_3, m'_4, *')\in \T_n\), where
\begin{equation}
    m'_1= m_1 \ge 1,\;  m'_2= m_2+1\ge 2,\;  m'_3=1,\;  m'_4 =m_4-1\ge 1, \text{ and } *'=*.
\end{equation}

\noindent \hspace*{4em}\textbf{Case 1.2.4}   If \(m_4= 1\), the short path (\ref{path3}) reaches \((m'_1, m'_2, m'_3, m'_4, *') \in \T_n\), where if \(*= \emptyset\),
\begin{equation}
    m'_1= m_1 \ge 1,\;  m'_2= m_2+1\ge 2,\;  m'_3=1, \;  m'_4 =0, \text{ and } *'=\emptyset;
\end{equation}
if \(*\neq \emptyset\),
\begin{equation}
    m'_1= m_1 \ge 1, \;  m'_2= m_2+1\ge 2,\;  m'_3\ge 2,\;  m'_4 \ge 0.
\end{equation}

\noindent \hspace*{2em}\textbf{Case 1.3} If \(m_3=1, m_4 =0\), then \(m_2 \ge 2\) since the sequence contains at least two \(b\)'s. From this vertex, there are exactly two short paths. The first  is
\begin{equation} \label{path4}
    \begin{tikzcd}
   & (\underbrace{a, \cdots,a}_{m_1}, \underbrace{b, \cdots,b}_{m_2}, a) \arrow[r,"(-1)^{m_1}"] & (\underbrace{a, \cdots,a}_{m_1-1}, ba, \underbrace{b, \cdots,b}_{m_2-1}, a) \arrow[ld,dashed,"(-1)^{m_1+1}",swap] \\
   & (\underbrace{a, \cdots,a}_{m_1-1}, b, a, \underbrace{b, \cdots,b}_{m_2-1}, a)   
        \end{tikzcd}
\end{equation}
with weight $-1$.

\noindent \hspace*{4em}\textbf{Case 1.3.1} The short path (\ref{path4}) reaches \((m'_1, m'_2, m'_3, m'_4, *') \in \T_n\),  where
\begin{equation}
    m'_1= m_1-1 \ge 0,\;  m'_2= 1,\;  m'_3=1,\;  m'_4 =m_2-1\ge 1, \text{ and } *'=(a).
\end{equation}
The second is
\begin{equation} \label{path5}
    \begin{tikzcd}
   & (\underbrace{a, \cdots,a}_{m_1}, \underbrace{b, \cdots,b}_{m_2}, a) \arrow[rr,"(-1)^{n} 1 \otimes a"] && (\underbrace{a, \cdots,a}_{m_1}, \underbrace{b, \cdots,b}_{m_2})  
        \end{tikzcd}
\end{equation}
with weight $(-1)^{n} 1 \otimes a$.

\noindent \hspace*{4em}\textbf{Case 1.3.2} The target of the short path (\ref{path5}) belongs to \(\calW^{(n-2)}\). 

\noindent \textbf{Case 2} If \(m_1 =0\).

\noindent \hspace*{2em}\textbf{Case 2.1} If \(m_3\ge 2\) and \(m_2 \ge 2\), there is no zigzag path from \((\underbrace{b, \cdots,b}_{m_2}, \underbrace{a, \cdots,a}_{m_3}, \underbrace{b, \cdots,b}_{m_4},*)\) to \(\calW^{(n-2)}\).

\noindent \hspace*{2em}\textbf{Case 2.2} If \(m_3\ge 2, m_2 =1\), then \(m_4 \ge 1\)  since the sequence contains at least two \(b\)'s.

\noindent \hspace*{4em}\textbf{Case 2.2.1} If \(*\neq \emptyset\), there is no zigzag path from \(( b, \underbrace{a, \cdots,a}_{m_3}, \underbrace{b, \cdots,b}_{m_4},*)\) to \(\calW^{(n-2)}\).

\noindent \hspace*{4em}\textbf{Case 2.2.2} If \(*=\emptyset\), from this vertex, there is only one short path 
\begin{equation} \label{path6}
    \begin{tikzcd}
   & (b, \underbrace{a, \cdots,a}_{m_3},  \underbrace{b, \cdots,b}_{m_4}) \arrow[r,"b \otimes 1"] & (\underbrace{a, \cdots,a}_{m_3}, \underbrace{b, \cdots,b}_{m_4}) 
        \end{tikzcd}
\end{equation}
with weight $b \otimes 1$. The target of the short path (\ref{path6}) belongs to  \(\calW^{(n-2)}\).

\noindent \hspace*{2em}\textbf{Case 2.3} If \(m_3=1\), then \(*\neq \emptyset\) and \(m_4 \ge 1\) since the sequence contains at least two \(a\)'s. From this vertex, there is only one short path
\begin{equation} \label{path7}
    \begin{tikzcd}
   & (\underbrace{b, \cdots,b}_{m_2}, a,  \underbrace{b, \cdots,b}_{m_4}, *) \arrow[r,"(-1)^{m_2+1}"] & (\underbrace{b, \cdots,b}_{m_2}, ba, \underbrace{b, \cdots,b}_{m_4-1}, *) \arrow[ld,dashed,"(-1)^{m_2+2}"]\\
   & (\underbrace{b, \cdots,b}_{m_2+1}, a, \underbrace{b, \cdots,b}_{m_4-1}, *)
        \end{tikzcd}
\end{equation}
with weight $-1$. This leads to two subcases:

\noindent \hspace*{4em}\textbf{Case 2.3.1} If \(m_4 \ge 2\), the short path (\ref{path7}) reaches \((m'_1, m'_2, m'_3, m'_4, *')\in \T_n\), where
\begin{equation}
    m'_1=0,\; m'_2=m_2+1 \ge 2,\; m'_3=1,\; m'_4=m_4-1 \ge 1, \text{ and } *'=*.
\end{equation}

\noindent \hspace*{4em}\textbf{Case 2.3.2} If \(m_4 =1\), the short path (\ref{path7}) reaches \((m'_1, m'_2, m'_3, m'_4, *')\in \T_n\), where
\begin{equation}
    m'_1=0,\; m'_2=m_2+1 \ge 2,\; m'_3\ge 2,\; m'_4\ge 0.
\end{equation}

\begin{remark}\label{rem: all used short path}
    For any vertex \((\underbrace{a, \cdots,a}_{m_1}, \underbrace{b, \cdots,b}_{m_2}, \underbrace{a, \cdots,a}_{m_3}, \underbrace{b, \cdots,b}_{m_4},*)\) in \(\T_n\), among all reachable vertices from this vertex via the above mentioned short paths,  the number of \(a\)'s following each \(b\) never decreases and the number of \(a\)'s preceding each \(b\) never increases.
\end{remark}

Using the discussion above, we determine all zigzag paths from \((\underbrace{a, \cdots,a}_{n-r-1}, b, a, \underbrace{b, \cdots,b}_{r-1}) \in \T_n\) to \(\calW^{(n-2)}\), where \(2 \le r \le n-2\). Under the notation introduced above, this vertex is denoted by \((n-r-1, 1, 1, r-1, *)\), where \(*=\emptyset\), and falls under \textbf{Case 1.2}. From this vertex, there are two short paths:
\begin{enumerate}
    \item The first choice is to follow the short path (\ref{path2}); this falls under \textbf{Case 1.2.2}, and reaches the vertex \((n-r-2, 1, 2, r-1, *)\in \T_n\).
    \item The second choice is to  follow the short path (\ref{path3}); this falls under \textbf{Case 1.2.3} or \textbf{Case 1.2.4}. Since \(*=\emptyset\), both cases reach the same vertex \((n-r-1, 2, 1, r-2, *)\in \T_n\).
\end{enumerate}

We now consider the subsequent zigzag paths for the first choice. If \(r=n-2\), then following the short path (\ref{path6}) from \((n-r-2, 1, 2, r-1, *)\), we reach \((\underbrace{a, \cdots,a}_{n-r}, \underbrace{b, \cdots,b}_{r-1})\in \calW^{(n-2)}\). 
If \(r<n-2\), the short path (\ref{path1}) from this vertex reaches \((n-r-3, 1, 3, r-1,*) \in \T_n\), which still falls under \textbf{Case 1.1.2} whenever  \(n-r-3 >0\). Hence, by repeatedly following the short path  (\ref{path1}) \((n-r-2)\) times from this vertex, we reach \((0,1, n-r, r-1,*)\in \T_n\), which falls under \textbf{Case 2.2.2}. Finally, following the short path (\ref{path6}), we reach \((\underbrace{a, \cdots,a}_{n-r}, \underbrace{b, \cdots,b}_{r-1})\in \calW^{(n-2)}\).

We now consider the subsequent zigzag paths for the second choice. Any zigzag path from \((\underbrace{a, \cdots,a}_{n-r-1}, b, b, a, \underbrace{b, \cdots,b}_{r-2})\) to a vertex in \(\calW^{(n-2)}\) must proceed in one of two ways: it either first follows some zigzag paths to  
\[(b, \underbrace{a, \cdots,a}_{n-r}, \underbrace{b, \cdots,b}_{r-1})\in \T_n,\]
which falls under \textbf{Case 2.2.2}, then follows the short path (\ref{path6}) to \((\underbrace{a, \cdots,a}_{n-r}, \underbrace{b, \cdots,b}_{r-1})\in \calW^{(n-2)}\); 
or first follows some zigzag paths to  
\[(\underbrace{a, \cdots,a}_{n-r-1}, \underbrace{b, \cdots,b}_{r}, a)\in \T_n,\] 
which falls under \textbf{Case 1.3.2}, then follows the short path (\ref{path5}) to \((\underbrace{a, \cdots,a}_{n-r-1}, \underbrace{b, \cdots,b}_{r})\in \calW^{(n-2)}\).

We continue with the first type. For the vertex \((n-r-1, 2, 1, r-2, *)\) with \(*\) empty, the number of \(a\)'s following the second \(b\) is one. By Remark \ref{rem: all used short path}, every vertex reached by a zigzag path from \((n-r-1, 2, 1, r-2, *)\) has at least one \(a\) following the second \(b\). Consequently, there is no zigzag path from  \((\underbrace{a, \cdots,a}_{n-r-1}, b, b, a, \underbrace{b, \cdots,b}_{r-2})\) to \((b, \underbrace{a, \cdots,a}_{n-r}, \underbrace{b, \cdots,b}_{r-1})\).

For the second type, 
\iffalse
we now consider the zigzag paths from \((n-r-1, 2, 1, r-2, *)\) to \((n-r-1, r, 1, 0, *)\). Since the two vertices have the same value of \(m_1 =n-r-1\ge 1\), it suffices to consider that \((n-r-1, 2, 1, r-2, *)\) falls under \textbf{Case 1.2.3}, \textbf{Case 1.2.4}, or \textbf{Case 1.3.2}. If \(r=2\), \((n-r-1, 2, 1, r-2, *)\) falls under \textbf{Case 1.3.2} and coincides with \((n-r-1, r, 1, 0, *)\), and there is no zigzag path of positive length from \((n-r-1, 2, 1, r-2, *)\) to \((n-r-1, r, 1, 0, *)\). If \(r=3\), \((n-r-1, 2, 1, r-2, *)\) falls under \textbf{Case 1.2.4} and then follows the short path (\ref{path3}) to \((n-r-1, r, 1, 0, *)\).  If \(r\ge 4\), \((n-r-1, 2, 1, r-2, *)\) falls under \textbf{Case 1.2.3}, then follows the short path (\ref{path3}) to \((n-r-1, 3, 1, r-3, *)\), which still falls under \textbf{Case 1.2.3} if \(r-3\ge 2\).  Thus, following the short path (\ref{path3}) \((r-3)\) times reaches  \((n-r-1,r-1, 1, 1,*)\in \T_n\), which falls under \textbf{Case 1.2.4}, and then following the short path (\ref{path3}), we reach \((n-r-1, r, 1, 0, *)\). In summary,
\fi
the only zigzag path from \((n-r-1, 2, 1, r-2, *)\) to \((n-r-1, r, 1, 0, *)\) is obtained by repeating the short path (\ref{path3}) \((r-2)\) times.

Based on the above discussion, for \(n\ge 4\) and $2\le r \le n-2$,
\[d_n^{\calM}((\underbrace{a, \cdots,a}_{n-r},\underbrace{b, \cdots,b}_r))= (a \otimes 1 + (-1)^{n+r}1\otimes a) (\underbrace{a, \cdots,a}_{n-r-1},\underbrace{b, \cdots,b}_r)+((-1)^n 1\otimes b+(-1)^{n-r}b \otimes 1)(\underbrace{a, \cdots,a}_{n-r},\underbrace{b, \cdots,b}_{r-1}).\]

% \( I_n = \{ (w_1, \dots, w_n) \mid w_i \in \Nontip(I) \setminus B_0 \text{ for } 1 \le i \le n, \text{and } w_1 \cdots w_n \text{ is a path in } Q \} \).
% \( \calU_n=\{(w_1,\cdots,w_n)\in V_{w,i-1} \mid w=w_1\cdots w_n\in \calB,1\le i\le n-1\}\).

% \( \calD_n=\{ (w_1,\cdots,w_{i-1},w_iw_{i+1},\cdots,w_{n+1})\mid (w_1,\cdots,w_{i-1},w_i,w_{i+1},\cdots,w_{n+1})\in V_{w,i-1},w=w_1\cdots w_n\in \calB,1\le i\le n\}\)

%---------------------------------------------------------------------------------------%
\section{Hochschild cohomology groups}\label{sec: Hochschild cohomology groups}
In the previous  section, we explicitly computed the two-sided Anick resolution of the Xu--Snashall algebra \(A\). In this section, we compute the Hochschild cohomology of \(A\)  using this two-sided Anick resolution.

Applying the functor \( \Hom_{A^e}(-, A) \) to the two-sided Anick resolution \(B_*^{\calM}\), we obtain a cochain complex $(B^{\calM})^*$:
\[ 0\rightarrow \Hom_{A^e}(B_0^{\calM}, A)\xrightarrow{\partial^0}\cdots\xrightarrow{\partial^{n-1}} \Hom_{A^e}(B_n^{\calM}, A)\xrightarrow{\partial^n}\Hom_{A^e}(B_{n+1}^{\calM}, A) \xrightarrow{\partial^{n+1}} \cdots \]
with $\partial^n(f)=-(-1)^n f\circ d_{n+1}^{\calM}$ for $n\ge 0$.
%Let \(d_0^{\calM}\) be the zero map. 
The \(n\)-th Hochschild cohomology group \(\rmHH^n(A)\) of A is
\[\rmHH^n(A)= \Ker( \partial^n) / \rmIm( \partial^{n-1} ). \]

%\(\Nontip(I)=\{e_1, e_2, a, b, c, ba, bc\}\) is a \(K\)-basis of \(A\). 
Set $z_1=e_1$, $z_2=e_2$, $z_3=a$, $z_4=b$, $z_5=c$, $z_6=ba$, $z_7=bc$.

Each \(A^e\)-module homomorphism \(f \) in \(\Hom_{A^e}(B_n^{\calM}, A)\cong \Hom_{A^e}(\bigoplus_{i=0}^{n+1} A^e \cdot s_i^n, A)\) is uniquely determined by the values \(f(s_i^n)\). Specifically, for \(n=0\),
\begin{equation}\label{euq: degree 0 homomorphisn}
    \begin{aligned}
    &f(s_0^0)= \alpha_{0,1} z_1+ \alpha_{0,3} z_3+ \alpha_{0,4} z_4+ \alpha_{0,6} z_6,\\
    &f(s_1^0)= \alpha_{1,2} z_2,
\end{aligned}
\end{equation}
where \(\alpha_{l,j} \in K\), for $l=0$, $j=1,3,4,6$ or $l=1$, $j=2$.
For \(n\ge 1\),
\begin{equation}\label{euq: degree n homomorphisn}
    \begin{aligned}
    &f(s_{l}^n)= \alpha_{l,1} z_1+ \alpha_{l,3} z_3+ \alpha_{l,4} z_4+ \alpha_{l,6} z_6, \; 0\le l \le n,\\
    &f(s_{n+1}^n)= \alpha_{n+1,5} z_5+ \alpha_{n+1,7} z_7,
\end{aligned}
\end{equation}
where \(\alpha_{l,j} \in K\), for $ 0\le l \le n$, $j=1,3,4,6$ or $l=n+1$, $j=5,7$.

Now we introduce a \(K\)-basis of \(\Hom_{A^e}(B_n^{\calM}, A)\). When \(n=0\), for each \(j\in \{1,3,4,6\}\), define  \(f_{0,j}^0 \in \Hom_{A^e}(B_0^{\calM}, A)\) by
\[f_{0,j}^0(s_0^0)=z_j,\; f_{0,j}^0(s_1^0)=0. \] 
Define  \(f_{1,2}^0\in \Hom_{A^e}(B_0^{\calM}, A) \) by
\[f_{1,2}^0(s_0^0)=0,\; f_{1,2}^0(s_1^0)=z_2. \]
These homomorphisms \( \{f_{0,j}^0, j=1,3,4,6; f_{1,2}^0\}\) form a \(K\)-basis of \(\Hom_{A^e}(B_0^{\calM}, A)\).
When \(n\ge 1\), for \(0\le l \le n\) and \(j=1,3,4,6\), define \(f_{l,j}^n\in \Hom_{A^e}(B_n^{\calM}, A) \) by
\[f_{l,j}^n(s_l^n)=z_j,\; f_{l,j}^n(s_k^n)=0,\; 0\le k \le n+1, k \neq l. \] 
For \(j=5,7\), define \(f_{n+1,j}^n\in \Hom_{A^e}(B_n^{\calM}, A) \) by
\[f_{n+1,j}^n(s_{n+1}^n)=z_j,\; f_{n+1,j}^n(s_k^n)=0,\; 0\le k \le n. \] 
These homomorphisms \(\{f_{l,j}^n, 0\le l \le n, j=1,3,4,6; f_{n+1,j}^n,j=5,7\}\) form a \(K\)-basis of \(\Hom_{A^e}(B_n^{\calM}, A)\).

Next we compute \(\rmHH^0(A) \). Let \(f \in \Hom_{A^e}(B_0^{\calM}, A)\) be as in (\ref{euq: degree 0 homomorphisn}). 
\[\begin{aligned}
    & \partial^0(f)(s_0^1) = -(-1)^0f(d_1^{\calM}(s_0^1))= -f((a \otimes1 -1\otimes a) s_0^0) =-af(s_0^0)+f(s_0^0)a =0, \\
    & \partial^0(f)(s_1^1)= -(-1)^0f(d_1^{\calM}(s_1^1))= -f((b \otimes1 -1\otimes b) s_0^0)= -bf(s_0^0) +f(s_0^0)b=0,\\
    & \partial^0(f)(s_2^1)= -(-1)^0f(d_1^{\calM}(s_2^1))= -f((-1\otimes c)s_0^0 +(c \otimes 1)s_1^0)=f(s_0^0)c - cf(s_1^0)= (\alpha_{0,1}- \alpha_{1,2})z_5 +\alpha_{0,4}z_7.
\end{aligned}\]
It follows that
\[ \rmIm( \partial^0)= \Span_K \{f_{2,5}^1, f_{2,7}^1\}.
\]
Moreover, \(f \in \Ker(\partial^0)\) if and only if \(\alpha_{0,1}-\alpha_{1,2}=0\) and \(\alpha_{0,4}=0\), so
\[\rmHH^{0}(A)=\Ker(\partial^0)=\Span_K \{f_{0,1}^0+f_{1,2}^0, f_{0,3}^0, f_{0,6}^0\}.
\]

Then we compute \(\rmHH^n(A)\) for \(n\ge 1\). Let \(f \in \Hom_{A^e}(B_n^{\calM}, A)\) be as in (\ref{euq: degree n homomorphisn}). 
\iffalse
\[\begin{aligned}
     (d_{n+1}^{\calM})^*(f)(s_0^{n+1}) &= f((a \otimes1 +(-1)^{n+1}1\otimes a) s_0^n)\\
     &=af(s_0^n)+(-1)^{n+1}f(s_0^n)a \\
     &=(1+(-1)^{n+1})(\alpha_{0,1}z_3+\alpha_{0,4}z_6), \\
     (d_{n+1}^{\calM})^*(f)(s_{n+1}^{n+1}) &= f((b \otimes1 +(-1)^{n+1} 1\otimes b) s_n^n)\\
     &= bf(s_n^n) +(-1)^{n+1}f(s_n^n)b\\
     &=(1+(-1)^{n+1})(\alpha_{n,1}z_4+ \alpha_{n,3}z_6),\\
     (d_{n+1}^{\calM})^*(f)(s_{n+2}^{n+1}) &=af(s_{n+1}^{n}) + (-1)^{n+1}f(s_{0}^{n})c\\
     &= (-1)^{n+1}(\alpha_{0,1} z_5+\alpha_{0,4}z_7),
\end{aligned}\]
for \(1\le r \le n\),
\[\begin{aligned}
    (d_{n+1}^{\calM})^*(f)(s_r^{n+1}) &= f((a\otimes 1+ (-1)^{n+r+1} 1\otimes a)s_r^n)+((-1)^{n+1} 1\otimes b+(-1)^{n-r+1}b \otimes 1)s_{r-1}^n)\\
     &=af(s_r^n)+(-1)^{n+r+1}f(s_r^n)a+(-1)^{n+1}f(s_{r-1}^n)b+(-1)^{n-r+1}bf(s_{r-1}^n)\\
     &= (1+(-1)^{n-r+1})(\alpha_{r,1} z_3+ \alpha_{r,4} z_6)+ ((-1)^{n+1}+(-1)^{n-r+1})(\alpha_{r-1,1} z_4 +\alpha_{r-1,3} z_6).
\end{aligned}\]
\fi

When \(\rmChar(K)=2\), 
\[\begin{aligned}
     \partial^n(f)(s_r^{n+1}) &= 0, \text{ for } 0\le r \le n+1, \\
     \partial^n(f)(s_{n+2}^{n+1}) &=\alpha_{0,1} z_5+\alpha_{0,4}z_7.    
\end{aligned}\]
It follows that \[\rmIm(\partial^n)= \Span_K \{f_{n+2,5}^{n+1}, f_{n+2,7}^{n+1}\}.\] 
Moreover, \(f \in \Ker(\partial^n)\) if and only if \(\alpha_{0,1}= \alpha_{0,4}=0\). Consequently,
\[\Ker(\partial^n)= \Span_K \{f_{0,j}^n, j=3,6; f_{l,j}^n, 1 \le l \le n, j=1,3,4,6; f_{n+1,j}^n, j=5,7 \}.
\]
The \(n\)-th Hochschild cohomology group
\[\begin{aligned}
    \rmHH^n(A) &= \frac{\Ker(\partial^n)}{\rmIm(\partial^{n-1})}\\
    &=\frac{\Span_K \{f_{0,j}^n, j=3,6; f_{l,j}^n, 1 \le l \le n, j=1,3,4,6; f_{n+1,j}^n, j=5,7 \}}{\Span_K \{f_{n+1,5}^n, f_{n+1,7}^n\}} \\
    &\cong \Span_K \{f_{0,j}^n, j=3,6; f_{l,j}^n, 1 \le l \le n, j=1,3,4,6\}.
\end{aligned}\]

When \(\rmChar(K)\neq 2\) and \(n\) is even, 
\[\begin{aligned}
     \partial^n(f)(s_0^{n+1}) &= 0, \\
     \partial^n(f)(s_{n+1}^{n+1}) &= 0,\\
     \partial^n(f)(s_{n+2}^{n+1}) &=\alpha_{0,1} z_5+\alpha_{0,4}z_7,\\
    \partial^n(f)(s_r^{n+1}) &=2(\alpha_{r-1,1} z_4+\alpha_{r-1,3}z_6), \text{ for } 1 \le r \le n, r \text{ even},\\
     \partial^n(f)(s_r^{n+1}) &=-2(\alpha_{r,1} z_3+\alpha_{r,4}z_6), \text{ for } 1 \le r \le n, r \text{ odd}.
\end{aligned}\]
It follows that
\[\rmIm(\partial^n)= \Span_K \{f_{l,3}^{n+1}-f_{l+1,4}^{n+1}, 1\le l \le n, l \text{ odd}; f_{l,6}^{n+1}, 1 \le l \le n; f_{n+2,j}^{n+1}, j=5,7\}.\]
Moreover, \(f \in  \Ker(\partial^n)\) if and only if \(\alpha_{0,1}=\alpha_{0,4}=0\) and $\alpha_{r,1}=\alpha_{r,3}=\alpha_{r,4}=0$, $1\le r \le n$, $r \text{ odd}$. Hence,
\[\Ker(\partial^n)=\Span_K\{ f_{0,j}^n, j=3,6; f_{l,j}^n, 1\le l \le n, l \text{ even}, j=1,3,4,6; f_{l,6}^n, 1\le l \le n, l \text{ odd}; f_{n+1,j}^n, j=5,7\}.\]

When \(\rmChar(K)\neq 2\) and \(n\) is odd, 
\[\begin{aligned}
     \partial^n(f)(s_0^{n+1}) &=2(\alpha_{0,1}z_3+\alpha_{0,4}z_6), \\
     \partial^n(f)(s_{n+1}^{n+1}) &= 2(\alpha_{n,1}z_4+ \alpha_{n,3}z_6),\\
     \partial^n(f)(s_{n+2}^{n+1}) &= \alpha_{0,1} z_5+\alpha_{0,4}z_7,\\
     \partial^n(f)(s_r^{n+1}) &= 2(\alpha_{r,1}z_3+ \alpha_{r-1,1}z_4+(\alpha_{r,4}+\alpha_{r-1,3})z_6),\; 1 \le r \le n,\; r \text{ even},\\
     \partial^n(f)(s_r^{n+1}) &=0, 1 \le r \le n, r \text{ odd}.
\end{aligned}\]
It follows that
\[\rmIm(\partial^n)= \Span_K \{2f_{0,3}^{n+1}+f_{n+2,5}^{n+1}, 2f_{0,6}^{n+1}+f_{n+2,7}^{n+1}; f_{l,j}^{n+1}, 1\le l \le n, l \text{ even}, j=3,4,6; f_{n+1,j}^{n+1}, j=4,6\}.\]
Moreover, \(f \in  \Ker(\partial^n)\) if and only if \(\alpha_{0,1}=\alpha_{0,4}=\alpha_{n,1}=\alpha_{n,3}=0\) and $\alpha_{r,1}=\alpha_{r-1,1}=\alpha_{r,4}+\alpha_{r-1,3}=0$ for $1 \le r \le n$, $r \text{ even}$. Consequently,
\[\Ker(\partial^n)=\Span_K \left\{\begin{aligned}
   & f_{0,j}^n, j=3,6; f_{l-1,3}^n-f_{l,4}^n, f_{l,3}^n,  f_{l,6}^n, 1\le l \le n, l \text{ even};\\
   & f_{l,j}^n, 1\le l \le n, l \text{ odd}, j=4,6; f_{n+1,j}^n, j=5,7
\end{aligned}\right\}.\]

Therefore, when \(\rmChar(K)\neq 2\), the first Hochschild cohomology group
\[\begin{aligned}
    \rmHH^1(A) &= \frac{\Ker(\partial^1)}{\rmIm(\partial^0)}\\
    &=\frac{\Span_K \{f_{0,j}^1, j=3,6;f_{1,j}^1, j=4,6; f_{2,j}^1, j=5,7 \}}{\Span_K \{f_{2,5}^1, f_{2,7}^1\}} \\
    &\cong \Span_K \{f_{0,j}^1, j=3,6; f_{1,j}^1, j=4,6\}.
\end{aligned}\]
For even \(n\ge 2\), the \(n\)-th Hochschild cohomology group
\[\begin{aligned}
    \rmHH^n(A) &= \frac{\Ker(\partial^n)}{\rmIm(\partial^{n-1})}\\
    &=\frac{\Span_K\{ f_{0,j}^n, j=3,6; f_{l,j}^n, 1\le l \le n, l \text{ even}, j=1,3,4,6; f_{l,6}^n, 1\le l \le n, l \text{ odd}; f_{n+1,j}^n, j=5,7\}} {\Span_K \{2f_{0,3}^n+f_{n+1,5}^n, 2f_{0,6}^n+f_{n+1,7}^n; f_{l,j}^n, 1\le l \le n-1, l \text{ even}, j=3,4,6; f_{n,j}^n, j=4,6\}}\\
    &\cong \Span_K \{2f_{0,3}^n-f_{n+1,5}^n, 2f_{0,6}^n-f_{n+1,7}^n;  f_{l,1}^n, 1\le l \le n-1, l \text{ even}; f_{l,6}^n, 1\le l \le n, l \text{ odd}; f_{n,j}^n, j=1,3\}.
\end{aligned}\]
For odd \(n\ge 2\), the \(n\)-th Hochschild cohomology group
\[\begin{aligned}
    \rmHH^n(A) &= \frac{\Ker(\partial^n)}{\rmIm(\partial^{n-1})}\\
    &=\frac{\Span_K \left\{ \begin{aligned}
        &f_{0,j}^n, j=3,6; f_{l-1,3}^n-f_{l,4}^n, f_{l,3}^n,  f_{l,6}^n, 1\le l \le n, l \text{ even};\\
        &f_{l,j}^n,  1\le l \le n, l \text{ odd}, j=4,6; f_{n+1,j}^n, j=5,7
    \end{aligned}\right\}} {\Span_K \{f_{l,3}^n-f_{l+1,4}^n, 1\le l \le n-1, l \text{ odd}; f_{l,6}^n, 
1 \le l \le n-1; f_{n+1,j}^n, j=5,7\}}\\
    &\cong \Span_K \{f_{0,j}^n, j=3,6;  f_{l,3}^n, 1\le l \le n, l \text{ even}; f_{l,4}^n, 1\le l \le n-1, l \text{ odd}; f_{n,j}^n, j=4,6\},
\end{aligned}\]
and the same formula holds for $n=1$ as well.

The dimensions of \(\rmHH^n(A)\) in each degree are as follows:
\begin{enumerate}
    \item When \(\rmChar(K)= 2\),
    \[\dim_K\rmHH^n(A)=\begin{cases}
        3, & n=0,\\
        4n+2, & n \ge 1,
    \end{cases}\]
    \item when \(\rmChar(K)\neq 2\),
    \[\dim_K\rmHH^n(A)=n+3.\]
\end{enumerate}

%---------------------------------------------------------------------------------------%
\section{Comparison morphisms}\label{sec: Comparison morphisms}
 In this section, we compute the comparison morphism $\Phi_*:B_*^{\calM} \to B_*$ using  the weak self-homotopy approach introduced in \cite{IIVZ15}. We then compute the comparison morphisms $\Psi_*:B_* \to B_*^{\calM}$ by applying \cite[Theorem 3.3]{CLZ24}.

Cibils\cite{Cib90} gave a weak self-homotopy \(\{t_n\}_{n\ge -1}\) of the reduced bar resolution \(B_*\), that is, for $n\geq0$,
\[t_n:B_n\to B_{n+1}\]
is defined by
\[
t_n((a_0\otimes 1)(a_1, \cdots, a_n)) =
\begin{cases}
    (a_0, a_1, \cdots , a_n), & \text{if } a_0 \in A_+, \\
    0, & \text{if } a_0 \in E,
\end{cases}
\]
and \(t_{-1}: A \to B_0\) is defined by $t_{-1}(a_0) =(1\otimes a_0)1$. Now we explicitly compute the comparison morphism \(\Phi_*:B_*^{\calM} \to B_*\) using the weak self-homotopy \(\{t_n\}_{n\ge -1}\).

Let $\Phi_{-1}=\id_A:A\rightarrow{A}$. $\Phi_0:B_0^{\calM}\to B_0$ is given by
\[\Phi_0(e_1)=t_{-1}\circ\Phi_{-1}\circ d_0^\calM(e_1)=(1 \otimes e_1)1, \Phi_0(e_2)=t_{-1}\circ\Phi_{-1}\circ d_0^\calM(e_2)=(1 \otimes e_2)1.\]
We compute \(\Phi_1:B_1^{\calM}\to B_1\) as follows:
\begin{equation*}
    \begin{aligned}
        \Phi_1(s_0^1) &= t_0\circ \Phi_0\circ d_1^{\calM}(s_0^1) = t_0\circ\Phi_0((a\otimes1-1\otimes a)e_1) = t_0((a\otimes e_1-1\otimes a)1) = s_0^1, \\
        \Phi_1(s_1^1) &= t_0\circ\Phi_0\circ d_1^{\calM}(s_1^1) = t_0\circ\Phi_0((b\otimes1-1\otimes b)e_1) = t_0((b\otimes e_1-1\otimes b)1)=s_1^1, \\
        \Phi_1(s_2^1) &= t_0\circ\Phi_0\circ d_1^{\calM}(s_2^1) = t_0\circ\Phi_0((c\otimes1)e_2-(1\otimes c)e_1) = t_0((c\otimes e_2-1\otimes c)1)= s_2^1.
    \end{aligned}
\end{equation*}
Suppose that for $1\le m\le n-1$, 
\begin{equation*}
    \begin{aligned}
        \Phi_m(s_0^m) &=s_0^m, \\
        \Phi_m(s_{m+1}^m) &=s_{m+1}^m,
    \end{aligned}
\end{equation*}
and for $1\le r\le m$,
\begin{equation*}
    \begin{aligned}
        \Phi_m(s_r^m) &=\sum_{i_0+\cdots+i_r = m-r} (-1)^{i_1+2i_2+\cdots+ri_r}(\underbrace{a,\cdots,a}_{i_0},b,\underbrace{a,\cdots,a}_{i_1},b,\cdots,b,\underbrace{a,\cdots,a}_{i_r}).
    \end{aligned}
\end{equation*}
In particular, when $r=m$, the sum reduces to the single term $i_0=\cdots=i_m=0$, hence $\Phi_m(s_m^m) = s_m^m$.

Consider $\Phi_n:B_n^{\calM}\rightarrow{}B_n$, 
\[\begin{aligned}
    \Phi_n(s_0^n) &= t_{n-1}\circ\Phi_{n-1}\circ d_n^{\calM}(s_0^n) = t_{n-1}\circ\Phi_{n-1}((a\otimes1+(-1)^n1\otimes a) s_0^{n-1})\\
                  &=t_{n-1}((a\otimes1+(-1)^n1\otimes a)s_0^{n-1}) = s_0^n,\\
    \Phi_n(s_n^n) &= t_{n-1}\circ\Phi_{n-1}\circ d_n^{\calM}(s_n^n)  = t_{n-1}((b\otimes1+(-1)^n1\otimes b)s_{n-1}^{n-1})= s_n^n,\\
    \Phi_n(s_{n+1}^n) &= t_{n-1}\circ \Phi_{n-1}\circ d_n^{\calM}(s_{n+1}^n) = t_{n-1}((a\otimes1)s_n^{n-1} +(-1)^n(1\otimes c)s_0^{n-1})= s_{n+1}^n,
\end{aligned}\]
\iffalse
\begin{flalign*}
    & \Phi_n(s_0^n) = t_{n-1}\circ\Phi_{n-1}\circ d_n^{\calM}(s_0^n) &\\
    & \hphantom{\Phi_n(s_0^n)} = t_{n-1}\Phi_{n-1}((a\otimes1+(-1)^n1\otimes a) s_0^{n-1}) &\\
    & \hphantom{\Phi_n(s_0^n)} = t_{n-1}((a\otimes1+(-1)^n1\otimes a) (\underbrace{a,\cdots,a}_{n-1})) &\\
    & \hphantom{\Phi_n(s_0^n)} = (\underbrace{a,\cdots,a}_{n}), &
\end{flalign*}
\begin{flalign*}
    & \Phi_n(s_n^n) = t_{n-1}\Phi_{n-1}d_n^{\calM}(s_n^n) &\\
    & \hphantom{\Phi_n(s_n^n)} = t_{n-1}((b\otimes1+(-1)^n1\otimes b) (\underbrace{b,\cdots,b}_{n-1})) &\\
    & \hphantom{\Phi_n(s_n^n)} = (\underbrace{b,\cdots,b}_{n}), &
\end{flalign*}
\begin{flalign*}
    & \Phi_n(s_{n+1}^n) = t_{n-1}\Phi_{n-1}d_n^{\calM}(s_{n+1}^n) &\\
    & \hphantom{\Phi_n(s_{n+1}^n)} = t_{n-1}((a\otimes1) (\underbrace{a,\cdots,a}_{n-2},c)+(-1)^n(1\otimes c) (\underbrace{a,\cdots,a}_{n-1})) &\\
    & \hphantom{\Phi_n(s_{n+1}^n)} = (\underbrace{a,\cdots,a}_{n-1},c), &
\end{flalign*}
\fi
and for $1\le r\le n-1$,
\[\begin{aligned}
    \Phi_n(s_r^n) &= t_{n-1}\circ\Phi_{n-1}\circ d_n^{\calM}(s_r^n) \\
    &= t_{n-1}\circ\Phi_{n-1}((a \otimes 1 + (-1)^{n+r}1\otimes a) s_r^{n-1}+((-1)^n 1\otimes b+(-1)^{n-r}b \otimes 1)s_{r-1}^{n-1}),\\
\end{aligned}\]
 we need to consider two cases $r=1$ and $2\le r \le n-1$ separately. The general argument below works for \(2\le r \le n-1\), and the case \(r=1\) can be verified in a similar way.
 \[\begin{aligned}
    \Phi_n(s_r^n) &= \sum_{i_0+\cdots+i_r = n-r-1}(-1)^{i_1+2i_2+\cdots+ri_r}(a,\underbrace{a,\cdots,a}_{i_0},b,\underbrace{a,\cdots,a}_{i_1},b,\cdots,b,\underbrace{a,\cdots,a}_{i_r})\\
       &\quad + (-1)^{n-r}\sum_{i_0+\cdots+i_{r-1} = n-r}(-1)^{i_1+2i_2+\cdots+(r-1)i_{r-1}}(b,\underbrace{a,\cdots,a}_{i_0},b,\cdots,b,\underbrace{a,\cdots,a}_{i_{r-1}})\\
       & = \sum_{\substack{i_0\ge 1 \\ i_0+\cdots+i_r = n-r}}(-1)^{i_1+2i_2+\cdots+ri_r}(\underbrace{a,\cdots,a}_{i_0},b,\underbrace{a,\cdots,a}_{i_1},b,\cdots,b,\underbrace{a,\cdots,a}_{i_r})\\
       & \quad + \sum_{\substack{i_0=0 \\ i_0+\cdots+i_r = n-r}}(-1)^{i_1+2i_2+\cdots+ri_r}(b,\underbrace{a,\cdots,a}_{i_1},b,\cdots,b,\underbrace{a,\cdots,a}_{i_r}) \\
       &= \sum_{i_0+\cdots+i_r = n-r}(-1)^{i_1+2i_2+\cdots+ri_r}(\underbrace{a,\cdots,a}_{i_0},b,\underbrace{a,\cdots,a}_{i_1},b,\cdots,b,\underbrace{a,\cdots,a}_{i_r}). 
\end{aligned}\]
\iffalse
\begin{flalign*}
     & \Phi_n(s_r^n) = t_{n-1}\Phi_{n-1}d_n^{\calM}(s_r^n) &\\
    & \hphantom{\Phi_n(s_{r+1}^n)} = \sum_{i_0+\cdots+i_r = n-r-1}(-1)^{i_1+\cdots+ri_r}(a,\underbrace{a,\cdots,a}_{i_0},b,\underbrace{a,\cdots,a}_{i_1},b,\cdots,b,\underbrace{a,\cdots,a}_{i_r}) &\\
    & \hphantom{\Phi_n(s_{r+1}^n)} \quad + (-1)^{n-r}\sum_{i_0+\cdots+i_{r-1} = n-r}(-1)^{i_1+\cdots+(r-1)i_{r-1}}(b,\underbrace{a,\cdots,a}_{i_0},b,\cdots,b,\underbrace{a,\cdots,a}_{i_{r-1}}) &\\
    & \hphantom{\Phi_n(s_{r+1}^n)} = \sum_{\substack{i_0\ge 1 \\ i_0+\cdots+i_r = n-r}}(-1)^{i_1+\cdots+ri_r}(\underbrace{a,\cdots,a}_{i_0},b,\underbrace{a,\cdots,a}_{i_1},b,\cdots,b,\underbrace{a,\cdots,a}_{i_r})&\\
    & \hphantom{\Phi_n(s_{r+1}^n)} \quad + \sum_{\substack{i_0=0 \\ i_0+\cdots+i_r = n-r}}(-1)^{i_1+\cdots+ri_r}(b,\underbrace{a,\cdots,a}_{i_1},b,\cdots,b,\underbrace{a,\cdots,a}_{i_r}) &\\
    & \hphantom{\Phi_n(s_{r+1}^n)} = \sum_{i_0+\cdots+i_r = n-r}(-1)^{i_1+\cdots+ri_r}(\underbrace{a,\cdots,a}_{i_0},b,\underbrace{a,\cdots,a}_{i_1},b,\cdots,b,\underbrace{a,\cdots,a}_{i_r}). &
\end{flalign*}
\fi

In \cite[Theorem 3.3]{CLZ24}, Jun Chen, Yuming Liu, and Guodong Zhou constructed the comparison morphisms between $B_*$ and $B_*^{\calM}$. The comparison morphism $\Psi_*:B_*\to B_*^{\calM}$ is given by
\[
\begin{aligned}
& \Psi_n : B_n \to B_n^{\calM} \\
& x \in \mathcal{V}_n\mapsto \Psi_n(x) := 
\begin{cases} 
\displaystyle \sum_{w \in \mathcal{V}_n^{\calM}} \sum_{p \in \calP_1^{\calM}(x,w)} \varphi_p^{\calM}(x), & x \in \calD_n, \\
x, & x \in \mathcal{V}_n^\calM=\calW^{(n-1)}, \\
0, & x \in \calU_n. 
\end{cases}
\end{aligned}
\]

\begin{remark}
Both \(\Phi_*\) and \(\Psi_*\) lift \(\id_A\colon A \to A\); by the Comparison Theorem, they are homotopy equivalences.
Moreover, a direct computation shows a stronger relation
\[
\Psi_* \circ \Phi_*= \id_{B_*^{\calM}}.
\]
\end{remark}
 Next we record the values \(\Psi_n(x)\) for the elements \(x\in \calD_n\) that are used below. They are obtained by enumerating the corresponding
zigzag paths in $\overline Q_B^{\calM}$.

For \(n\ge 1\), \((\underbrace{a,\cdots,a}_{n-1},ba)\in \calD_n\),
\[ \Psi_n((\underbrace{a,\cdots,a}_{n-1},ba))=(b\otimes1)s_0^n+(1\otimes a)s_1^n.\]

For \(n\ge 2\) and \(1\le i\le n-1\), \((\underbrace{a,\cdots,a}_{i-1},ba,\underbrace{a,\cdots,a}_{n-i})\in \calD_n\),
\[ \Psi_n((\underbrace{a,\cdots,a}_{i-1\ge 0},ba,\underbrace{a,\cdots,a}_{n-i\ge 1}))=(b\otimes1)s_0^n.\]

For \(n\ge 1\), \((\underbrace{a,\cdots,a}_{n-1}, bc)\in \calD_n\),
\[ \Psi_n((\underbrace{a,\cdots,a}_{n-1}, bc))=(b\otimes1)s_{n+1}^n+ (1\otimes c)s_1^n.\]

For \(n\ge 2\) and \( 1\le i\le n-1\), \((\underbrace{a,\cdots,a}_{i-1},ba,\underbrace{a,\cdots,a}_{n-i-1},c)\in \calD_n\),
\[ \Psi_n((\underbrace{a,\cdots,a}_{\ge 0},ba,\underbrace{a,\cdots,a}_{\ge 0},c))=(b\otimes1)s_{n+1}^n.\]

For \(n=2\), \((ba,b)\in \calD_2\),
\[ \Psi_2((ba,b))=(b\otimes 1)s_1^2+(1\otimes a)s_2^2.\]
\iffalse
\[
\begin{tikzpicture}[>=Stealth, baseline=(current bounding box.center)]
    % 节点
    \node (A) at (0,0) {\((ba,b)\)};
    \node (B) at (-4,-3) {\((b,a,b)\)};
    \node (C) at (0,-3) {\((a,b)\)};
    \node (D) at (0,-4) {\((b,ba)\)};
    \node (E) at (-4,-7) {\((b,b,a)\)};
    \node (F) at (0,-7) {\((b,b)\)};
    % 斜箭头：从 A 的下方到 B 的上方
    \draw[->, dashed] (A.south) -- (B.north) node[midway, right, font=\small] {\((-1)^{1+1}\)};
    \draw[->] (B.east) -- (C.west) node[midway,above, font=\small] {\(b\otimes1\)};
    \draw[->] (B.south) -- (D.north) node[midway, below, font=\small] {\((-1)^2\)};
    \draw[->] (D.south) -- (E.north) node[midway, above, font=\small] {\((-1)^{2+1}\)};
    \draw[->] (E.east) -- (F.west) node[midway, above, font=\small] {\((-1)^31\otimes a\)};
\end{tikzpicture}
\]
\fi

For \(n\ge 3\) and \(1\le i\le n-1\), \((\underbrace{a,\cdots,a}_{i-1}, ba,\underbrace{a,\cdots,a}_{n-i-1},b)\in \calD_n\). If \(1\le i\le n-2\), 
\iffalse
it has one dotted outgoing arrow 
 \[ \begin{tikzcd}
   & &(\underbrace{a,\cdots,a}_{i-1},ba,\underbrace{a,\cdots,a}_{n-i-1},b) \arrow[ld,dashed,"(-1)^{i+1}"] \\
   & (\underbrace{a,\cdots,a}_{i-1},b,\underbrace{a,\cdots,a}_{n-i},b) \in \T_{n+1}
    \end{tikzcd}\]
 
with weight $(-1)^{i+1}$. From \((\underbrace{a,\cdots,a}_{i-1},b,\underbrace{a,\cdots,a}_{n-i\ge2},b)\), following the short path (\ref{path1}) repeatedly \((i-1)\) times and then the short path (\ref{path6}), we reach \((a,\cdots,a,b)\in \calW^{(n-1)}\). Hence,
\fi
$$\Psi_n((\underbrace{a,\cdots,a}_{\ge 0},ba,\underbrace{a,\cdots,a}_{\ge 1},b))=(b\otimes1)s_1^n.$$
If $i=n-1$, 
\iffalse
$(a,\cdots,a,ba,b)$ has exactly one dotted outgoing arrow to \((\underbrace{a,\cdots,a}_{i-1},b,a,b)\) with weight \((-1)^{i+1}\).

 \[ \begin{tikzcd}
   & &(\underbrace{a,\cdots,a}_{i-1},ba,b) \arrow[ld,dashed,"(-1)^{i+1}"] \\
   & (\underbrace{a,\cdots,a}_{i-1},b,a,b) \in \T_{n+1}
    \end{tikzcd}\]
\fi  
\iffalse
\((\underbrace{a,\cdots,a}_{i-1},b,a,b)\) falls under \textbf{Case 1.2} and admits two short paths. If it follows the short path (\ref{path2}) to \((\underbrace{a,\cdots,a}_{i-2},b,a,a,b)\), then by following the short path repeatedly (\ref{path1}) \((i-2)\) times and the short path (\ref{path6}), we reach \((\underbrace{a,\cdots,a}_{i},b)\). If it follows the short path (\ref{path3}) to \((\underbrace{a,\cdots,a}_{i-1},b,b,a)\), which falls under \textbf{Case 1.3} and admits two short paths. One is the short path (\ref{path4}), which reaches \((\underbrace{a,\cdots,a}_{i-2},b,a,b,a)\). However, there is no zigzag path from \((\underbrace{a,\cdots,a}_{i-2},b,a,b,a)\) to \((a,\cdots,a,b,b,a)\) or \((b,a,\cdots,a,b) \in \T_{n+1}\) since the number of \(a\)'s following the first and second \(b\) never decreases along any zigzag path.  The other is the short path (\ref{path5}), which reaches \((\underbrace{a,\cdots,a}_{i-1},b,b)\in \calW^{(n-1)}\). Hence,
\fi
$$\Psi_n((a,\cdots,a,ba,b))=(b\otimes1)s_1^n+(1\otimes a)s_2^n.$$

For \(n\ge 2\), \((a,\cdots,a,b,ba)\in \calD_n\),
$$\Psi_n((a,\cdots,a,b,ba))=(1\otimes a)s_2^n.$$

For \(n\ge 3\),  \((\underbrace{a,\cdots,a}_{x},ba,\underbrace{a,\cdots,a}_{y},b,\underbrace{a,\cdots,a}_{z})\in \calD_n\), where \(x, y\ge 0\), \(z\ge1\) and \(x+y+z=n-2\). 
It has exactly one dotted outgoing arrow to \((\underbrace{a,\cdots,a}_{x},b,\underbrace{a,\cdots,a}_{y+1},b,\underbrace{a,\cdots,a}_{z})\in \T_{n+1}\) with weight \((-1)^{x+2}\). Since the number of \(a\)'s following the first \(b\) and the number of \(a\)'s following the second \(b\) never decrease along any zigzag path, there is no zigzag path from $(\underbrace{a,\cdots,a}_{x},b,\underbrace{a,\cdots,a}_{y+1},b,\underbrace{a,\cdots,a}_{z})$ to  \((b,a,\cdots,a,b)\) or \((a,\cdots,a,b,b,a)\in \T_{n+1}\). Then

\begin{equation}\label{equ: ba,b}
    \Psi_n((\underbrace{a,\cdots,a}_{\ge 0},ba,\underbrace{a,\cdots,a}_{\ge 0},b,\underbrace{a,\cdots,a}_{\ge 1}))=0.
\end{equation}

For \(n\ge 3\), $(\underbrace{a,\cdots,a}_{x},b,\underbrace{a,\cdots,a}_{y},ba,\underbrace{a,\cdots,a}_{z})\in \calD_n \cup \calU_n$, where 
 $x,y,z\ge0$, $y+z\ge1$ and $x+y+z=n-2$.
 \iffalse
 If $y\ge 1$, $(\underbrace{a,\cdots,a}_{x},b,\underbrace{a,\cdots,a}_{y},ba,\underbrace{a,\cdots,a}_{z}) \in \calU_n$.
 If $y=0$, $(\underbrace{a,\cdots,a}_{x},b, ba,\underbrace{a,\cdots,a}_{z})\in \calD_n$, where $z\ge 1$.
 It has exactly one dotted outgoing arrow to $(\underbrace{a,\cdots,a}_{x},b, b,\underbrace{a,\cdots,a}_{z+1}) \in \T_{n+1}$ with weight $(-1)^{x+y+3}$. There is no zigzag path from $(\underbrace{a,\cdots,a}_{x},b,\underbrace{a,\cdots,a}_{y},b,\underbrace{a,\cdots,a}_{z+1})$ to $(b,a,\cdots,a,b)$ or $(a,\cdots,a,b,b,a)\in \T_{n+1}$. Then
 \fi
 We have
\begin{equation}\label{equ: b,ba}
\Psi_n((\underbrace{a,\cdots,a}_x,b,\underbrace{a,\cdots,a}_y,ba,\underbrace{a,\cdots,a}_z))=0.
\end{equation}

For \(n\ge3\),
$$\begin{aligned}
    \Psi_n((\underbrace{a,\cdots,a}_{n-3}, ba, b,b)) &=(1\otimes a)s_3^n+ (b\otimes 1)s_2^n,\\
    \Psi_n((\underbrace{a,\cdots,a}_{n-3}, b, ba, b)) &=(1\otimes a)s_3^n,\\
    \Psi_n((\underbrace{a,\cdots,a}_{n-3}, b, b,ba)) &=(1\otimes a)s_3^n.\\
\end{aligned}$$

For \(n\ge4\) and \(1\le i \le n-3\),
$$\Psi_n((\underbrace{a,\cdots,a}_{i-1}, ba, \underbrace{a,\cdots,a}_{n-i-2}, b,b))= (b\otimes 1)s_2^n.$$

For \(n\ge4\), \((a,\cdots,a, ba, \underbrace{a,\cdots,a}_{x}, b, \underbrace{a,\cdots,a}_{y}, b, \underbrace{a,\cdots,a}_{z})\in \calD_n\), where \(x,y,z\ge 0\) and \(y+z\ge1\). We have 
\begin{equation}\label{equ: ba,b,b}
    \Psi_n((a,\cdots,a, ba, \underbrace{a,\cdots,a}_{x}, b, \underbrace{a,\cdots,a}_{y}, b, \underbrace{a,\cdots,a}_{z}))=0.
\end{equation}

For \(n\ge4\), \((a,\cdots,a, b, \underbrace{a,\cdots,a}_{x}, ba, \underbrace{a,\cdots,a}_{y}, b, \underbrace{a,\cdots,a}_{z})\in \calD_n \cup \calU_n\) , where \(x,y,z\ge 0\) and \(x+y+z\ge1\). We have 
\begin{equation}\label{equ: b,ba,b}
    \Psi_n((a,\cdots,a, b, \underbrace{a,\cdots,a}_{x}, ba, \underbrace{a,\cdots,a}_{y}, b, \underbrace{a,\cdots,a}_{z}))=0.
\end{equation}

For \(n\ge4\), \((a,\cdots,a, b, \underbrace{a,\cdots,a}_{x}, b, \underbrace{a,\cdots,a}_{y}, ba, \underbrace{a,\cdots,a}_{z})\in \calD_n\cup \calU_n\) , where \(x,y,z\ge 0\) and \(x+y+z\ge1\). We have 
\begin{equation}\label{equ: b,b,ba}
    \Psi_n((a,\cdots,a, b, \underbrace{a,\cdots,a}_{x}, b, \underbrace{a,\cdots,a}_{y}, ba, \underbrace{a,\cdots,a}_{z}))=0.
\end{equation}

Applying the functor \( \Hom_{A^e}(-, A) \) to the reduced bar resolution $B_*$, we obtain a reduced Hochschild cochain complex $B^*$:
\[ 0\rightarrow \Hom_{A^e}(A\otimes_E A, A)\xrightarrow{\epsilon^0}\cdots\xrightarrow{\epsilon^{n-1}} \Hom_{A^e}(A^e\otimes_{E^e} (A_+)^{\otimes n}, A)\xrightarrow{\epsilon^n}\Hom_{A^e}(A^e\otimes_{E^e} (A_+)^{\otimes (n+1)}, A) \rightarrow \cdots, \]
where $\epsilon^n(f)=-(-1)^n f\circ d_{n+1}$ for $n\ge 0$.

\iffalse
Applying the functor \( \Hom_{A^e}(-, A) \) to the two-sided Anick resolution $B_*^{\calM}$ of \(A\), we obtain the cochain complex $(B^{\calM})^*$:
\[ 0\rightarrow \Hom_{A^e}(A\otimes_E A, A)\xrightarrow{(d_1^{\calM})^*}\cdots\rightarrow \Hom_{A^e}(A^e\otimes_{E^e} K\calW^{(n-1)}, A)\xrightarrow{(d_{n+1}^{\calM})^*}\Hom_{A^e}(A^e\otimes_{E^e} K\calW^{(n)}, A) \rightarrow \cdots \]
with $(P_{n}^{\calM})^*(f)=f\circ d_n^{\calM}$ for $n\ge 1.$
\fi

The comparison morphisms $\Phi_*$ and $\Psi_*$ induce cochain maps $\Phi^*:B^*\to (B^{\calM})^*$ and $\Psi^*:(B^{\calM})^*\to B^*,$ whose \(n\)-th components are defined by
\[
\Phi^n:\Hom_{A^e}(A^e\otimes_{E^e}(A_+)^{\otimes n},A)
\longrightarrow
\Hom_{A^e}(A^e\otimes_{E^e}K\calW^{(n-1)},A),
\qquad
f\longmapsto f\circ\Phi_n,
\]
and
\[
\Psi^n:\Hom_{A^e}(A^e\otimes_{E^e}K\calW^{(n-1)},A)
\longrightarrow
\Hom_{A^e}(A^e\otimes_{E^e}(A_+)^{\otimes n},A),
\qquad
g\longmapsto g\circ\Psi_n.
\]

\section{Cup Product}\label{sec: Cup Product}
It is well known that the cup product on $\rmHH^*(A)$ coincides with the cup product on the cohomology of \(B^*\). For \(f\in B^m\) and \(g\in B^n\), the cup product
\(f\cup_B g\in B^{m+n}\) is defined by
\[f \cup_B g((a_1, \cdots , a_{m+n}))=(-1)^{mn}f((a_1, \cdots, a_m))g((a_{m+1}, \cdots, a_{m+n})).\]

Using the induced cochain maps \(\Phi^*\) and \(\Psi^*\), we transfer the product \(\cup_B\) to \((B^{\mathcal M})^*\) and denote it by \(\lor\) as in \cite{BIKLZ26}.
For \(f\in (B^{\mathcal M})^m\) and \(g\in (B^{\mathcal M})^n\), define the product \(f \lor g \in (B^{\mathcal M})^{m+n}\) by
\[f \lor g= \Phi^{m+n}(\Psi^m(f)\cup_B \Psi^n(g))=((f\circ\Psi_m) \cup_B (g\circ\Psi_n))\circ\Phi_{m+n}.\]
The product that \(\lor\) induces on the cohomology \(H((B^{\mathcal M})^*)\cong\rmHH^*(A)\) coincides with the cup product on \(\rmHH^*(A)\). Hence, the induced product \(\lor\) on \(\rmHH^*(A)\) is graded commutative.

\begin{prop} \label{prop: cup product be zero}
      Let \(m,n\ge 1\), and let \(f\in (B^{\calM})^m\) and \(g\in (B^{\calM})^n\).
      \begin{enumerate}
          \item Assume that $2\le r \le m+n$, $i_0+\cdots+i_r=m+n-r$, $i_0,\cdots,i_r\ge0$ and at least two of $i_1,\cdots,i_r$ are nonzero. Then 
      \[ ((f\circ\Psi_m) \cup_B (g\circ\Psi_n))(\underbrace{a,\cdots,a}_{i_0},b,\underbrace{a,\cdots,a}_{i_1},b,\cdots,b,\underbrace{a,\cdots,a}_{i_r})=0.\]
          \item Assume that \(1\le x\le  r \le m+n\), \(i_0\ge 0\), \(i_x\ge 1\), \(i_0+i_x=m+n-r\) and \(i_y=0\) for all \(y\neq 0,x\). If \(m\ne i_0+x\), then
      \[ ((f\circ\Psi_m) \cup_B (g\circ\Psi_n))(\underbrace{a,\cdots,a}_{i_0},\underbrace{b,\cdots,b}_{x},\underbrace{a,\cdots,a}_{i_x},\underbrace{b,\cdots,b}_{r-x})=0.\]
      \end{enumerate}
\end{prop}
\begin{proof}
    (1) Suppose that $i_x,i_y\ge 1$ for $1\le x< y\le r$, and write
\[
(\star1):=((f\circ\Psi_m) \cup_B (g\circ\Psi_n))(\underbrace{a,\cdots,a}_{i_0},b,\underbrace{a,\cdots,a}_{i_1},b,\cdots,b,\underbrace{a,\cdots,a}_{i_x},b,\cdots b,\underbrace{a,\cdots,a}_{i_y},b,\cdots,b,\underbrace{a,\cdots,a}_{i_r}).
\]
Let $*$ be a  (possibly empty) part of the sequence $(\underbrace{a,\cdots,a}_{i_0},b,\cdots,b,\underbrace{a,\cdots,a}_{i_x},b,\cdots b,\underbrace{a,\cdots,a}_{i_y},b,\cdots,b,\underbrace{a,\cdots,a}_{i_r}).$

If \(m\le i_0+\cdots+i_{x-1}+x\),
\[
(\star1)=(-1)^{mn}f\circ\Psi_m((*)) g\circ\Psi_n((*,\underbrace{a,\cdots,a}_{i_x\ge1},\underbrace{b,\cdots,b}_{\ge1},\underbrace{a,\cdots,a}_{i_y\ge1},b,\cdots,b,\underbrace{a,\cdots,a}_{i_r}))=0
\]
since $(*,\underbrace{a,\cdots,a}_{i_x\ge1},\underbrace{b,\cdots,b}_{\ge1},\underbrace{a,\cdots,a}_{i_y\ge1},b,\cdots,b,\underbrace{a,\cdots,a}_{i_r})\in \calU_n.$

If \(m\ge i_0+\cdots+i_{x-1}+x+1\),
\[
(\star1)=(-1)^{mn}f\circ\Psi_m((\underbrace{a,\cdots,a}_{i_0},b,\cdots,b,\underbrace{a,\cdots,a}_{i_{x-1}},\underset{\substack{\uparrow\\(i_0+\cdots+i_{x-1}+x)\text{-th position}}}{b},a,*)) g\circ\Psi_n((*))=0
\]
since $(\underbrace{a,\cdots,a}_{\ge 0},b,\cdots,b,\underbrace{a,\cdots,a}_{\ge 0},b,a,*)\in \calU_m.$

(2) The proof is similar to that of (1). 
\iffalse
Write
\[(\star2):=((f\circ\Psi_m) \cup_B (g\circ\Psi_n))(\underbrace{a,\cdots,a}_{i_0},\underbrace{b,\cdots,b}_{x},\underbrace{a,\cdots,a}_{i_x},\underbrace{b,\cdots,b}_{r-x}).\]

 If $m\le i_0,$
\[ (\star2)=(-1)^{mn}f\circ\Psi_m((a,\cdots,a))g\circ\Psi_n((\underbrace{a,\cdots,a}_{i_0-m\ge 0},\underbrace{b,\cdots,b}_{x\ge 1},\underbrace{a,\cdots,a}_{i_x\ge 1},\underbrace{b,\cdots,b}_{r-x\ge0}))=0,\]
since $(\underbrace{a,\cdots,a}_{\ge 0},\underbrace{b,\cdots,b}_{\ge 1},\underbrace{a,\cdots,a}_{\ge 1},\underbrace{b,\cdots,b}_{\ge0})\in \calU_n.$

If $i_0+1\le m \le i_0+x-1,$
\[(\star2)=(-1)^{mn}f\circ\Psi_m((\underbrace{a,\cdots,a}_{i_0\ge0},\underbrace{b,\cdots,b}_{m-i_0\ge1}))g\circ\Psi_n((\underbrace{b,\cdots,b}_{x-(m-i_0)\ge 1},\underbrace{a,\cdots,a}_{i_x\ge 1},\underbrace{b,\cdots,b}_{r-x\ge0}))=0,\]
since $(\underbrace{b,\cdots,b}_{\ge 1},\underbrace{a,\cdots,a}_{\ge 1},\underbrace{b,\cdots,b}_{\ge0})\in \calU_n.$

If $i_0+x+1\le m,$
\[(\star2)=(-1)^{mn}f\circ\Psi_m((\underbrace{a,\cdots,a}_{i_0\ge0},\underbrace{b,\cdots,b}_{x\ge1},\underbrace{a,\cdots,a}_{\ge1},*))g\circ\Psi_n((*))=0,\]
since $(\underbrace{a,\cdots,a}_{\ge0},\underbrace{b,\cdots,b}_{\ge1},\underbrace{a,\cdots,a}_{\ge1},*)\in \calU_m.$
\fi
\end{proof}
\begin{remark}
     Let $m,n\ge 1$, $1\le r \le m+n$. We consider general computations of $f\lor g$ for any \(f\in (B^{\calM})^m\) and \(g\in (B^{\calM})^n\). 
    \[\begin{aligned}
        f\lor g(s_r^{m+n})&=( (f\circ \Psi_m) \cup_B (g \circ \Psi_n)) \circ \Phi_{m+n}(s_r^{m+n})\\
        &= ((f\circ \Psi_m) \cup_B (g\circ\Psi_n)) (\sum_{i_0+\cdots+i_r = m+n-r}(-1)^{i_1+\cdots+ri_r}(\underbrace{a,\cdots,a}_{i_0},b,\underbrace{a,\cdots,a}_{i_1},b,\cdots,b,\underbrace{a,\cdots,a}_{i_r})) \\
        &= \sum_{i_0+\cdots+i_r = m+n-r}(-1)^{i_1+\cdots+ri_r} ((f\circ \Psi_m) \cup_B (g\circ\Psi_n))(\underbrace{a,\cdots,a}_{i_0},b,\underbrace{a,\cdots,a}_{i_1},b,\cdots,b,\underbrace{a,\cdots,a}_{i_r}).\\
    \end{aligned}\]
 By Proposition \ref{prop: cup product be zero}, only the following two terms in the sum may contribute:
\begin{enumerate}
    \item the term with \(i_0=m+n-r\) and \(i_1=\cdots=i_r=0\);
    \item the term with \(1\le x\le r\), \(i_0\ge 0\), \(i_x\ge 1\), \(i_0+i_x=m+n-r\), and \(m= i_0+x\).
\end{enumerate} 
For the second term, the condition becomes $1\le x\le r$, $m-x\ge 0$, and $n-r+x\ge 1$ by substituting $i_0=m-x$ and $i_x=n-r+x$. 
Then
$$ \begin{aligned}
    f\lor g(s_r^{m+n})&=((f\circ \Psi_m) \cup_B (g\circ\Psi_n))(\underbrace{a,\cdots,a}_{m+n-r},\underbrace{b,\cdots,b}_{r}) \\
        & \quad + \sum_{\substack{1\le x \le r \\ m-x\ge 0 \\ n-r+x\ge 1}} (-1)^{x(n-r+x)} ((f\circ \Psi_m) \cup_B (g\circ\Psi_n))(\underbrace{a,\cdots,a}_{m-x\ge 0},\underbrace{b,\cdots,b}_{x\ge 1},\underbrace{a,\cdots,a}_{n-r+x\ge 1},\underbrace{b,\cdots,b}_{r-x\ge 0}),\\
      % &= ((f\circ \Psi_m) \cup_B (g\circ\Psi_n))(s_r^{m+n}) + \sum_{\substack{1\le x \le r \\ m-x\ge 0 \\ n-r+x\ge 1}} (-1)^{mn+x(n-r+x)} f(\underbrace{a,\cdots,a}_{m-x\ge 0},\underbrace{b,\cdots,b}_{x\ge 1}) g(\underbrace{a,\cdots,a}_{n-r+x\ge 1},\underbrace{b,\cdots,b}_{r-x\ge 0})\\
       % & = ((f\circ \Psi_m) \cup_B (g\circ\Psi_n))(s_r^{m+n}) + \sum_{\substack{1\le x \le m \\ 0\le r-x\le n-1}} (-1)^{mn+x(n-r+x)} f(s_x^m) g(s_{r-x}^n).
\end{aligned}$$
hence,
\begin{equation}\label{equ: general cup product}
    f\lor g(s_r^{m+n}) = ((f\circ \Psi_m) \cup_B (g\circ\Psi_n))(s_r^{m+n}) + \sum_{\substack{1\le x \le m \\ 0\le r-x\le n-1}} (-1)^{mn+x(n-r+x)} f(s_x^m) g(s_{r-x}^n).
\end{equation}
\end{remark}

\begin{lem}\label{lem: zero cup product in cohomology}
   Let $m,n\ge 0$.
    \begin{enumerate}
        \item \(f_{0,1}^0+f_{1,2}^0\) is the unit of \(((B^\calM)^*, \lor)\).
        %i.e., $(f_{0,1}^0+f_{1,2}^0)\lor f=f$ for any $f\in (B^\calM)^n $, $n\ge 0$.
        %\([f_{0,1}^0+f_{1,2}^0]\lor [f_{l,j}^n]=[f_{l,j}^n]\), for \(0\le l \le n\), \(j=1,3,4,6\); 
        %\([f_{0,1}^0+f_{1,2}^0]\lor [f_{n+1,j}^n]=[f_{n+1,j}^n]\), for \(n\ge 1\), \(j=5,7\); \([f_{0,1}^0+f_{1,2}^0]\lor [f_{1,2}^0]=[f_{1,2}^0]\).
        \item For \(0\le k\le m\), \(0\le l\le n\), \(i,j=3,6\), \(f_{k,i}^m \lor f_{l,j}^n=0\).
        \item For \(0\le k\le m\), \(0\le l\le n\), \(j=4,6\), \(f_{k,4}^m \lor f_{l,j}^n=0\).
        \item For \(m\ge 1\), \(0\le l \le n\), \(i=5,7\), \(j=1,3,4,6\), \(f_{m+1,i}^m \lor f_{l,j}^n=0\). 
        \item For \(m,n\ge 1\), \(i,s=5,7\), \(f_{m+1,i}^m \lor f_{n+1,s}^n=0\).
    \end{enumerate}
\end{lem}
\begin{proof}
    (1) Let $n\ge 0$ and $f\in (B^\calM)^n $. Let \(u\in \calW^{(n-1)}\).
       \[
\begin{aligned}
(f_{0,1}^0+f_{1,2}^0)\lor f(u)
&= (((f_{0,1}^0+f_{1,2}^0)\circ \Psi_0)\cup_B (f\circ\Psi_n)) (\Phi_n(u))\\
&= (f_{0,1}^0+f_{1,2}^0) (\Psi_0(1))  f(\Psi_n(\Phi_n(u)))\\
&= f(u).
\end{aligned}
\]

(2) The values of \(f_{k,i}^m\) on \(B_m^{\calM}\) and \(f_{l,j}^n\) on \(B_n^{\calM}\) lie in \(\Span_K\{a,ba\}\); moreover, \(aa=aba=baa=baba=0\). Hence, by the definition of the cup product, we have  \(f_{k,i}^m \lor f_{l,j}^n=0\).

%(3) The values of \(f_{k,4}^m\) on \(B_m^{\calM}\) lie in \(\Span_K\{b,ba,bc\}\), and their products with the values of $f_{l,j}^n$ on \(B_n^{\calM}\) are zero.

The remaining cases are analogous, and we omit their proofs.
\iffalse
    The values of \(f_{m+1,i}^m\) on \(B_m^{\calM}\) are contained in \(\Span_K\{c,bc\}\). The values of \(f_{l,j}^n\) on \(B_n^{\calM}\) are contained in \(\Span_K\{e_1, a, b,ba\}\). By the definition of the cup product and \(ce_1=ca=cb=0\), it follows that 
    \(f_{m+1,i}^m \lor f_{l,j}^n=0\). Similarly, we can prove the second formula.
\fi
 \end{proof}
 
\begin{theorem} \label{thm: char 2 cup product}
    When \(\rmChar(K)=2\), for all \(m,n\ge0\), the cup products between the \(K\)-basis elements of $\rmHH^m(A)$ and $\rmHH^n(A)$ are given as follows: 
    \begin{enumerate}
          \item  For $x\in \rmHH^n(A)$, $[f_{0,1}^0+f_{1,2}^0]\lor x=x$. Hence, $[f_{0,1}^0+f_{1,2}^0]$ is the unit of $\rmHH^*(A)$.
        \item For $m, n\ge 0$, \(i,j=3,6\), if \(f_{k,i}^m\) and \(f_{l,j}^n\) are basis elements, $[f_{k,i}^m]\lor[f_{l,j}^n]=0$.
        \item For \( m=0, n\ge 1\), \(i=3,6\), \(1\le l\le n\), \(j=1,3,4,6\), \([f_{0,i}^0]\lor [f_{l,j}^n]=
                        \begin{cases}
                            [f_{l,3}^n], & (i,j)=(3,1),\\
                            [f_{l,6}^n], & (i,j) \in \{(3,4), (6,1)\},\\
                           0, & \text{otherwise}.
                         \end{cases}\)

        \item For \(m,n\ge 1\),
            \begin{enumerate}
                \item \(i=3,6\), \(1\le l\le n\), \(j=1,3,4,6\), \([f_{0,i}^m]\lor [f_{l,j}^n]=
                    \begin{cases}
(-1)^{mn}[f_{l,3}^{m+n}], & (i,j)=(3,1),\\
(-1)^{mn}[f_{l,6}^{m+n}], & (i,j) \in \{(3,4), (6,1)\},\\
0, & \text{otherwise},
                    \end{cases}\)
                \item  \(1\le k\le m\), \(1\le l\le n\), \(i,j=1,3,4,6\), \\
                 \([f_{k,i}^m]\lor [f_{l,j}^n]=
\begin{cases}
(-1)^{mn+k(n-l)}[f_{k+l,1}^{m+n}], & (i,j)=(1,1),\\
(-1)^{mn+k(n-l)}[f_{k+l,3}^{m+n}], & (i,j) \in \{(1,3), (3,1)\},\\
(-1)^{mn+k(n-l)}[f_{k+l,4}^{m+n}], & (i,j) \in \{(1,4), (4,1)\},\\
(-1)^{mn+k(n-l)}[f_{k+l,6}^{m+n}], & (i,j) \in \{(1,6), (3,4), (4,3), (6,1)\},\\
0, & \text{otherwise}.
\end{cases}\)
            \end{enumerate}
    \end{enumerate}
\end{theorem}

\begin{proof}

    Parts (1) and (2) follow directly from Lemma \ref{lem: zero cup product in cohomology}.
We only prove (4)(b); the remaining parts (3) and (4)(a) can be proved similarly.

    (4)(b) Let $i, j = 1,3,4,6$. Let \(1\le r\le m+n\).
\[\begin{aligned}
     ((f_{k,i}^m\circ \Psi_m) \cup_B (f_{l,j}^n\circ\Psi_n))(s_r^{m+n}) &=\begin{cases}
         (-1)^{mn}f_{k,i}^m(\Psi_m(s_k^m)) f_{l,j}^n(\Psi_n(s_n^n))=(-1)^{mn}z_iz_j, & l=n, r=k+l, \\
         0, & \text{otherwise}.
     \end{cases}
\end{aligned}\]

According to equation (\ref{equ: general cup product}), when \(1\le l \le n-1\),
\[\begin{aligned}
   f_{k,i}^m\lor f_{l,j}^n(s_r^{m+n}) &= \sum_{\substack{1\le x \le m \\ 0\le r-x \le n-1}} (-1)^{mn+x(n-r+x)} f_{k,i}^m(s_x^m) f_{l,j}^n(s_{r-x}^n) 
\end{aligned}\]
The only possible non-zero term is obtained for \(x=k\) and \(r-x=l\). Hence, 
\[f_{k,i}^m\lor f_{l,j}^n(s_r^{m+n})= \begin{cases}
       (-1)^{mn+k(n-l)}  f_{k,i}^m(s_k^m) f_{l,j}^n(s_{r-k}^n)=(-1)^{mn+k(n-l)}z_iz_j, &r=k+l,\\
       0, &r \neq k+l.
   \end{cases}\]
   
When \(l=n\), \(f_{l,j}^n(s_{r-x}^n)=0\) for \(0\le r-x\le n-1\).
%\[\sum_{\substack{1\le x \le m \\ 0\le r-x \le n-1}} (-1)^{mn+x(n-r+x)} f_{k,i}^m(s_x^m) f_{l,j}^n(s_{r-x}^n)=0,\]
Hence,
\[\begin{aligned}
    f_{k,i}^m\lor f_{l,j}^n(s_r^{m+n}) &=((f_{k,i}^m\circ \Psi_m) \cup_B (f_{l,j}^n\circ\Psi_n))(s_r^{m+n})\\
    &=\begin{cases}
        (-1)^{mn}z_iz_j, & r=k+l, \\
         0, & r\neq k+l.
     \end{cases}\\
\end{aligned}\]
Therefore, we have 
\[     f_{k,i}^m\lor f_{l,j}^n(s_r^{m+n}) = \begin{cases}
\displaystyle
\begin{cases}
(-1)^{mn+k(n-l)} e_1, & (i,j) = (1,1), \\
(-1)^{mn+k(n-l)} a, & (i,j) \in \{(1,3), (3,1)\}, \\
(-1)^{mn+k(n-l)} b, & (i,j) \in \{(1,4), (4,1)\}, \\
(-1)^{mn+k(n-l)} ba, & (i,j) \in \{(1,6), (3,4), (4,3), (6,1)\}, \\
0, & \text{otherwise},
\end{cases} & r = k+l, \\
0, & r \neq k+l.
\end{cases}\]
    The values of \(f_{k,i}^m\lor f_{l,j}^n\) on \(s_0^{m+n}, s_{m+n+1}^{m+n}\) are zero.
\end{proof}
\begin{remark} 
Under the assumption that \(\rmChar(K)=2\), all the coefficients appearing in Theorem \ref{thm: char 2 cup product} are equal to \(1\). We keep these coefficients in the formulas since the same expressions will also be useful for the case \(\rmChar(K)\neq2\). 
\end{remark}

\begin{theorem}\label{thm: char neq 2 cup product}
    When \(\rmChar(K)\neq 2\), for all \(m,n \ge 0\), the cup products between the \(K\)-basis elements of $\rmHH^m(A)$ and $\rmHH^n(A)$ are  given as follows:
    \begin{enumerate}
        \item Let $x\in \rmHH^n(A)$ be a basis element.
             \begin{enumerate}
                 \item $[f_{0,1}^0+f_{1,2}^0]\lor x=x$. Hence, $[f_{0,1}^0+f_{1,2}^0]$ is the unit of $\rmHH^*(A)$.
                 \item When even $m\ge 2$ and $
                 x \neq [f_{0,1}^0+f_{1,2}^0]$, $[2f_{0,3}^m-f_{m+1,5}^m]\lor x=0$, $[2f_{0,6}^m-f_{m+1,7}^m]\lor x=0$.
             \end{enumerate}
        \item For $m, n\ge 0$, if \(f_{k,i}^m\) and \(f_{l,j}^n\) are basis elements, $[f_{k,i}^m]\lor[f_{l,j}^n]=0$ whenever \(i,j=3,6\) or \(i=4\), \(j=4,6\).
        \item For \(m=0\), \(n\) even,
            \begin{enumerate}
              \item when \(n\ge 2\), $[f_{0,3}^0]\lor [f_{n,1}^n]=[f_{n,3}^{n}]$, $[f_{0,6}^0]\lor [f_{n,1}^n]=0$;
              \item  when \(n\ge 4\), \(i=3,6\), \(1\le l\le n-1\), \(l\text{ even}\), \([f_{0,i}^0]\lor [f_{l,1}^n]=0\).
                \end{enumerate}
            
        \item For $m=0$, $n \text{ odd}$,
            \begin{enumerate}
                \item when \(n\ge 1\),  $[f_{0,3}^0]\lor [f_{n,4}^n]= [f_{n,6}^{n}]$;

                \item when \(n\ge 3\), \(1\le l\le n-1\), \(l\text{ odd}\), \([f_{0,3}^0]\lor [f_{l,4}^n]=0\).
            \end{enumerate}
  
        \item For \(m, n \) even,
            \begin{enumerate}
            \item when \(m,n\ge 2\), 
            \begin{enumerate}
   \item \(1\le k\le m\), \(k\text{ odd}\),
    $[f_{k,6}^m]\lor[f_{n,1}^n]= [f_{n+k,6}^{m+n}]$,

    \item \(i=1,3\), \(j=1,3\),
   $[f_{m,i}^m]\lor[f_{n,j}^n]=
    \begin{cases}
    [f_{m+n,1}^{m+n}], & (i,j)=(1,1),\\[2pt]
    [f_{m+n,3}^{m+n}], & (i,j)\in\{(1,3),(3,1)\},\\[2pt]
    0, & (i,j)=(3,3);
    \end{cases}$
            \end{enumerate}
            
           \item when \(m\ge 2\), \(n\ge 4\),  
           \begin{enumerate}
    \item \(1\le k\le m\), \(k\text{ odd}\), \(1\le l\le n-1\), \(l\text{ even}\),
    \([f_{k,6}^m]\lor[f_{l,1}^n]=[f_{k+l,6}^{m+n}]\),

    \item \(1\le l\le n-1\), \(l\text{ even}\), $[f_{m,1}^m]\lor[f_{l,1}^n]=[f_{m+l,1}^{m+n}]$,
    $[f_{m,3}^m]\lor[f_{l,1}^n]=0$;
           \end{enumerate}
           
           \item when \(m,n\ge 4\),   \(1\le k\le m-1\), \(1\le l\le n-1\), \(k,l\text{ even}\),
    \([f_{k,1}^m]\lor[f_{l,1}^n]=[f_{k+l,1}^{m+n}]\).
\end{enumerate}
            
        \item For \(m\) odd, \(n\) even,
\begin{enumerate}
    \item when \(m\ge1\), \(n\ge 2\),
    \begin{enumerate}
        \item $[f_{0,3}^m]\lor[f_{n,1}^n]=[f_{n,3}^{m+n}]$,
       $[f_{0,6}^m]\lor[f_{n,1}^n]= 0$,
        \item \(i=4,6\), \(j=1,3\),
        $[f_{m,i}^m]\lor[f_{n,j}^n]=
        \begin{cases}
            [f_{m+n,4}^{m+n}], & (i,j)=(4,1),\\[2pt]
            [f_{m+n,6}^{m+n}], & (i,j)\in\{(4,3),(6,1)\},\\[2pt]
            0, & (i,j)=(6,3);
        \end{cases}$
    \end{enumerate}

    \item when \(m\ge1\), \(n\ge 4\),
    \begin{enumerate}
        \item  \(1\le l\le n-1\), \(l\text{ even}\), $[f_{0,3}^m]\lor[f_{l,1}^n]=[f_{l,3}^{m+n}]$,
        $[f_{0,6}^m]\lor[f_{l,1}^n]= 0$,
        \item  \(1\le l\le n-1\), \(l\text{ even}\), $[f_{m,4}^m]\lor[f_{l,1}^n]=[f_{m+l,4}^{m+n}]$,
        $[f_{m,6}^m]\lor[f_{l,1}^n]= 0$;
    \end{enumerate}
    
    \item when \(m\ge3\), \(n\ge 2\),
    \begin{enumerate}
        \item \(1\le k\le m\), \(k\text{ even}\),
        $[f_{k,3}^m]\lor[f_{n,1}^n]= [f_{n+k,3}^{m+n}]$,
        \item \(1\le k\le m-1\), \(k\text{ odd}\), $[f_{k,4}^m]\lor[f_{n,1}^n]=[f_{n+k,4}^{m+n}]$,
        $[f_{k,4}^m]\lor[f_{n,3}^n]=0$;
    \end{enumerate}
    
    \item when \(m\ge3\), \(n\ge 4\),
    \begin{enumerate}
        \item \(1\le k\le m\), \(1\le l\le n-1\), \(k,l\text{ even}\),
        \([f_{k,3}^m]\lor[f_{l,1}^n]=[f_{k+l,3}^{m+n}]\),

        \item \(1\le k\le m-1\), \(k\text{ odd}\), \(1\le l\le n-1\), \(l\text{ even}\),
        \([f_{k,4}^m]\lor[f_{l,1}^n]=[f_{k+l,4}^{m+n}]\).
    \end{enumerate}
\end{enumerate}
              
        \item For \(m, n\) odd, 
            \begin{enumerate}
                \item when \(m,n\ge1\),
  $[f_{0,3}^m]\lor[f_{n,4}^n]= -[f_{n,6}^{m+n}]$;
                \item  when \(m\ge1\), \(n\ge3\),
    \begin{enumerate}
        \item \(1\le l\le n-1\), \(l\text{ odd}\),
        $[f_{0,3}^m]\lor[f_{l,4}^n]= -[f_{l,6}^{m+n}]$,
        \item \(1\le l\le n\), \(l\text{ even}\),
        $[f_{m,4}^m]\lor[f_{l,3}^n]= [f_{m+l,6}^{m+n}]$;
    \end{enumerate}
               \item when \(m,n\ge3\), \(1\le k\le m-1\), \(k\text{ odd}\), \(1\le l\le n\), \(l\text{ even}\),
        \([f_{k,4}^m]\lor[f_{l,3}^n]=[f_{k+l,6}^{m+n}]\).
            \end{enumerate}   
    \end{enumerate}
\end{theorem}

\begin{proof}
The formulas are derived from Lemma \ref{lem: zero cup product in cohomology} and Theorem \ref{thm: char 2 cup product}.
\end{proof}

%----------------------------------------------------------------------------------------------------------------%

\section{The Hochschild cohomology ring}\label{sec: The Hochschild cohomology ring}
In this section, we determine the ring structure of \(\rmHH^*(A)\). Recall that \(\rmHH^*(A)\) is a graded commutative algebra. When \(\rmChar(K)=2\), this graded commutativity becomes ordinary commutativity. %We first describe the ring structure of \(\rmHH^*(A)\) in this case $\rmChar(K)=2$.
\begin{theorem}\label{thm: char 2 ring}
    When \(\rmChar(K)=2\), \(\rmHH^*(A)\) is isomorphic to the commutative algebra \(K[X]/J\), where \(K[X]\) is the  graded commutative polynomial algebra generated by \( X = \{x_{n,1}, x_{n,2} \mid n \ge 0\} \cup \{x_{n,3}, x_{n,4} \mid n \ge 1\}\) with \(|x_{n,i}| = n\), and  
the ideal \(J\) is generated by the following homogeneous relations:
\begin{itemize}
     \item \(x_{m,1} x_{n,1}\), \(x_{m,1} x_{n,2}\), \(x_{m,2} x_{n,2}\), for \(m,n \ge 0\);
    \item \(x_{m,2} x_{n,4}\),  \(x_{m,2}x_{n,3}-x_{m,1}x_{n,4} \),  for \(m \ge 0\), \(n \ge 1\);
    \item  \(x_{m,4} x_{n,4}\), \(x_{m,1} x_{n,3}- x_{0,1} x_{m+n,3}\),  \(x_{m,1} x_{n,4} - x_{0,1} x_{m+n,4} \), for \(m,n \ge 1\);
    \item  \( x_{m,3} x_{n,3}- x_{1,3} x_{m+n-1,3} \), \(x_{m,3} x_{n,4} - x_{1,3} x_{m+n-1,4} \),
    for \( m\ge 2\), \( n \ge 1\).
\end{itemize}

As a consequence,    the ideal   $\calN$ generated by all homogeneous nilpotent elements has the following linear basis: 
$$\begin{aligned}
    \{x_{n,1}, x_{n,2} \mid n\ge 0\}\cup \{x_{n,4} \mid n\ge 1\}
    \cup\{x_{0,1}(x_{1,3})^l x_{n,3}, x_{0,1}(x_{1,3})^l x_{n,4}, (x_{1,3})^{l+1} x_{n,4} \mid n \ge 1,\ l \ge 0\}.
\end{aligned}$$

\end{theorem}
\begin{proof}
Define a graded map
\[
\phi:X\longrightarrow \rmHH^*(A)
\]
by
\[
\begin{aligned}
x_{n,1} &\longmapsto [f_{0,3}^n], \quad n \ge 0,\\
x_{n,2} &\longmapsto [f_{0,6}^n], \quad n \ge 0,\\
x_{n,3} &\longmapsto [f_{1,1}^n], \quad n \ge 1,\\
x_{n,4} &\longmapsto [f_{1,4}^n], \quad n \ge 1.
\end{aligned}
\]
By the universal property of the polynomial algebra \(K[X]\), the map extends uniquely to a graded \(K\)-algebra homomorphism
\[
\iota:K[X]\longrightarrow \rmHH^*(A)
\]
satisfying
\(
\iota(1_K)= 1_{\rmHH^*(A)}=[f_{0,1}^0+f_{1,2}^0]\).

\iffalse
    We define a homomorphism of graded $K$-algebras
\[
\begin{aligned}
\iota \colon K[X] &\longrightarrow \rmHH^*(A),\\
1_K &\longmapsto [f_{0,1}^0 + f_{1,2}^0] \in \rmHH^0(A),\\
x_{n,1} &\longmapsto [f_{0,3}^n] \in \rmHH^n(A), \quad n \ge 0,\\
x_{n,2} &\longmapsto [f_{0,6}^n] \in \rmHH^n(A), \quad n \ge 0,\\
x_{n,3} &\longmapsto [f_{1,1}^n] \in \rmHH^n(A), \quad n \ge 1,\\
x_{n,4} &\longmapsto [f_{1,4}^n] \in \rmHH^n(A), \quad n \ge 1,
\end{aligned}
\]
and extend multiplicatively: $\iota(x_{m,i}x_{n,j}) := \iota(x_{m,i}) \lor \iota(x_{n,j})$.
\fi

For \( n\ge 1\), \(\{[f_{0,j}^n],j=3,6; [f_{l,j}^n], 1\le l \le n, j=1,3,4,6\} \) forms a \(K\)-basis of \(\rmHH^n(A) \).
For \(n\ge 1\), \( [f_{0,3}^0] \lor [f_{1,1}^n] = [f_{1,3}^n], [f_{0,6}^0] \lor [f_{1,1}^n] = [f_{1,6}^n]\).
For \( n\ge 2\), let $l$ satisfy \( 2\le l\le n\). Since \(\rmChar(K)=2\), all the coefficients appearing in the formulas of Theorem \ref{thm: char 2 cup product} are equal to \(1\). Then
$$\begin{aligned}
    &\underbrace{[f_{1,1}^1] \lor \cdots \lor [f_{1,1}^1]}_{l-1} \lor [f_{1,1}^{n-l+1}] = [f_{l,1}^n], \\
    & [f_{0,3}^0] \lor  [f_{l,1}^n] =  [f_{l,3}^n],\quad [f_{0,6}^0] \lor  [f_{l,1}^n] =  [f_{l,6}^n],
\end{aligned}$$
$$\underbrace{[f_{1,1}^1] \lor \cdots \lor [f_{1,1}^1]}_{l-1} \lor [f_{1,4}^{n-l+1}]= [f_{l-1,1}^{l-1}] \lor [f_{1,4}^{n-l+1}] = [f_{l,4}^n].$$

Consequently, the above basis elements are contained in the \(K\)-subalgebra generated by \( \{[f_{0,1}^0 + f_{1,2}^0]; [f_{0,j}^n],j=3,6, n\ge 0; [f_{1,j}^n], j=1,4, n\ge 1\} \). Hence, these elements generate \(\rmHH^*(A)\) as a \(K\)-algebra, and the homomorphism \(\iota\) is surjective.

For \( m,n \ge 0\) and \( k,l=3,6\), we have \( [f_{0,k}^m] \lor [f_{0,l}^n]=0 \), which implies 
\[ \iota (x_{m,1}x_{n,1})=\iota (x_{m,1}x_{n,2})=\iota (x_{m,2}x_{n,2})=0.\]

For \( m \ge 0, n \ge 1\), we have \( [f_{0,6}^m] \lor [f_{1,4}^n]=0 \) and \( [f_{0,6}^m] \lor [f_{1,1}^n]=[f_{1,6}^{m+n}]=[f_{0,3}^m] \lor [f_{1,4}^n] \), which implies 
\[ \iota (x_{m,2}x_{n,4})=\iota (x_{m,2}x_{n,3}-x_{m,1}x_{n,4})=0. \]

For \( m,n \ge 1\), we have \( [f_{1,4}^m] \lor [f_{1,4}^n]=0 \), \( [f_{0,3}^{m}] \lor [f_{1,1}^{n}]= [f_{1,3}^{m+n}] = [f_{0,3}^{0}] \lor [f_{1,1}^{m+n}]\), and \( [f_{0,3}^{m}] \lor [f_{1,4}^{n}]= [f_{1,6}^{m+n}] = [f_{0,3}^0] \lor [f_{1,4}^{m+n}] \), which implies
\[\iota (x_{m,4}x_{n,4})=\iota (x_{m,1}x_{n,3}-x_{0,1}x_{m+n,3})=\iota(x_{m,1}x_{n,4}-x_{0,1}x_{m+n,4})=0.\]

For \( m \ge 2, n \ge 1\), we have \( [f_{1,1}^{m}] \lor [f_{1,1}^{n}]= [f_{2,1}^{m+n}] = [f_{1,1}^{1}] \lor [f_{1,1}^{m+n-1}] \) and \( [f_{1,1}^{m}] \lor [f_{1,4}^{n}]= [f_{2,4}^{m+n}] = [f_{1,1}^{1}] \lor [f_{1,4}^{m+n-1}]\), which implies 
\[\iota (x_{m,3}x_{n,3}-x_{1,3}x_{m+n-1,3})=\iota(x_{m,3}x_{n,4}-x_{1,3}x_{m+n-1,4})=0. \] 
Hence, \( J \subseteq \Ker(\iota)\), and \(\iota\) induces a surjective graded \(K\)-algebra homomorphism
\[ \bar{\iota} \colon K[X]/J \longrightarrow \rmHH^*(A).\]

Let \(\mathcal{C}=\{1_K\}\cup \{a_1a_2\cdots a_n\mid n\ge 1, a_i\in X\}\) be the multiplicative basis of \(K\langle X \rangle\).
Equip  \(\mathcal{C}\) with the left length-lexicographic order induced by
\(x_{0,1}<x_{1,1}<x_{2,1}<\cdots<x_{0,2}<x_{1,2}<x_{2,2}<\cdots <x_{1,3}<x_{2,3}<x_{3,3}<\cdots<x_{1,4}<x_{2,4}<x_{3,4}<\cdots\).
 Let \(\mathscr{G}\) be the set consisting of the defining relations of \(J\) together with the following commutativity relations: 
  \begin{itemize}
    \item \(x_{m,2}x_{n,1}-x_{n,1}x_{m,2}\), for \(m,n\ge 0\);
    \item \(x_{m,3}x_{n,1}-x_{n,1}x_{m,3}\), \(x_{m,4}x_{n,1}-x_{n,1}x_{m,4}\), \(x_{m,3}x_{n,2}-x_{n,2}x_{m,3}\), \(x_{m,4}x_{n,2}-x_{n,2}x_{m,4}\), for \(m\ge1\), \(n\ge 0\);
     \item  \(x_{m,4}x_{n,3}-x_{n,3}x_{m,4}\), for \(m, n\ge 1\).
 \end{itemize}

It can be easily seen that the set \(\mathscr{G}\) is tip reduced and that each element of \(\mathscr{G}\) is uniform.
Moreover, a direct computation shows that all overlap relations of \(\mathscr{G}\) reduce to \(0\).
Therefore, by Theorem \ref{thm: GS basis}, \(\mathscr{G}\) is a Gr\"obner-Shirshov system for \(\langle\mathscr{G}\rangle\) in \(K\langle X\rangle\). Hence, the set
\[\begin{aligned}
    &\{1_K\}\cup \{x_{n,1}, x_{n,2}\mid n \ge 0\}\cup \{x_{n,3}, x_{n,4}\mid n \ge 1\}\cup\\
&\quad\quad\quad\{x_{0,1}(x_{1,3})^l x_{n,3}, x_{0,1}(x_{1,3})^l x_{n,4}, (x_{1,3})^{l+1} x_{n,3},  (x_{1,3})^{l+1} x_{n,4} \mid n \ge 1, l \ge 0\}.
\end{aligned}
\]
forms a $K$-basis of \(K\langle X \rangle/\langle\mathscr{G} \rangle \cong K[X]/J \). 

\iffalse
The set \( \{1_K,x_{0,1},x_{0,2}\} \) forms a linear basis of \( ( K[X]/J)_0 \). The set \(\{x_{1,1},  x_{1,2},  x_{1,3}, x_{1,4},  x_{0,1}x_{1,3},  x_{0,1}x_{1,4}\} \) forms a linear basis of \( ( K[X]/J)_1 \).
When \(n\ge 2\),
\[ \begin{aligned}
    \{ &x_{n,1},   x_{n,2},   x_{n,3},   x_{n,4}\} \\
    &\cup \{x_{0,1}(x_{1,3})^l x_{n-l,3},  x_{0,1}(x_{1,3})^l x_{n-l,4} \mid 0 \le l \le n-1 \} \\
    &\cup \{(x_{1,3})^l x_{n-l,3},  (x_{1,3})^l x_{n-l,4} \mid 1\le l \le n-1\}
\end{aligned}
\]
forms a linear basis of \( ( K[X]/J)_n \).
\fi
Counting elements in this basis by degree gives
\[
\dim (K[X]/J)_0 = 3 = \dim \rmHH^0(A)
\]
and 
\[
\dim (K[X]/J)_n = 4n+2 = \dim \rmHH^n(A)
\]
for $n \ge 1$. Therefore, \(\bar{\iota} \) is an isomorphism of graded \(K\)-algebras.

Let $L_1$ be the set
\[\begin{aligned}
    &\{x_{n,1},  x_{n,2}\mid n \ge 0\}\cup \{  x_{n,4}\mid n \ge 1\} \cup\\
&\quad\quad\{x_{0,1}(x_{1,3})^l x_{n,3},  x_{0,1}(x_{1,3})^l x_{n,4},  (x_{1,3})^{l+1} x_{n,4} \mid n \ge 1,  l \ge 0\}.
\end{aligned}
\] 
Let $I_1$ be the ideal generated by $x_{n,1}$, $x_{n,2}$ ($n\ge0$), and $x_{n,4}$ ($n\ge1$). These generators are homogeneous and nilpotent, so
$I_1\subseteq\calN$. Moreover, by multiplying these generators of $I_1$ by the basis elements of $K[X]/J$, we obtain $I_1\subseteq \Span_K(L_1)$, and thus $I_1= \Span_K(L_1)$. 

Consider the quotient
\[ 
    R_1:=(K[X]/J)/I_1
    \cong K[x_{n,3} \mid n\ge 1]/\langle x_{m,3}x_{n,3}-x_{1,3}x_{m+n-1,3} \mid m\ge 2, n \ge 1 \rangle.
\]
This algebra $R_1$ has a $K$-basis \(L_2=\{1_K\}\cup\{(x_{1,3})^l x_{n,3}\mid n\ge1, l\ge0\}\), with multiplication
\[(x_{1,3})^kx_{m,3}\cdot (x_{1,3})^lx_{n,3} = (x_{1,3})^{k+l+1}x_{m+n-1,3}.\]  
 The displayed multiplication identifies $R_1$ with the subalgebra $K\oplus K[x,y]y$ of the polynomial ring $K[x,y]$ via the assignment
    \[
    1_K\mapsto 1_K,\qquad (x_{1,3})^l x_{n,3}\mapsto x^{n-1}y^{l+1}, n\geq 1, l\geq 0,
    \] 
where \(|x|=|y|=1\). Note that $x_{1,3}$ (resp. $x_{2,3}$) is sent to $y$ (resp. $xy$). Consequently, $R_1$
has no nonzero homogeneous nilpotent elements. It follows that every homogeneous nilpotent element of $K[X]/J$ belongs to $I_1$, and hence
$\calN=I_1=\Span_K(L_1)$.
\end{proof}

\begin{corollary}[{\cite[Section 3.1]{Xu08}}] When $\rmChar (K)=2$, there exists  an isomorphism of graded algebras:
\[
\rmHH^*(A)/\mathcal{N}
\cong K\oplus K[x,y]y,
\]
with \(|x|=|y|=1\). 
    In particular, \(\rmHH^*(A)/\mathcal{N}\) is not finitely generated.
\end{corollary}
\iffalse
\begin{proof}
The assignment
    \[
    1_K\mapsto 1_K,\qquad (x_{1,3})^l x_{n,3}\mapsto x^{n-1}y^{l+1}, n\geq 1, l\geq 0
    \]
    induces an isomorphism of graded algebras
\[
R\cong K\oplus K[x,y]y,
\]
with \(|x|=|y|=1\). Hence, $\rmHH^*(A)/\mathcal{N}
\cong K\oplus K[x,y]y$.  

Note that $x_{1,3}$ (resp. $x_{2,3}$) is sent to $y$ (resp. $xy$).
\end{proof}
\fi
\iffalse
\begin{proof}
    By Theorem \ref{thm: char 2 ring},
    \[ 
         \rmHH^*(A)/\mathcal{N} \cong K[x_{n,3} \mid n\ge 1]/\langle x_{m,3}x_{n,3}-x_{1,3}x_{m+n-1,3} \mid m\ge 2, n \ge 1 \rangle.
    \]
   This algebra is spanned by \(\{1_K\}\cup\{(x_{1,3})^l x_{n,3}\mid n\ge1, l\ge0\}\) with multiplication
\[(x_{1,3})^kx_{m,3}\cdot (x_{1,3})^lx_{n,3} = (x_{1,3})^{k+l+1}x_{m+n-1,3}.\]
    The assignment
    \[
    1_K\mapsto 1_K,\qquad (x_{1,3})^l x_{n,3}\mapsto x^{n-1}y^{l+1}
    \]
    induces an isomorphism of graded algebras
\[
\rmHH^*(A)/\mathcal{N}
\cong K\oplus K[x,y]y,
\]
with \(|x|=|y|=1\).  
\end{proof}
\fi
\begin{theorem}\label{thm: char neq 2 ring}
     When $\rmChar(K) \neq 2$, $\rmHH^*(A)$ is isomorphic to the    graded commutative algebra \(K [X]/J\), where \(  K [X]\) is the  graded commutative polynomial algebra  generated by \(X =\{x_{n,1}, x_{n,2}\mid
     n \ge 0\}\cup \{x_{n,3}\mid n \ge 1\}\) with \(|x_{n,i}| = n\) and the ideal \(J\) is generated by the following relations:
\begin{itemize}
     \item \(x_{m,1} x_{n,1}\), \(x_{m,2} x_{n,1}\), \(x_{m,2} x_{n,2}\),  for \(m,n \ge 0\);

    \item \(x_{m,3} x_{n,2} \),  for $m \ge 1$, $n \ge 0$;
    \item  \(x_{m,3} x_{0,1} \),  for \(m \ge 3\);
    \item  \(x_{m,3} x_{n,1} \),  for \(m\ge 1\), \(n \ge 2\), \(n \text{ even}\);
    \item  \(x_{m,3} x_{n,1} -x_{1,3} x_{m+n-1,1} \),  for \(m\ge 3\), \(m \text{ odd}\),  \(n \ge 1\), \(n \text{ odd}\);
    \item \(x_{m,3} x_{n,1} -x_{2,3} x_{m+n-2,1} \),  for \(m\ge 4\), \(m \text{ even}\),  \(n \ge 1\), \(n \text{ odd}\);
    \item \(x_{m,3} x_{n,3} \),  for \(m,n \text{ odd}\), \(m\le n\);

    \item  \(x_{m,3} x_{n,3} - x_{1,3} x_{m+n-1,3} \),  for \(n\ge m \ge 3\), \(m \text{ odd}\), \(n \text{ even}\);
     \item \(x_{m,3} x_{n,3} - x_{1,3} x_{m+n-1,3}\), for \(n\ge m\ge 2\), \(m \text{ even}\), \(n \text{ odd}\);
    \item \(x_{m,3} x_{n,3} - x_{2,3} x_{m+n-2,3} \),  for \(n\ge m \ge 4\), \(m \text{ even}\), \(n \text{ even}\).
   
\end{itemize}

As a consequence, the   ideal   $\calN$ generated by all homogeneous nilpotent elements has the following linear basis: 
$$\begin{aligned}
    &\{x_{n,1}, x_{n,2}\mid n\ge 0\} \cup \{x_{n,3}\mid n \text{ odd}\}\cup\{x_{1,3}(x_{2,3})^lx_{0,1}, (x_{2,3})^{l+1}x_{0,1}\mid l\ge 0\}\cup \\
&\quad\quad\quad\{x_{1,3}(x_{2,3})^l x_{n,1},  (x_{2,3})^{l+1} x_{n,1} \mid n \ge 1, n\text{ odd}, l \ge 0\} \cup \{x_{1,3}(x_{2,3})^l x_{n,3} \mid n \ge 2, n\text{ even}, l \ge 0\}.
\end{aligned}$$
 
\end{theorem}
\begin{proof}
Define a graded map
\[
\phi:X\longrightarrow \rmHH^*(A)
\]
by
\[
\begin{aligned}
\phi(x_{n,1}) &=
\begin{cases}
[f_{0,3}^n], & n = 0 \text{ or } n \text{ odd},\\
[2f_{0,3}^n - f_{n+1,5}^n], & n \ge 2, n \text{ even},
\end{cases} \\
\phi(x_{n,2}) &=
\begin{cases}
[f_{0,6}^n], & n = 0 \text{ or } n \text{ odd},\\
[2f_{0,6}^n - f_{n+1,7}^n], &  n \ge 2, n \text{ even},
\end{cases} \\
\phi(x_{n,3}) &=
\begin{cases}
[f_{1,4}^n], & n \text{ odd},\\
[f_{2,1}^n], & n \ge 2, n \text{ even}.
\end{cases}
\end{aligned}
\]
By the universal property of the graded commutative polynomial algebra \(K[X]\), the map extends uniquely to a graded \(K\)-algebra homomorphism
\[
\iota:K[X]\longrightarrow \rmHH^*(A)
\]
satisfying
\(
\iota(1_K)= 1_{\rmHH^*(A)}=[f_{0,1}^0+f_{1,2}^0]\).
One checks that \(\iota\) induces a surjective homomorphism
\[ \bar{\iota} \colon K[X]/J \longrightarrow \rmHH^*(A).\]

The set \(\mathcal{C}=\{1_K\}\cup \{a_1a_2\cdots a_n\mid n\ge 1, a_i\in X\}\) is a multiplicative basis of \(K\langle X \rangle\). Equip  \(\mathcal{C}\) with the left length-lexicographic order induced by
\(x_{1,3}<x_{2,3}<x_{3,3}<\cdots<x_{0,2}<x_{1,2}<x_{2,2}<\cdots <x_{0,1}<x_{1,1}<x_{2,1}<\cdots\). Let \(\mathscr{G}\) be the set consisting of the defining relations of \(J\) together with the following graded commutativity relations:
\begin{enumerate}
    \item \(x_{m,1} x_{n,2}-(-1)^{mn}x_{n,2}x_{m,1}\), for \(m,n \ge 0\);
    \item \(x_{m,1}x_{n,3}-(-1)^{mn}x_{n,3}x_{m,1}\), \(x_{m,2}x_{n,3}-(-1)^{mn}x_{n,3}x_{m,2}\), for \(m\ge 0\), \(n\ge 1\);
     \item \(x_{m,3}x_{n,3}-(-1)^{mn}x_{n,3}x_{m,3}\), for \(m>n\ge 1\).
\end{enumerate}

It can be easily seen that the set \(\mathscr{G}\) is tip reduced and that each element of \(\mathscr{G}\) is uniform.
Moreover, a direct computation shows that all overlap relations of \(\mathscr{G}\) reduce to \(0\). Therefore, \(\mathscr{G}\) is a Gr\"obner-Shirshov system for \(\langle\mathscr{G}\rangle\).
Hence \(K\langle X \rangle /\langle\mathscr{G}\rangle \cong K[X]/J\) has a  $K$-basis $L_1 \cup L_2$, where
$$\begin{aligned}
    L_1 &=\{x_{n,1}, x_{n,2}\mid n\ge 0\} \cup \{x_{n,3}\mid n \text{ odd}\}\cup\{x_{1,3}(x_{2,3})^lx_{0,1}, (x_{2,3})^{l+1}x_{0,1}\mid l\ge 0\}\cup\\
& \quad\quad\quad \{x_{1,3}(x_{2,3})^l x_{n,1},  (x_{2,3})^{l+1} x_{n,1} \mid n \ge 1, n\text{ odd}, l \ge 0\} \cup \{x_{1,3}(x_{2,3})^l x_{n,3} \mid n \ge 2, n\text{ even}, l \ge 0\},\\
L_2 &=\{1_K\}\cup\{(x_{2,3})^l x_{n,3}\mid n\ge2, n \text{ even}, l\ge0\}.
\end{aligned}$$
\iffalse
\[\begin{aligned}
    &\{1_K\}\cup \{x_{n,1}, x_{n,2}\mid n \ge 0\}\cup \{x_{n,3}\mid n \ge 1\}
\cup\{x_{1,3}(x_{2,3})^lx_{0,1}, (x_{2,3})^{l+1}x_{0,1}\mid l\ge 0\} \cup\\
&\quad\quad \{x_{1,3}(x_{2,3})^l x_{n,1},  (x_{2,3})^{l+1} x_{n,1} \mid n \ge 1, n\text{ odd}, l \ge 0\} \cup \{x_{1,3}(x_{2,3})^l x_{n,3},  (x_{2,3})^{l+1} x_{n,3} \mid n \ge 2, n\text{ even}, l \ge 0\}.
\end{aligned}
\] 

The set \( \{1,x_{0,1},x_{0,2}\} \) forms a linear basis of \( ( K[X]/J)_0 \). The set \(\{x_{1,1},x_{1,2},x_{1,3},x_{1,3}x_{0,1}\}\) forms a linear basis of \( (K[X]/J)_1 \). 

The set \( \{x_{2,1}, x_{2,2}, x_{2,3}, x_{2,3}x_{0,1}, x_{1,3}x_{1,1}\} \) forms a linear basis of \( ( K[X]/J)_2 \).

When \(n=2h+1 \) with \(h\ge 1\), 
\[ \begin{aligned}
    &\{x_{n,1}, x_{n,2}, x_{n,3}, x_{1,3}(x_{2,3})^hx_{0,1}\} \\
      &\cup \{(x_{2,3})^l x_{2(h-l)+1,1} \mid 1\le l \le h\}\\
     &\cup \{x_{1,3}(x_{2,3})^l x_{2(h-l),3}\mid 0\le l \le h-1\}
\end{aligned}
\]
forms a  linear basis of \( ( K[X]/J)_n \).

When \(n=2h \) with \(h\ge 2\), 
\[ \begin{aligned}
    &\{x_{n,1}, x_{n,2}, x_{n,3}, (x_{2,3})^hx_{0,1}\} \\
      &\cup \{x_{1,3}(x_{2,3})^l x_{2(h-l)-1,1}\mid 0\le l\le h-1\}\\
     &\cup \{(x_{2,3})^l x_{2(h-l),3}\mid 1\le l \le h-1\}
\end{aligned}
\]
forms a  linear basis of \( ( K[X]/J)_n \). Counting these basis elements yields \(\dim (K[X]/J)_n = n+3 = \dim \rmHH^n(A)\). Therefore, \(\bar{\iota} \) is an isomorphism of graded \(K\)-algebras.
\fi

Counting elements in this basis by degree gives
\[
\dim (K[X]/J)_n = n+3
\]
for all \(n\ge 0\). Since \(\dim \operatorname{HH}^n(A)=n+3\) and \(\bar{\iota}\) is surjective, \(\bar{\iota}\) is an isomorphism of graded \(K\)-algebras.

Let $I_2$ be the ideal generated by $x_{n,1}$, $x_{n,2}$ ($n\ge0$), and by $x_{n,3}$ for odd $n$. Then $I_2= \Span_K(L_1)$. These generators of $I_2$ are homogeneous and nilpotent, so $I_2\subseteq\calN$. 
Consider the quotient
$$\begin{aligned}
 R_2:=(K[X]/J)/I_2 \cong K[x_{n,3} \mid n\ge 2,\ n \text{ even}]/\langle x_{m,3}x_{n,3}-x_{2,3}x_{m+n-2,3} \mid m\ge4,\ n\ge2,\ m,n\text{ even}\rangle.
\end{aligned}$$
The algebra $R_2$ has the $K$-basis $L_2$, with multiplication
\[(x_{2,3})^kx_{m,3}\cdot (x_{2,3})^lx_{n,3} = (x_{2,3})^{k+l+1}x_{m+n-2,3}.\]
The assignment
 \[
    1_K\mapsto 1_K,\qquad (x_{2,3})^l x_{n,3}\mapsto x^{n-2}y^{2(l+1)}, n\geq 2\  \mathrm{even}, l\geq 0
\]
identifies $R_2$ with the subalgebra $K\oplus K[x^2,y^2]y^2$ of $K[x,y]$, where \(|x|=|y|=1\). Note that $x_{2,3}$ (resp. $x_{4,3}$) is sent to $y^2$ (resp. $x^2 y^2$). Thus the quotient $R_2$ has no nonzero homogeneous nilpotent elements. Hence
every homogeneous nilpotent element of $K[X]/J$ belongs to $I_2$, and therefore $\calN=I_2=\operatorname{Span}_K(L_1)$.
\end{proof}

\iffalse
Consequently, a linear basis of $\calN$ is
$$\begin{aligned}
    &\{x_{n,1}, x_{n,2}\mid n\ge 0\} \cup \{x_{n,3}\mid n \text{ odd}\}\cup\{x_{1,3}(x_{2,3})^lx_{0,1}, (x_{2,3})^{l+1}x_{0,1}\mid l\ge 0\}\\
&\cup \{x_{1,3}(x_{2,3})^l x_{n,1},  (x_{2,3})^{l+1} x_{n,1} \mid n \ge 1, n\text{ odd}, l \ge 0\} \cup \{x_{1,3}(x_{2,3})^l x_{n,3} \mid n \ge 2, n\text{ even}, l \ge 0\}.
\end{aligned}$$
\fi

\begin{corollary}[{\cite[Theorem 4.5]{Sna09}}]
When $\rmChar(K)\neq 2$, there exists  an isomorphism of graded algebras:
 \[ \begin{aligned}
         \rmHH^*(A)/\mathcal{N} 
         &\cong  K \oplus K[x^2,y^2]y^2,
    \end{aligned}
    \]
with \(|x|=|y|=1\).
In particular, \(\rmHH^*(A)/\mathcal{N}\) is not finitely generated.    
\end{corollary}
\iffalse
\begin{proof}
The assignment
    \[
    1_K\mapsto 1_K,\qquad (x_{2,3})^l x_{n,3}\mapsto x^{n-2}y^{2(l+1)}, n\geq 2\  \mathrm{even}, l\geq 0
    \]
    induces an isomorphism of graded algebras
\[
R\cong K\oplus K[x^2,y^2]y^2,
\]
with \(|x|=|y|=1\). Hence, $\rmHH^*(A)/\mathcal{N}
\cong K\oplus K[x^2,y^2]y^2$.  

Note that $x_{2,3}$ (resp. $x_{4,3}$) is sent to $y^2$ (resp. $x^2 y^2$).
\end{proof}
\fi

%----------------------------------------------------------------------------------------------------------------%

\section{Gerstenhaber Lie Bracket}\label{Section: Gerstenhaber Lie Bracket}

In this section, we compute the Gerstenhaber Lie bracket on the Hochschild cohomology ring.

S\'anchez-Flores \cite{SF08} introduced a reduced bracket using the cochain complex \(B^*\) of the algebra \(A\). Let \(f\in B^m \) and \(g\in B^n\). If \(m\ge 1\), \( n\ge 0\), then for a fixed index \(1\le i \le m\), $f \bar{\circ}_i g \in B^{m+n-1}$ is defined by 
\[f \bar{\circ}_i g ((a_1, \cdots , a_{m+n-1})) :=  (-1)^{(i-1)(n-1)}f((a_1, \cdots , \pi g((a_i, \cdots , a_{i+n-1})), \cdots , a_{m+n-1})),\]
where \(\pi: A\to A_+\) is the canonical projection. For \( n = 0 \), the expression \(\pi g((a_i, \cdots , a_{i+n-1}))\) is interpreted as \(\pi g(1_A)\). If $m=0$, then define $f \bar{\circ}_i g=0$. Set
\[f \bar{\circ} g = \sum_{i=1}^{m} f \bar{\circ}_i g, \quad [f, g]_B = f \bar{\circ} g - (-1)^{(m-1)(n-1)} g \bar{\circ} f. \]
The differential \(\epsilon^*\) is a  derivation with respect to  this reduced Lie bracket:
\[ \epsilon^*([f,g]_B) = [\epsilon^*(f), g]_B + (-1)^{m-1}[f, \epsilon^*(g)]_B, \]
hence, the differential $\epsilon^*$ induces a well-defined bracket on the Hochschild cohomology groups:
\[[-,-]_B: \rmHH^m(A) \times \rmHH^n(A) \to \rmHH^{m+n-1}(A).\]
S\'anchez-Flores showed that the reduced bracket on
\(\rmHH^{*+1}(A)= \bigoplus_{n=1}^{\infty} \rmHH^n(A)\) coincides with the Gerstenhaber bracket on \(\rmHH^{*+1}(A)\).

We now transfer this bracket to \((B^{\calM})^*\), via the cochain maps \(\Phi^*\) and \(\Psi^*\), and denote it by \([-, -]_{B^{\calM}}\) as in \cite{BIKLZ26}. The bilinear map
\(\tilde{\circ}_i : (B^\calM)^m \times (B^\calM)^n \longrightarrow (B^\calM)^{m+n-1}\)
is defined by 
\[f \tilde{\circ}_i g := \Phi^{m+n-1}(\Psi^m(f) \bar{\circ}_i \Psi^n(g))=((f\circ\Psi_m) \bar{\circ}_i (g\circ\Psi_n))\circ\Phi_{m+n-1},\]
where \(f\in (B^{\calM})^m \) and \(g\in (B^{\calM})^n\). Define
\[f \tilde{\circ} g = \sum_{i=1}^{m} f \tilde{\circ}_i g, \quad [f, g]_{B^{\calM}} = f \tilde{\circ} g - (-1)^{(m-1)(n-1)} g \tilde{\circ} f.\]

For any \(f\in (B^{\calM})^m \) and \(g\in (B^{\calM})^n\), we have $[f,g]_{B^{\calM}} = \Phi^{m+n-1} ([\Psi^{m}(f), \Psi^{n}(g)]_B)$ by a direct computation. Then
\[\begin{aligned}
    \partial^*([f,g]_{B^{\calM}}) &= \partial^*\circ \Phi^{m+n-1} ([\Psi^{m}(f), \Psi^{n}(g)]_B) \\
    &= \Phi^{m+n}\circ\epsilon^* ([\Psi^{m}(f), \Psi^{n}(g)]_B) \\
    &= \Phi^{m+n} ([\epsilon^*(\Psi^{m}(f)), \Psi^{n}(g)]_B) +(-1)^{m-1} \Phi^{m+n} ([\Psi^{m}(f), \epsilon^*(\Psi^{n}(g))]_B)\\
    &=\Phi^{m+n} ([\Psi^{m+1}(\partial^*(f)), \Psi^{n}(g)]_B) +(-1)^{m-1}  \Phi^{m+n} ([\Psi^{m}(f), \Psi^{n+1}(\partial^*(g))]_B) \\
    &=[\partial^*(f),g]_{B^{\calM}} + (-1)^{m-1} [f, \partial^*(g)]_{B^{\calM}}.
\end{aligned} \]
Hence, the Lie bracket \([-, -]_{B^{\calM}}\) induces a well-defined bracket on the Hochschild cohomology groups:
\[[-,-]_{B^{\calM}}: \rmHH^m(A) \times \rmHH^n(A) \to \rmHH^{m+n-1}(A),\]
and this induced bracket on \(\rmHH^{*+1}(A)\) coincides with the reduced bracket on \(\rmHH^{*+1}(A)\).

\begin{lem} \label{lem: easy circle product}
    Let \(m\ge 1\), \(n\ge 0\), \(0\le k \le m+1\), \(0\le l \le n+1\), \(q\in \{1,2,3,4,5,6,7\}\), \(1 \le i \le m\).
    \begin{enumerate}
        \item If \(j=1,2\), then \(f_{k,q}^m \tilde{\circ}_i f_{l,j}^n =0\).
        \item If \(j=3\) \textup{(}resp. \(4,5\)\textup{)} and \(a\) \textup{(}resp. \(b,c\)\textup{)} does not appear in \(s_k^m\), then \(f_{k,q}^m \tilde{\circ}_i f_{l,j}^n =0\).
    \end{enumerate}
\end{lem}

\begin{proof}
    (1) The values of \(f_{l,j}^n\) on $B_n^{\calM}$ are in \(\Span_K\{e_1,e_2\}\). Since \(\pi(0)= \pi(e_1)=\pi (e_2)=0\), the assertion follows.
    
    (2) \[\begin{aligned}
        f_{k,q}^m \tilde{\circ}_i f_{l,j}^n(s_0^{m+n-1}) &=(f_{k,q}^m\Psi_m \bar{\circ} _i f_{l,j}^n\Psi_n)\circ\Phi_{m+n-1}(s_0^{m+n-1})= (f_{k,q}^m\Psi_m \bar{\circ} _i f_{l,j}^n\Psi_n)(s_0^{m+n-1})\\
        &= (-1)^{(i-1)(n-1)}f_{k,q}^m\Psi_m(a,\cdots,a,\pi(f_{l,j}^n\Psi_n(a,\cdots,a)),a,\cdots,a)\\
       &=0,
     \end{aligned}\]
   where the last equality follows from the fact that 
   $$(a,\cdots,a,\pi(f_{l,j}^n\Psi_n(a,\cdots,a)),a,\cdots,a) \in {0}\cup \calU_m\cup \calW^{(m-1)}$$ does not equal \(s_k^m \in \calW^{(m-1)}\), since $\pi(f_{l,j}^n\Psi_n(a,\cdots,a))$ does not appear in \(s_k^m\). The same argument applies to \(s_r^{m+n-1}\), \(1\le r \le m+n\), and hence \( f_{k,q}^m \tilde{\circ}_i f_{l,j}^n=0\).
\end{proof}

\begin{prop}\label{prop: circle product}
    Let \(m \ge  1\), \(n \ge  0\) and fix an index \(i\) such that \(1 \le i \le m\).  Set
$\varepsilon_{i,n}=(-1)^{(i-1)(n-1)}$.
    \begin{enumerate}
        \item For \(n = 0\), \begin{enumerate}
            \item \(q=3,6\), $f_{0,q}^m \tilde{\circ}_i (f_{0,1}^0 + f_{1,2}^0) =0$; $q=1,4$, $f_{1,q}^m \tilde{\circ}_i (f_{0,1}^0 + f_{1,2}^0) =0$.
            \item \(q=3,6\), $f_{0,q}^m \tilde{\circ}_i f_{0,3}^0 = (-1)^{i-1} f_{0,q}^{m-1}$, $f_{0,q}^m \tilde{\circ}_i f_{0,6}^0 = \begin{cases}
                (-1)^{i-1} f_{0,6}^{m-1}, & q=3, \\
                0, &q=6.
            \end{cases} $
             \item \(q=1,4\), $f_{1,q}^m \tilde{\circ}_i f_{0,3}^0 =\begin{cases}
                (-1)^{i-1} f_{1,q}^{m-1}, & 1 \le i \le m-1,\\
                0, & i=m.
            \end{cases}$
            \item $f_{1,1}^m \tilde{\circ}_i f_{0,6}^0 =\begin{cases}
                (-1)^{i-1}f_{1,4}^{m-1}, & 1 \le i \le m-1,\\
                (-1)^{m-1} f_{0,3}^{m-1}, & i=m,
            \end{cases}$   
             $f_{1,4}^m \tilde{\circ}_i f_{0,6}^0 =\begin{cases}
               0, & 1\le i\le m-1,\\
                (-1)^{m-1} f_{0,6}^{m-1}, & i=m.
             \end{cases}$ 
        \end{enumerate}
        \item For \(m\ge 2\), \(m \text{ even}\), \(n = 0\), \begin{enumerate}
            \item $f_{2,1}^m \tilde{\circ}_i (f_{0,1}^0 + f_{1,2}^0) =0$;  \(q=5,7\), $f_{m+1,q}^m \tilde{\circ}_i (f_{0,1}^0 + f_{1,2}^0) =0$.
            \item  \(q=5,7\), $f_{m+1,q}^m \tilde{\circ}_i f_{0,3}^0 = \begin{cases}
               (-1)^{i-1} f_{m,q}^{m-1}, & 1 \le i \le m-1,\\
                0, & i=m.
                \end{cases}$ 
            \item \( q=5,7\), $f_{m+1,q}^m \tilde{\circ}_i f_{0,6}^0 = \begin{cases}
                (-1)^{i-1}f_{m,7}^{m-1}, & q=5, 1 \le i \le m-1,\\
                0, & q=7 \text{ or }  i=m.
                \end{cases}$
            \item $f_{2,1}^m \tilde{\circ}_i f_{0,3}^0 =\begin{cases}
                (-1)^{i-1}f_{2,1}^{m-1}, & 1 \le i \le m-2,\\
                0, & i=m-1, m,
            \end{cases}$
            $f_{2,1}^m \tilde{\circ}_i f_{0,6}^0 =\begin{cases}
                (-1)^{i-1}f_{2,4}^{m-1}, & 1 \le i \le m-2,\\
                (-1)^{i-1}f_{1,3}^{m-1}, & i=m-1, m.
            \end{cases}$
        \end{enumerate}
        \item For \(m,n\ge 1\), \begin{enumerate}
            \item \(q=3,6\), $f_{0,q}^m \tilde{\circ}_i f_{0,3}^n = \varepsilon_{i,n}f_{0,q}^{m+n-1}$, 
            $f_{0,q}^m \tilde{\circ}_i f_{0,6}^n = \begin{cases}
                \varepsilon_{i,n}f_{0,6}^{m+n-1}, & q=3,\\
                0, & q=6.
            \end{cases}$ 
             \item \(q=3,6\), \(j=1,4\), \(f_{0,q}^m \tilde{\circ}_i f_{1,j}^n = 0\).
             \item \( q=1,4\), \(f_{1,q}^m \tilde{\circ}_i f_{0,3}^n = \begin{cases}
                \varepsilon_{i,n}f_{1,q}^{m+n-1}, & 1\le i\le m-1,\\
                0, & i=m.
            \end{cases}\)
             \item $f_{1,1}^m \tilde{\circ}_i f_{0,6}^n =\begin{cases}
                \varepsilon_{i,n}f_{1,4}^{m+n-1}, & 1 \le i \le m-1,\\
                \varepsilon_{m,n}f_{0,3}^{m+n-1}, & i=m,
            \end{cases}$ 
            $f_{1,4}^m \tilde{\circ}_i f_{0,6}^n = \begin{cases}
                0, & 1\le i\le m-1,\\
               \varepsilon_{m,n} f_{0,6}^{m+n-1}, & i=m.
            \end{cases}$
            \item \(q=1,4\), \( f_{1,q}^m \tilde{\circ}_i f_{1,4}^n = \begin{cases}
                0, & 1\le i\le m-1,\\
               \varepsilon_{m,n} f_{1,q}^{m+n-1}, & i=m.
            \end{cases} \)
            \item $q=1,4$, \( f_{1,q}^m \tilde{\circ}_i f_{1,1}^n = 0\).
        \end{enumerate}
        \item For \(m\ge 1\), \(n\ge 2\), \(n \text{ even}\), \begin{enumerate}
            \item \(q=3,6\), \(j=5,7\), \(f_{0,q}^m \tilde{\circ}_i f_{n+1,j}^n = 0\).
            \item $f_{1,4}^m \tilde{\circ}_i f_{n+1,5}^n = 0$,  $f_{1,4}^m \tilde{\circ}_i f_{n+1,7}^n = \begin{cases}
                0, & 1\le i \le m-1, \\
                (-1)^{m-1}f_{m+n,7}^{m+n-1}, &i=m.
            \end{cases}$ 
            \item \(q=3,6\), $f_{0,q}^m \tilde{\circ}_i f_{2,1}^n = 0$; $f_{1,4}^m \tilde{\circ}_i f_{2,1}^n = 0$.
        \end{enumerate}
        \item  For \(m\ge 2\), \(m \text{ even}\), \(n\ge 1\), \begin{enumerate}
            \item \(q=5,7\),\\
            $f_{m+1,q}^m \tilde{\circ}_i  f_{0,3}^n = \begin{cases}
                \varepsilon_{i,n}f_{m+n,q}^{m+n-1}, & 1 \le i \le m-1,\\
                0, & i=m,
               \end{cases} $
              $ f_{m+1,q}^m \tilde{\circ}_i  f_{0,6}^n = \begin{cases}
                \varepsilon_{i,n}f_{m+n,7}^{m+n-1}, & q=5, 1 \le i \le m-1,\\
                0, & q=7 \text{ or } i=m.
               \end{cases} $
             \item \(q=5,7\), \( f_{m+1,q}^m \tilde{\circ}_i  f_{1,4}^n = 0\).
            \item  $f_{2,1}^m \tilde{\circ}_i  f_{0,3}^n = \begin{cases}
                \varepsilon_{i,n}f_{2,1}^{m+n-1}, & 1 \le i \le m-2,\\
                0, & i=m-1,m,
            \end{cases}$ 
            $f_{2,1}^m \tilde{\circ}_i  f_{0,6}^n = \begin{cases}
                \varepsilon_{i,n}f_{2,4}^{m+n-1}, & 1 \le i \le m-2,\\
                f_{1,3}^{m+n-1}, & i=m-1, \\
               -f_{1,3}^{m+n-1}, & i=m.
             \end{cases} $
            \item \(f_{2,1}^m \tilde{\circ}_i  f_{1,4}^n = \begin{cases}
                0, & 1 \le i \le m-2,\\
                f_{2,1}^{m+n-1}, & i=m-1,m.
            \end{cases}\)
        \end{enumerate}
     \item  For \(m, n\ge 2\), \(m,n \text{ even}\), \begin{enumerate}
         \item \( q=5,7\),  $f_{m+1,q}^m \tilde{\circ}_i f_{n+1,5}^n = \begin{cases}
                0, \quad\quad\quad 1 \le i \le m-1,\\
               -f_{m+n,q}^{m+n-1}, \quad i=m,
                \end{cases}$ 
        $f_{m+1,q}^m \tilde{\circ}_i f_{n+1,7}^n = \begin{cases}
                -f_{m+n,7}^{m+n-1}, & q=5, i=m,\\
                0, & \text{otherwise}.  
                \end{cases}$
         \item \( q=5,7\),  $f_{m+1,q}^m \tilde{\circ}_i f_{2,1}^n =0$,  $f_{2,1}^m \tilde{\circ}_i f_{n+1,q}^n = 0$. 
         \item $f_{2,1}^m \tilde{\circ}_i f_{2,1}^n=0$.
     \end{enumerate}
    \end{enumerate}
\end{prop}
\begin{proof}
Let  \(f\in (B^{\calM})^m \) and \(g\in (B^{\calM})^n\). We first describe the value of \(f\tilde{\circ}_i g\) on the elements \(s_r^{m+n-1}\), \(0\le r \le m+n\). 
\[\begin{aligned}
    f \tilde{\circ}_i g(s_0^{m+n-1}) &=(f\Psi_m \bar{\circ}_i g\Psi_n)(\Phi_{m+n-1}(s_0^{m+n-1}))=f\Psi_m \bar{\circ}_i g\Psi_n(s_0^{m+n-1})\\
&=\varepsilon_{i,n}f\Psi_m(\underbrace{a,\cdots,a}_{i-1},\pi(g\Psi_n(a,\cdots,a)),\underbrace{a,\cdots,a}_{m-i}),\\
     f \tilde{\circ}_i g(s_{m+n-1}^{m+n-1})&=(f\Psi_m \bar{\circ}_i g\Psi_n)(\Phi_{m+n-1}(s_{m+n-1}^{m+n-1}))=f\Psi_m \bar{\circ}_i g\Psi_n(s_{m+n-1}^{m+n-1}) \\ 
     &= \varepsilon_{i,n}f\Psi_m(\underbrace{b,\cdots,b}_{i-1},\pi(g\Psi_n(b, \cdots, b)),\underbrace{b,\cdots,b}_{m-i}),\\
      f \tilde{\circ}_i g(s_{m+n}^{m+n-1})&=(f\Psi_m \bar{\circ}_i g\Psi_n)(\Phi_{m+n-1}(s_{m+n}^{m+n-1})=f\Psi_m \bar{\circ}_i g\Psi_n(s_{m+n}^{m+n-1}) \\ 
     &= \begin{cases}
         \varepsilon_{i,n}f\Psi_m(\underbrace{a,\cdots,a}_{i-1},\pi(g\Psi_n(a\cdots,a)),\underbrace{a,\cdots,a}_{m-i-1},c),  & 1\le i\le m-1,\\
        \varepsilon_{m,n}f\Psi_m(\underbrace{a,\cdots,a}_{m-1},\pi(g\Psi_n(a\cdots,a,c))),&i=m,
     \end{cases}
\end{aligned}
\]
and for \(1\le r \le m+n-2\),
\[\begin{aligned}
    f \tilde{\circ}_i g(s_r^{m+n-1}) &=(f\Psi_m \bar{\circ}_i g\Psi_n)(\Phi_{m+n-1}(s_r^{m+n-1})) \\
    &=f\Psi_m \bar{\circ}_i g\Psi_n(\sum_{i_0+i_1+\cdots+i_r=m+n-r-1}(-1)^{i_1+2i_2+\cdots+ri_r}(\underbrace{a,\cdots,a}_{i_0},b,\underbrace{a,\cdots,a}_{i_1},b,\cdots,b,\underbrace{a,\cdots,a}_{i_r})).\\     
\end{aligned}\]

The formulas (1)(a), (2)(a), (3)(b)(f), (4)(c), (5)(b), (6)(c) follow directly from the Lemma \ref{lem: easy circle product}. We prove (1)(d) only, the remaining cases are obtained in the same way.
    
(1)(d) Let \(q=1,4\). If \(m=1\),
    \[\begin{aligned}
        f_{1,q}^1 \tilde{\circ}_1 f_{0,6}^0(e_1) &=f_{1,q}^1\Psi_1((1\otimes e_1) (\pi f_{0,6}^0\Psi_0(1)))\\
        &=f_{1,q}^1\Psi_1((1\otimes e_1)(ba))\\
        &=f_{1,q}^1((b\otimes e_1)s_0^1+(1\otimes a) s_1^1)\\
        &= \begin{cases}
            a, &q=1,\\
            ba, &q=4,
        \end{cases} \\
     f_{1,q}^1 \tilde{\circ}_1 f_{0,6}^0(e_2) &=f_{1,q}^1\Psi_1((1\otimes e_2)(\pi f_{0,6}^0\Psi_0(1)))=0,
    \end{aligned}\]
where the last equality follows from  $(1 \otimes e_2)(ba) = 0$.
If \(m\ge 2\), 
\[\begin{aligned}
    f_{1,q}^m \tilde{\circ}_i f_{0,6}^0(s_0^{m-1}) 
    &=(-1)^{i-1}f_{1,q}^m\Psi_m((\underbrace{a,\cdots,a}_{i-1},ba,\underbrace{a,\cdots,a}_{m-i}))\\
    &=\begin{cases}
        (-1)^{i-1}f_{1,q}^m((b\otimes1)s_0^m)=0, &1\le i\le m-1,\\
       (-1)^{m-1} f_{1,q}^m((b\otimes1)s_0^m+ (1\otimes a)s_1^m)=\begin{cases}
            (-1)^{m-1}a, &q=1,\\
            (-1)^{m-1}ba, &q=4,
        \end{cases} & i=m.
    \end{cases}
\end{aligned}\]
\[\begin{aligned}
     (\#):=f_{1,q}^m \tilde{\circ}_i f_{0,6}^0(s_1^{m-1}) 
     &=(-1)^{i-1}\sum_{\substack{i_0+i_1=m-2\\ i\le i_0+1}}(-1)^{i_1}f_{1,q}^m\Psi_m(\underbrace{a,\cdots,a}_{i-1},ba,\underbrace{a,\cdots,a}_{i_0-i+1},b,\underbrace{a,\cdots,a}_{i_1})\\
     & \quad +(-1)^{i-1}\sum_{\substack{i_0+i_1=m-2\\ i \ge i_0+2}}(-1)^{i_1}f_{1,q}^m\Psi_m(\underbrace{a,\cdots,a}_{i_0},b,\underbrace{a,\cdots,a}_{i-i_0-2},ba,\underbrace{a,\cdots,a}_{m-i})\\
\end{aligned}\]
For \(m=2\),
\[ (\#) = \begin{cases}
    f_{1,q}^m\Psi_m((ba,b)) =  f_{1,q}^m((b\otimes1)s_1^2+(1\otimes a)s_2^2) =\begin{cases}
        b, & q=1,\\
        0, & q=4,
    \end{cases} & i=1,\\
    -f_{1,q}^m\Psi_m((b,ba)) =  -f_{1,q}^m((1\otimes a)s_2^2) =0, & i=2.
\end{cases}\]
For \(m\ge 3\), by equation (\ref{equ: ba,b}), all terms in the first sum with \(i_1\ge 1\) vanish; the only possible nonzero term in the first sum is the one with \(i_1=0\), \(i_0=m-2\). By equation (\ref{equ: b,ba}), the second sum can be nonzero only when \(i=m\), \(i_0=m-2\). Hence, according to the value of \(i\), we have
\[(\#)  =\begin{cases}
         (-1)^{i-1}f_{1,q}^m\Psi_m(\underbrace{a,\cdots,a}_{i-1},ba,\underbrace{a,\cdots,a}_{m-i-1},b)=(-1)^{i-1}f_{1,q}^m((b\otimes1)s_1^m)=\begin{cases}
             (-1)^{i-1}b, &q=1,\\
             0, &q=4,
         \end{cases} \quad\quad 1\le i\le m-2,\\
         (-1)^{m-2} f_{1,q}^m\Psi_m(a,\cdots,a,ba,b)= (-1)^{m-2}f_{1,q}^m((b\otimes1)s_1^m+(1\otimes a)s_2^m)=\begin{cases}
              (-1)^{m-2}b, &q=1,\\
             0, &q=4,
          \end{cases}\quad i=m-1,\\
          (-1)^{m-1} f_{1,q}^m\Psi_m(a,\cdots,a,b,ba)=(-1)^{m-1}f_{1,q}^m((1\otimes a)s_2^m)=0, \quad\quad\quad \quad i=m.
     \end{cases} \]
The values of \(f_{1,q}^m \tilde{\circ}_i f_{0,6}^0\) on \(s_l^{m-1}\) (\(2\le l \le m\)) are zero by the definition of \(f_{1,q}^m\).
\end{proof}

Using Proposition \ref{prop: circle product}, one can obtain the following description of Lie brackets.
\begin{theorem}\label{thm:  char two Gerstenhaber bracket}
    When \( \rmChar(K) = 2\),  \(\rmHH^*(A)\)  is isomorphic, as Gerstenhaber algebras,  to the graded commutative ring $K[X]/J$ as in Theorem \ref{thm: char 2 ring},  whose  Lie brackets between the   generators  are given as follows:
    \begin{enumerate}
     \item Let \(m\ge 1\), \(n = 0\). \begin{enumerate} 
            \item For any generator $x$ in \(\rmHH^m(A)\), $[1_K, x]=0$.
            \item $q=1,2$, $[x_{m,q}, x_{0,1}] = mx_{m-1,q}$.
            \item $q=3,4$, $[x_{m,q}, x_{0,1}] = \begin{cases}
                0, & m=1,\\
             (m-1)x_{m-1,q}, & m\ge 2.
            \end{cases}$
            \item $[x_{m,1}, x_{0,2}] = mx_{m-1,2}$, $[x_{m,2}, x_{0,2}] =0$.
            \item $[x_{m,3}, x_{0,2}] = \begin{cases}
                x_{m-1,1}, & m= 1,\\
                (m-1)x_{m-1,4} + x_{m-1,1}, & m\ge 2,
            \end{cases}$ $[x_{m,4}, x_{0,2}] = x_{m-1,2}$.
        \end{enumerate}
     \item Let \(m,n\ge 1\). \begin{enumerate}
            \item $j=1,2$, $[x_{m,1}, x_{n,j}] =
      \begin{cases}
        x_{m+n-1,j}, &  m+n \text{ odd},\\
        0, & m+n \text{ even}.
      \end{cases}$
           \item $j=3,4$,  $[x_{m,1}, x_{n,j}] =
      \begin{cases}
        0, & n \text{ odd},\\
       x_{m+n-1,j}, & n \text{ even}.
      \end{cases}$
          \item $[x_{m,2}, x_{n,2}] =0$, $[x_{m,2}, x_{n,4}] =x_{m+n-1,2}$.
          \item $[x_{m,2}, x_{n,3}] =
      \begin{cases}
        x_{m+n-1,1}, &  n \text{ odd},\\
        x_{m+n-1,4} + x_{m+n-1,1}, & n \text{ even}.
      \end{cases}$
         \item $[x_{m,3}, x_{n,3}] =0$, $[x_{m,4}, x_{n,3}] = x_{m+n-1,3}$,
      $[x_{m,4},x_{n,4}] = 0$.
        \end{enumerate}
    \end{enumerate}
\end{theorem}

\begin{theorem}   
\label{thm:  char not two Gerstenhaber bracket}
    When \( \rmChar(K) \neq 2\), \(\rmHH^*(A)\)  is isomorphic, as Gerstenhaber algebras,  to the graded commutative ring $K[X]/J$ as in Theorem \ref{thm: char neq 2 ring},  whose  Lie brackets between the   generators  are given as follows:  
    \begin{enumerate}
        \item Let \(m\ge 1\). For any generator $x$ in \(\rmHH^m(A)\), $[1_K, x]=0$.
        \item Let \(m \text{ odd}\), \(n = 0\). \begin{enumerate}
            \item Put $\lambda_1=1$ and $\lambda_m=\frac14$ for odd
            $m\ge3$. Then, for $q=1,2$, $[x_{m,q},x_{0,1}] =\lambda_m x_{m-1,q}$.
            \item $[x_{m,3},x_{0,1}] =0$. 
            \item $[x_{m,1},x_{0,2}] =\lambda_m x_{m-1,2}$, $[x_{m,2},x_{0,2}] =0$, $[x_{m,3},x_{0,2}] =\lambda_m x_{m-1,2}$. 
        \end{enumerate}
    \item Let \(m\ge 2\), \(m \text{ even}\),  \(n = 0\). \begin{enumerate}
        \item $q=1,2,3$, $j=1,2$, $[x_{m,q},x_{0,j}] =0$.
        \end{enumerate}
        \item Let \(m,n \text{ odd}\). \begin{enumerate}
            \item $j=1,2$, $[x_{m,1}, x_{n,j}] = (m-n)x_{m+n-1,j}$, $[x_{m,1}, x_{n,3}] =(1-n)x_{m+n-1,3}$.
           \item $ [x_{m,2}, x_{n,2}] =0$, $[x_{m,2}, x_{n,3}] =-x_{m+n-1,2}$, $[x_{m,3}, x_{n,3}] =0$.
        \end{enumerate}
        \item Let \(m\ge 2\), \(m \text{ even}\), \(n\text{ odd}\). \begin{enumerate}
            \item $j=1,2$, $[x_{m,1},x_{n,j}] = (m-1)x_{m+n-1,j}$, $[x_{m,1},x_{n,3}] = 0$.
            \item $[x_{m,2},x_{n,1}] = (m-1)x_{m+n-1,2}$, $[x_{m,2},x_{n,2}] = 0$, $[x_{m,2},x_{n,3}] = -x_{m+n-1,2}$.
           \item $[x_{m,3},x_{n,1}] = (m-2)x_{m+n-1,3}$, $[x_{m,3},x_{n,2}] =0$, $[x_{m,3},x_{n,3}] = 2x_{m+n-1,3}$.
        \end{enumerate}
        \item Let \(m,n\ge 2\), \(m,n \text{ even}\). For $q, j=1,2,3$, $[x_{m,q},x_{n,j}] = 0$.
    \end{enumerate}
\end{theorem}

\begin{definition} [{\cite[Section 9.2]{Her16}}]
    Let \(A\) be a \(K\)-algebra with unit. Let \(\calN\) be the ideal generated by all homogeneous nilpotent elements of \(\rmHH^*(A)\). 
    \begin{itemize}
        \item[(1)]   The weak Gerstenhaber ideal \(G\) generated by $\calN$ is the minimal ideal of \(\rmHH^*(A)\) with respect to the cup product  containing $\calN$  that is also a Lie subalgebra; equivalently, $[G,G]\subseteq G$.
        \item[(2)]  The Gerstenhaber ideal \(\calG\) generated by $\calN$ is the minimal ideal of \(\rmHH^*(A)\) with respect to the cup product containing $\calN$ that is also a Lie ideal; equivalently, $[\rmHH^*(A), \calG]\subseteq \calG$.
    \end{itemize}
\end{definition}

%------------------------------------------------------------------------------

\begin{theorem} \label{Main Theorem}

Let $A$ be the Xu--Snashall algebra. 
    Then the following statements hold.    
    \begin{itemize}
    \item[(1)] The weak Gerstenhaber ideal $G$ generated by $\calN$ is equal to $\calN$.  
   Hence, \(\rmHH^*(A)/G =\rmHH^*(A)/\mathcal{N}\) is not a finitely generated \(K\)-algebra either. 
   \item[(2)]  We have 
     $\rmHH^*(A)/\calG \cong K$. 
\end{itemize}
   
\end{theorem}

\begin{proof}
(1) The bracket tables in Theorems
\ref{thm:  char two Gerstenhaber bracket} and
\ref{thm:  char not two Gerstenhaber bracket}, together with the Poisson
rule, show that \([\calN,\calN]\subseteq\calN\). Thus, $\calN$ is a weak Gerstenhaber ideal. Since $G$ is the minimal ideal containing $\calN$, we have $G=\calN$.

(2)  When $\rmChar(K)=2$, let $V_1=\Span_K(S_1)$, where 
\[\begin{aligned}
   S_1= &\{x_{n,1}, x_{n,2}\mid n \ge 0\}\cup \{x_{n,3}, x_{n,4}\mid n \ge 1\}\cup\\
&\quad\quad\quad\{x_{0,1}(x_{1,3})^l x_{n,3}, x_{0,1}(x_{1,3})^l x_{n,4}, (x_{1,3})^{l+1} x_{n,3},  (x_{1,3})^{l+1} x_{n,4} \mid n \ge 1, l \ge 0\}.
\end{aligned}
\]
Since $\{1_K\}\cup S_1$ is a $K$-basis of $K[X]/J$ and the product of any two elements of $S_1$ has zero constant term, $V_1$ is an ideal.  
Since $ x_{m,4} \in\calN\subseteq \calG$ and $[x_{m,4}, x_{1,3}] = x_{m,3}$ for $m\ge 1$, we obtain $x_{m,3} \in \calG$ for $m\ge 1$. Thus $(x_{1,3})^l x_{n,3}\in \calG$ for $n\ge 1$, $l\ge 0$; together with $\calN \subseteq \calG$, this implies $V_1\subseteq \calG$. 
Conversely,  the bracket table and the Poisson rule give $[K[X]/J,V_1]\subseteq V_1$. Thus $V_1$ is a Gerstenhaber ideal containing $\calN$, so $\calG \subseteq V_1 $. Therefore, $\calG = V_1 $ and  $\rmHH^*(A)/\calG \cong K$.

    When $\rmChar(K)\neq 2$, let $V_2=\Span_K(S_2)$, where
\[\begin{aligned}
    S_2=&\{x_{n,1}, x_{n,2}\mid n \ge 0\}\cup \{x_{n,3}\mid n \ge 1\}
\cup\{x_{1,3}(x_{2,3})^lx_{0,1}, (x_{2,3})^{l+1}x_{0,1}\mid l\ge 0\}\\
&\cup \{x_{1,3}(x_{2,3})^l x_{n,1},  (x_{2,3})^{l+1} x_{n,1} \mid n \ge 1, n\text{ odd}, l \ge 0\} \cup \{x_{1,3}(x_{2,3})^l x_{n,3},  (x_{2,3})^{l+1} x_{n,3} \mid n \ge 2, n\text{ even}, l \ge 0\}.
\end{aligned}
\]
Since $\{1_K\}\cup S_2$ is a $K$-basis of $K[X]/J$ and the product of any two elements of $S_2$ has zero constant term, $V_2$ is an ideal of $K[X]/J$.
Since $x_{1,3}\in\calN \subseteq \calG$ and $[x_{m,3}, x_{1,3}]= 2x_{m,3}$ for $m\ge 2$, $m \text{ even}$, we obtain $x_{m,3}\in \calG$. Thus $(x_{2,3})^l x_{n,3}\in \calG$ for $n \ge 2$, $n\text{ even}$, $l \ge 0$; together with $\calN \subseteq \calG$, this implies $V_2\subseteq \calG$. 
Conversely,  the bracket table and the Poisson rule give $[K[X]/J,V_2]\subseteq V_2$. Thus $V_2$ is a Gerstenhaber ideal containing $\calN$. Therefore, $\calG=V_2$ and $\rmHH^*(A)/\calG \cong K$.
\end{proof}

\bigskip

\textbf{Acknowledgements}  

The authors were supported by the National Key R$\&$D Program of China (No. 2024YFA1013803) and by the Shanghai Key Laboratory of PMMP (No. 22DZ2229014).

Computations in this paper have been done by hand and the authors take full responsibility of the correctness of the results in this paper.
 AI was used for checking the computations and for revising the paper. 

 This paper is based on the master theses \cite{Shi26}\cite{Long27} of the first two authors under the supervision of the third author. Both the first two authors want to express their sincere gratitude to the third author for his pertinent suggestions and constant encouragements.  

% We are grateful to

\textbf{Conflict of Interest}

None of the authors has any conflict of interest in the conceptualization or publication of this
work.

\textbf{Data availability}

Data sharing is not applicable to this article as no new data were created or analyzed in this study.

\bigskip

\bibliographystyle{amsplain}

%------------------------------------------------------------------------------
% End of journal.tex
%------------------------------------------------------------------------------

\end{document}